\documentclass[a4paper,11pt,reqno]{amsart}
\usepackage{amscd,graphicx,amssymb,color}
\usepackage[all,cmtip]{xy}
\usepackage[small,nohug,heads=vee]{diagrams}

\def\M{\overline{M}}

\def\De{\Delta}

\def\H{\text{H}}
\def\dra{\dashrightarrow}

\def\lra{\leftrightarrow}

\def\ra{\rightarrow}

\newcommand{\bp}{\begin{proofx}}
\newcommand{\ep}{\end{proof}}  

\def\bi{\begin{itemize}}
\def\ei{\end{itemize}}

\def\uone{\underline1}
\def\uzero{\underline0}

\def\oM{\overline{M}}

\def\bP{\Bbb P}

\def\Aut{\operatorname{Aut}}

\def\Sym{\operatorname{Sym}}

\def\de{\delta}

\def\cO{\Cal O}

\def\codim{\operatorname{codim}}

\def\bZ{\Bbb Z}

\def\bQ{\Bbb Q}
\def\bG{\Bbb G}
\def\bR{\Bbb R}
\def\bA{\Bbb A}

\def\cD{\Cal D}

\def\cE{\Cal E}

\def\tS{\tilde S}

\def\oS{\overline{S}}

\def\tS{\tilde{S}}

\def\cW{\Cal W}

\def\oN{\overline{N}}

\def\cD{\Cal D}

\def\cH{\Cal H}
\def\cM{\Cal M}

\def\Spec{\operatorname{Spec}}

\def\PGL{\operatorname{PGL}}
\def\HH{\operatorname{H}}

\def\cE{\Cal E}

\def\Cap{\operatorname{cap}}

\def\cK{\Cal K}

\def\cO{\Cal O}

\def\rk{\operatorname{rk}}
\def\O{\cO}

\def\Pic{\operatorname{Pic}}

\def\oPic{\overline{\Pic}}
\def\oPicone{\oPic^{\uone}}
\def\ocPic{\overline{\cPic}}

\def\Spec{\operatorname{Spec}}
\def\Im{\operatorname{Im}}

\def\Hom{\operatorname{Hom}}
\def\Cal{\mathcal}

\def\ov{\overline}

\def\arrow{\mathop{\longrightarrow}\limits}

\def\arrow{\mathop{\longrightarrow}\limits}

\def\oM{\overline{M}}
\def\os{\overline{S}}

\def\ocM{\overline{\Cal M}}

\def\tS{\tilde{S}}
\def\ts{\tS}
\def\ts7{\tilde{S}_7}

\def\bP{\Bbb P}

\def\cC{\Cal C}

\def\Bl{\operatorname{Bl}}

\def\Aut{\operatorname{Aut}}

\def\Sym{\operatorname{Sym}}

\def\Hom{\operatorname{Hom}}

\def\cO{\Cal O}

\def\cD{\Cal D}

\def\Pic{\operatorname{Pic}}

\def\codim{\operatorname{codim}}

\def\bZ{\Bbb Z}

\def\bQ{\Bbb Q}
\def\bG{\Bbb G}
\def\bR{\Bbb R}
\def\bA{\Bbb A}
\def\PP{\Bbb P}
\def\dra{\dashrightarrow}

\def\cD{\Cal D}

\def\cE{\Cal E}

\def\tS{\tilde S}

\def\tG{\tilde G}

\def\ep{\ddot p}

\def\oS{\overline{S}}

\def\tS{\tilde{S}}

\def\cW{\Cal W}

\def\oN{\overline{N}}

\def\cH{\Cal H}
\def\cM{\Cal M}

\def\Spec{\operatorname{Spec}}

\def\PGL{\operatorname{PGL}}
\def\SL{\operatorname{SL}}
\def\PGL{\operatorname{PGL}}

\def\cE{\Cal E}

\def\Im{\operatorname{Im
}}

\def\Exc{\operatorname{Exc}}
\def\Eff{\operatorname{Eff}}

\def\cK{\Cal K}
\def\Cl{\operatorname{Cl}}

\def\Cal{\mathcal}

\def\Pic{\operatorname{Pic}}

\def\oN{\overline{N}}
\def\os7p{\oS_7'}
\def\os7{\oS_7}
\def\on6{\oN_6}
\def\n6{\oN_6}

\def\bSi{\bold\Sigma}

\def\Mor{\operatorname{Mor}}

\def\cG{\Cal G}

\def\eset{\emptyset}
\def\cPic{\Cal P ic}

\def\oEff{\ov{\operatorname{Eff}}}

\newtheoremstyle{mystyle}{}{}{\itshape}{}{\scshape}{.}{ }{}
\theoremstyle{mystyle}

\swapnumbers
\newtheorem{Theorem}{Theorem}[section]

\newtheorem{Proposition}[Theorem]{Proposition}
\newtheorem{Lemma}[Theorem]{Lemma}

\newtheorem{Corollary}[Theorem]{Corollary}
\newtheorem{Claim}[Theorem]{Claim}
\newtheorem{Conjecture}[Theorem]{Conjecture}

\newtheoremstyle{myreview}{}{}{}{}{\scshape}{.}{ }{}
\theoremstyle{myreview}

\newtheorem{Definition}[Theorem]{Definition}
\newtheorem{Example}[Theorem]{Example}

\newtheorem{Remark}[Theorem]{Remark}
\newtheorem{Remarks}[Theorem]{Remarks}
\newtheorem{Notation}[Theorem]{Notation}
\newtheorem{Review}[Theorem]{}

\def\lra{\leftrightarrow}

\newcounter{et}[Theorem]
\def\cooltag{\tag{\arabic{section}.\arabic{Theorem}.\arabic{et}}\addtocounter{et}{1}}

\begin{document}

\title[Hypertrees, projections, and moduli of stable rational curves]{Hypertrees, projections, and \\ moduli of stable rational curves}
\author{Ana-Maria Castravet and Jenia Tevelev}

\address{\vskip -.5cm Ana-Maria Castravet: \sf Department of Mathematics, 
University of Arizona, 617 N Santa Rita Ave, Tucson, AZ 85721-0089, USA} 
\email{noni@math.arizona.edu}

\address{\vskip -.5cm Jenia Tevelev: \sf Department of Mathematics and Statistics, 
Lederle Graduate Research Tower,  University of Massachusetts at Amherst, MA 01003-9305, USA} 
\email{tevelev@math.umass.edu}

\begin{abstract}
We give a conjectural description for the cone of effective divisors of 
the Grothendieck--Knudsen moduli space $\oM_{0,n}$ of stable rational curves with $n$~marked points.
Namely, we introduce new combinatorial structures called hypertrees
and show that they give exceptional divisors on  $\oM_{0,n}$ with many remarkable properties.
\end{abstract}

\def\thefootnote{}

\maketitle

\section{Introduction}

A major open problem
inspired by the pioneering work of Harris and Mumford \cite{HM} 
on the Kodaira dimension of 
the moduli space of stable curves,
is to understand geometry of its birational models, 
and in particular to describe its cone of effective divisors and a decomposition  
of this cone into Mori chambers \cite{HK} encoding ample divisors on birational models.

Here we study the genus zero case. The moduli spaces $\oM_{0,n}$ parametrize  
stable rational curves, i.e.,~nodal trees of $\bP^1$'s with $n$ marked points
and without automorphisms. For any subset $I$ of marked points, $\oM_{0,n}$
has a natural boundary divisor $\delta_I$ whose general element
parametrizes stable rational curves with two irreducible components,
one marked by points in $I$ and another marked by points in $I^c$.
We will introduce new combinatorial objects called {\em hypertrees} 
with an eye towards the following

\begin{Conjecture} The effective cone of 
$\oM_{0,n}$ is generated by boundary divisors 
and by divisors $D_\Gamma$ (defined below)
parametrized by irreducible
hypertrees.
\end{Conjecture}

\begin{Definition}\label{hypertree def}
A {\em hypertree} 
$\Gamma=\{\Gamma_1,\ldots,\Gamma_d\}$ on
a set~$N$
is a collection of subsets of $N$
such that the following conditions are satisfied:
\begin{itemize}
\item Any subset $\Gamma_j$ has at least three elements;
\item Any $i\in N$ 
is contained in at least
two subsets~$\Gamma_j$;
\item ({\em convexity axiom})
\begin{equation}\label{CondS}
\bigl|\bigcup_{j\in S}\Gamma_j\bigr|-2\ge\sum_{j\in S}(|\Gamma_j|-2)
\quad\hbox{\rm for any $S\subset\{1,\ldots,d\}$};
\tag{\ddag}\end{equation}
\item  ({\em normalization})
\begin{equation}\label{normalizat}
|N|-2=\sum\limits_{j\in\{1,\dots,d\}}(|\Gamma_j|-2).
\tag{\dag}\end{equation}
\end{itemize}
A~hypertree $\Gamma$ is {\em irreducible} if \eqref{CondS}
is a strict inequality for $1<|S|<d$.
\end{Definition}

\begin{figure}[htbp]
\includegraphics[width=4in]{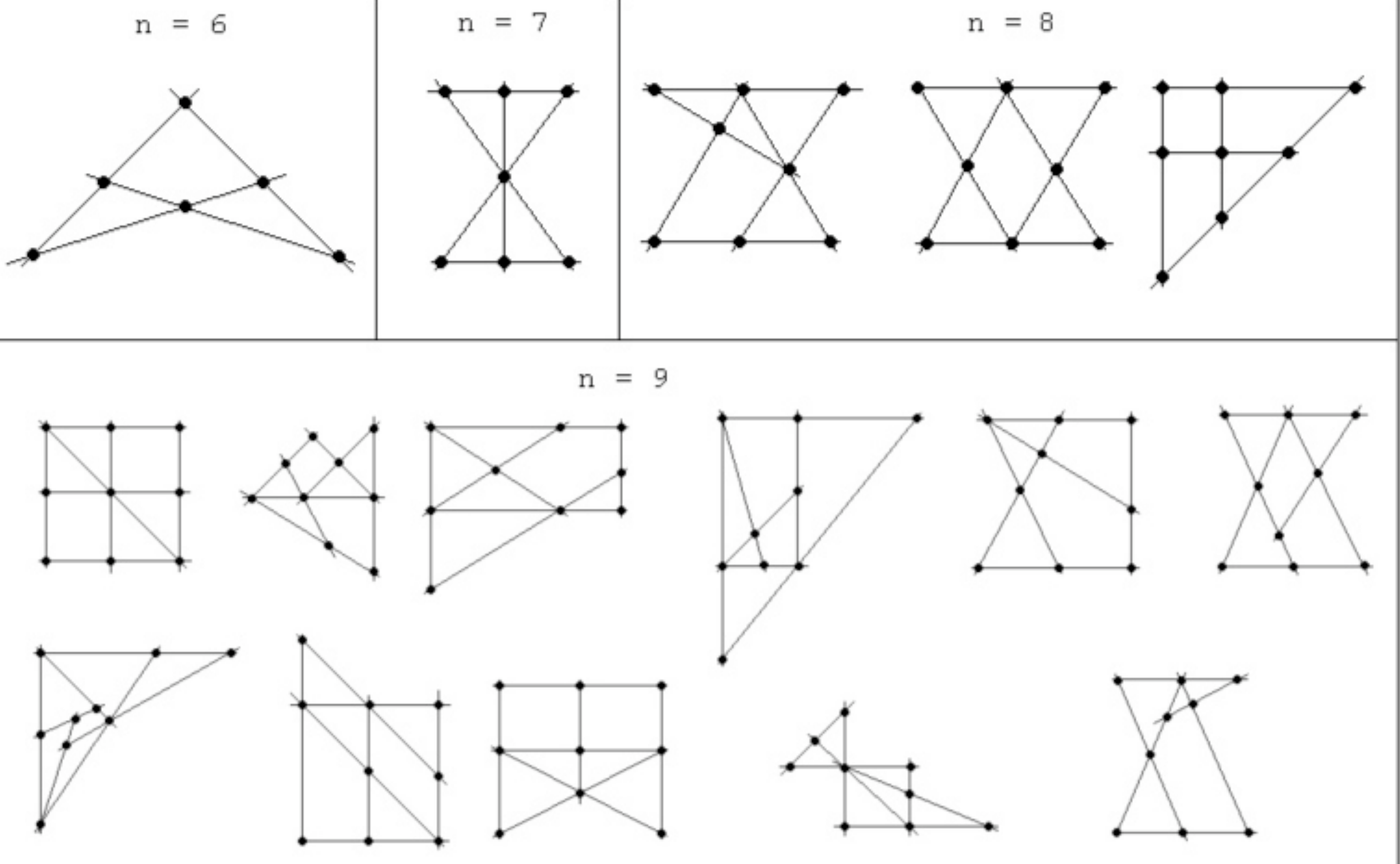}
\caption{\small Irreducible hypertrees for $n<10$. Points correspond to elements of $N$
and lines correspond to $\Gamma_1,\ldots,\Gamma_d$.}\label{IlyaPics}
\end{figure}

\begin{Remark}
The most common hypertrees are composed of triples.
In~this case \eqref{normalizat} becomes $d=n-2$ and \eqref{CondS} becomes
$$
\bigl|\bigcup_{j\in S}\Gamma_j\bigr|\ge |S|+2
\quad\hbox{\rm for any $S\subset\{1,\ldots,n-2\}$},
$$
i.e., ~$\Gamma$ is sufficiently capacious.
If  we consider pairs instead of triples, and change $2$ to $1$ in \eqref{normalizat} and \eqref{CondS}, then 
it is easy to see that $\Gamma$ will be a connected tree on vertices $\{1,\ldots,n\}$. This explains our term ``hypertree''.
\end{Remark}

\begin{Definition}\label{def D_Gamma}
For any irreducible hypertree $\Gamma$ on the set $\{1,\ldots,n\}$, 
let 
$D_\Gamma\subset\oM_{0,n}$
be the closure of the locus in $M_{0,n}$ obtained
by
\begin{figure}[htbp]
\includegraphics[width=2.5in]{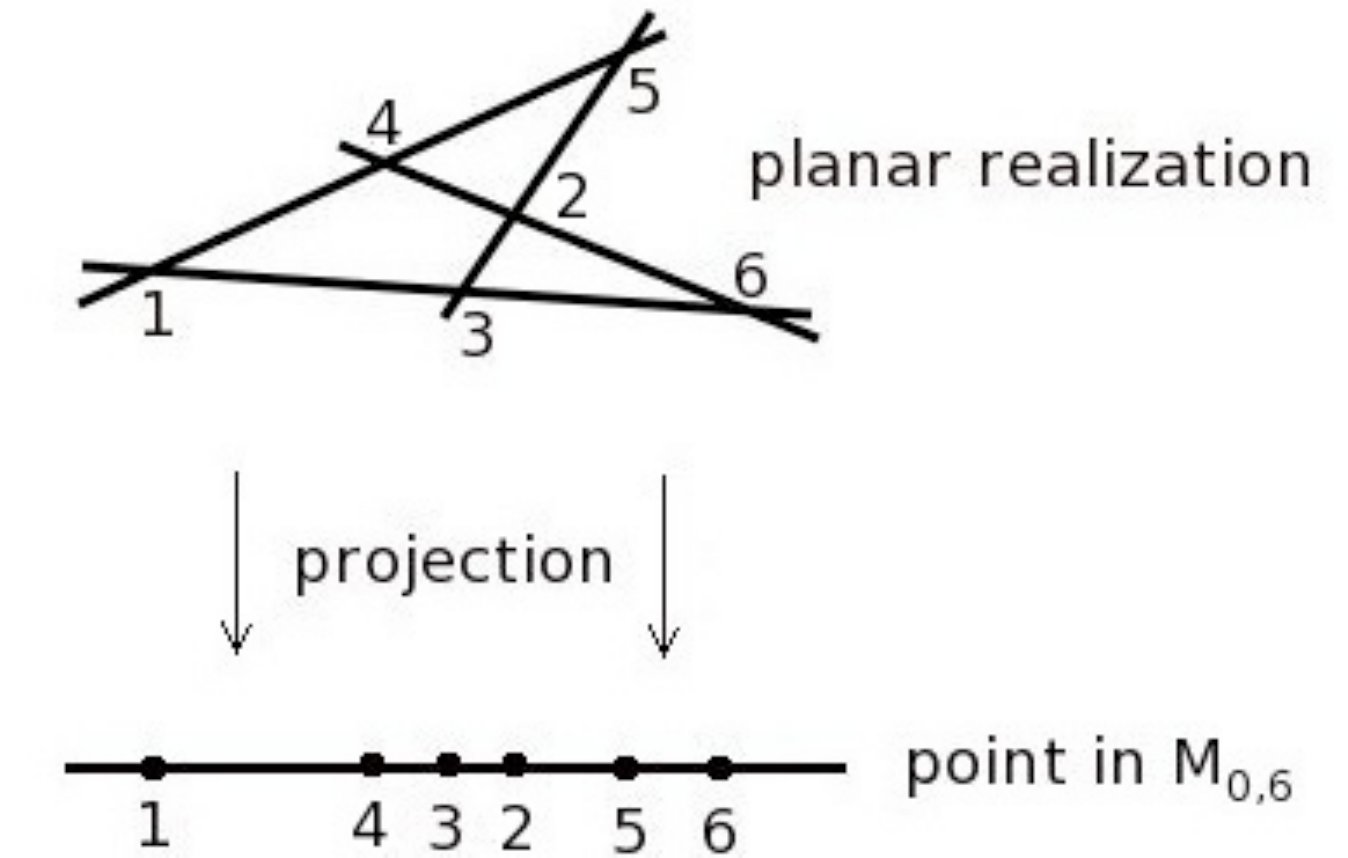}
\caption{Hypertree divisor 
as the locus of projections}
\label{KVpicture}
\end{figure}
\begin{itemize}
\item choosing a 
{\em planar realization of~$\Gamma$}:
a configuration of different points $p_1,\ldots,p_n\in\bP^2$ such that, for any subset $S\subset\{1,\ldots,n\}$
with at least three points,
$\{p_i\}_{i\in S}$ are collinear if and only if $S\subset\Gamma_j$ for some~$j$.
\item projecting $p_1,\ldots,p_n$ from a point 
$p\in\bP^2$ to points $q_1,\ldots,q_n\in\bP^1$;
\item representing 
the datum $(\bP^1;q_1,\ldots,q_n)$ by a point of $M_{0,n}$.
\end{itemize}
If $\Gamma$ is an irreducible hypertree on a subset $K\subset\{1,\ldots,n\}$,
we abuse notation and let
$D_\Gamma\subset\oM_{0,n}$ 
be the pull-back of $D_\Gamma\subset\oM_{0,K}$ with respect to 
the  forgetful map $\oM_{0,n}\to\oM_{0,K}$.
\end{Definition}

Here is our first result:

\begin{Theorem}\label{MainThOne}
For any irreducible hypertree $\Gamma$, the locus
$D_\Gamma\subset\oM_{0,n}$ 
is a non-empty irreducible divisor,
which generates an extremal ray of the effective cone of $\oM_{0,n}$. 
Moreover, this divisor is exceptional: there exists a birational contraction
$$\oM_{0,n}\dra X_\Gamma$$
onto a normal projective variety $X_\Gamma$ (see Theorem \ref{jacobianmain}),
and $D_\Gamma$ is the irreducible component of its exceptional locus that intersects $M_{0,n}$.
\end{Theorem}

Notice that apriori it is not at all clear that an irreducible hypertree has a planar realization,
but we will show that this is always the case. Moreover, any irreducible hypertree on any subset
$K\subset N$ gives rise, by pull-back via the forgetful map $\pi_K$ to an effective divisor 
which generates an extremal ray of the effective cone of $\oM_{0,n}$ (see Lemma \ref{pull-back}).

\begin{Review}\textsc{Spherical Hypertrees.}\label{10vertexexample} 
We discovered that any even (i.e.,~bicolored) triangulation of a $2$-sphere gives a hypertree.
Any such triangulation has
a collection of ``black'' faces and a collection
of ``white'' faces. 
\begin{figure}[htbp]
\includegraphics[width=4.5in]{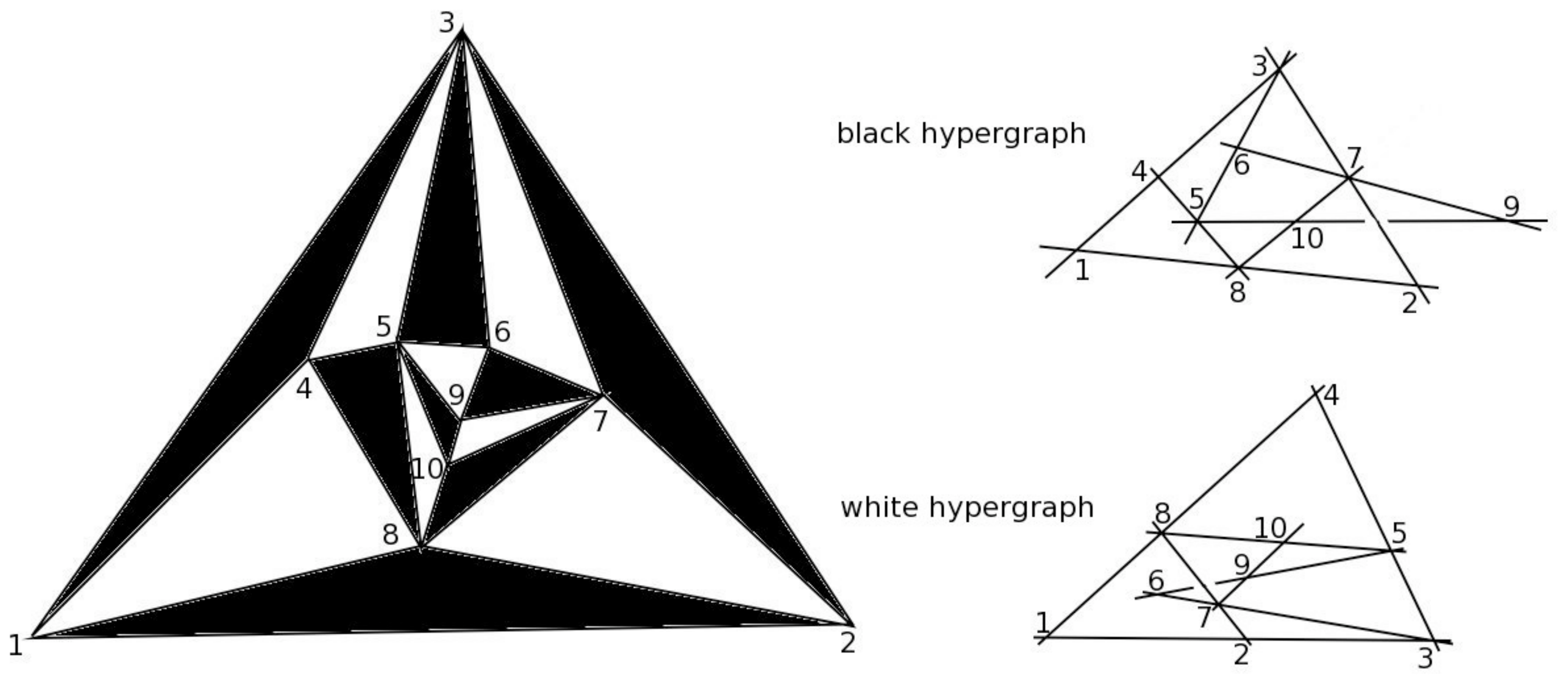}
\end{figure}
We will show that each of these collections is a hypertree.
These {\em spherical hypertrees} are irreducible unless
the triangulation is a connected sum
of two triangulations obtained by removing
a white triangle from one triangulation, a black triangle from another,
and then gluing along the cuts.
\end{Review}

One can ask if different hypertrees can give the same divisor of~$\oM_{0,n}$. 
This turns out to be a difficult question. We can prove the following

\begin{Theorem}
Let $\Gamma$ and $\Gamma'$ be \emph{generic} hypertrees (see Definition \ref{generic hypertree}). 
Then $$D_\Gamma=D_{\Gamma'}$$ if and and only if 
$\Gamma$ and $\Gamma'$ are the black and white hypertrees of an even triangulation of a sphere that is not a connected sum.
\end{Theorem}

In other words,  the map from the discrete ``moduli space'' of hypertrees 
to the set of vertices of the effective cone of $\oM_{0,n}$
generically looks like the normalization of the node 
(which corresponds to triangulations 
of~a $2$-sphere).
\begin{figure}[htbp]
\includegraphics[width=2.5in]{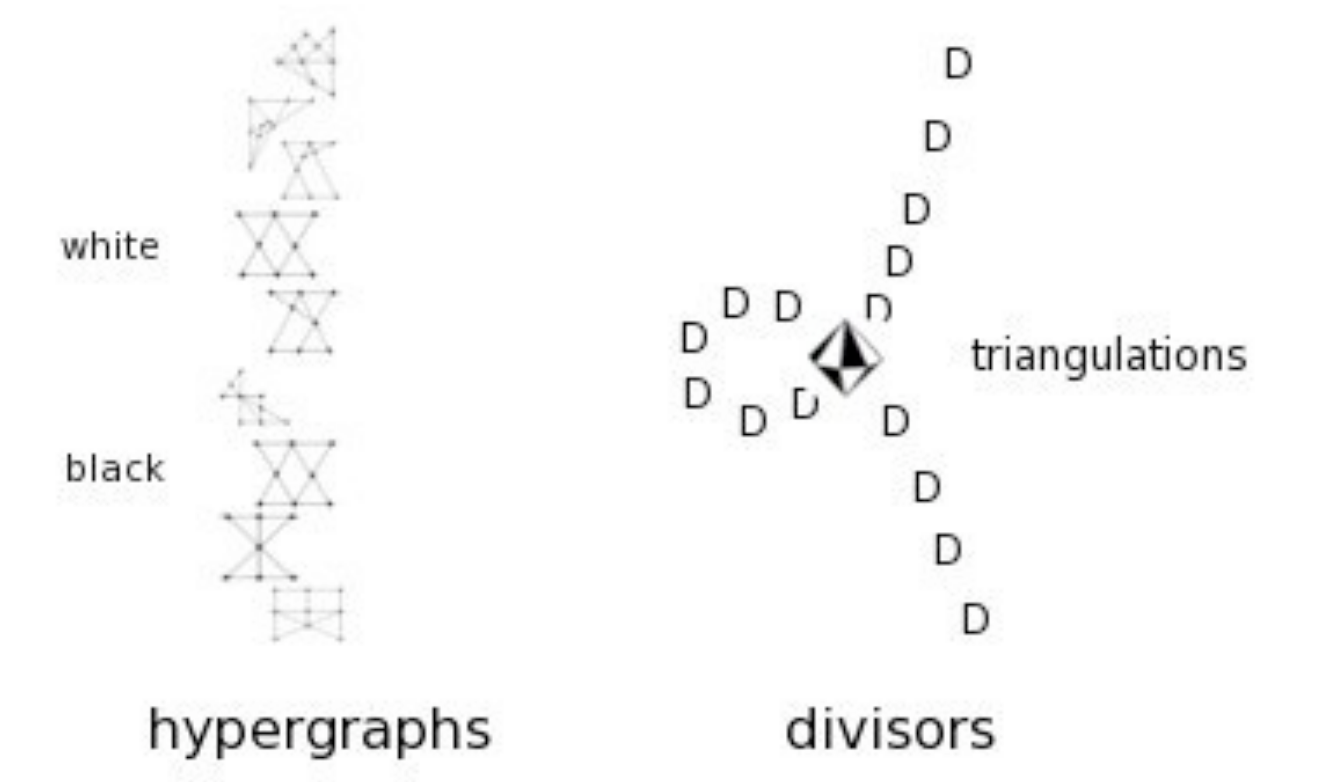}
\end{figure}
However, on the boundary of this discrete moduli space of hypertrees, the map is a more complicated ``contraction''.
For example, in \ref{HistoricalRemarks} we study the triangulation
of a bipyramid, when many hypertrees collapse to the same vertex.
This is an interesting case because the corresponding divisor $D_\Gamma$
is a pull-back of the classical Brill--Noether ``gonality'' divisor on $\oM_{2k+1}$ 
used by Harris and Mumford~\cite{HM}.

We would like to explain why divisors $D_\Gamma$ are exceptional,
i.e.,~how to construct a contracting birational map 
$f:\,\oM_{0,n}\dra X_\Gamma$ in Theorem~\ref{MainThOne}.
The map is called contracting if for one (and hence for any) resolution 
$$\begin{matrix}
&Z\cr
&\!\!\!\!\!\!\!\!\!\!\!\!\!\!\!\!\!\!\!\!\!\!\!{}^g\swarrow & \!\!\!\!\!\!\!\!\!\searrow^h\cr
\oM_{0,n} & \dra & X_\Gamma\cr
\end{matrix}$$
$g$-exceptional divisors are also $h$-exceptional.
A typical example is a composition of a small modification and a morphism.

To~explain the idea, take a general smooth curve $\Sigma$ of genus~$g$. 
By Brill--Noether theory~\cite{ACGH}, the variety 
$G^1_{g+1}$, parameterizing pencils of divisors of degree $g+1$ on~$\Sigma$, is smooth.
We have a natural morphism
$$v:\,G^1_{g+1}\to W^1_{g+1}\simeq \Pic^{g+1}(\Sigma),$$ 
which assigns to a pencil of divisors its linear equivalence class.
By Brill--Noether theory, $v$ is birational, 
and has an exceptional divisor $D$ 
over 
$$W^2_{g+1}=\{L\in\Pic^{g+1}(\Sigma)\,|\,h^0(L)\ge 3\},$$
which is non-empty and has codimension~$3$ in $\Pic^{g+1}(\Sigma)$.
So, for example, it is immediately clear that $D$ is an extremal ray of  $\Eff(G^1_{g+1})$\footnote{
Other extremal rays can be found using methods of Bauer--Szemberg \cite{BS}.}.

Generically, $G^1_{g+1}$ parameterizes globally generated pencils,
i.e.,~it contains a scheme of degree $g+1$ morphisms $\Sigma\to\bP^1$ (modulo automorphisms)
as an open subset. So $D$ generically parameterizes pencils that can be obtained by choosing a ``planar realization'',
i.e.,~a morphism $\Sigma\to\bP^2$, and then taking composition with the projection from a general point.

Next we degenerate a smooth curve to the union of rational curves
with combinatorics encoded in a hypertree.

\begin{Definition}\label{valence}
We work with schemes over an algebraically closed field~$k$.
A curve $\Sigma_\Gamma$ of genus
$$g=d-1$$ is called a {\em hypertree curve} if it has $d$ irreducible
components, each isomorphic to $\bP^1$ and marked by $\Gamma_j$, $j=1,\ldots,d$.
These components are glued at identical markings as a scheme-theoretic push-out:
at each singular point $i\in N$, $\Sigma_\Gamma$
is locally isomorphic to the union of coordinate axes in $\bA^{v_i}$, where $v_i$
is the valence of $i$, i.e., the number of subsets $\Gamma_j$ that contain~$i$.
We consider $\Sigma_\Gamma$ as a marked curve (by indexing its singularities).
\end{Definition}

The most common case is when all $\Gamma_j$'s are triples.
If this is not the case, then hypertree curves have moduli, namely
$$M_\Gamma:=\prod_{j=1,\ldots,d}M_{0,\Gamma_j}.$$
Then we have to adjust our construction a little bit:
$\Sigma_\Gamma$ will be the universal curve over the moduli space $M_\Gamma$.

By definition of the push-out,
 $M_{0,n}$ can be identified with a variety of morphisms $f:\,\Sigma_\Gamma\to\bP^1$
(modulo the free action of $PGL_2$) that send singular points
$p_1,\ldots,p_n$ of $\Sigma_\Gamma$  to different points $q_1,\ldots,q_n\in \bP^1$.
This gives a morphism 
\begin{equation}\label{vmap}
v:\,M_{0,n}\to \Pic^{\underline{1}},\quad f\mapsto f^*\cO_{\bP}(1)
\cooltag\end{equation} 
from $M_{0,n}$ to the (relative over $M_\Gamma$) Picard scheme 
$\Pic^{\underline{1}}$ of line bundles on $\Sigma$ of degree $1$ on each irreducible component.
This is the analogue of the map $G^1_{g+1}\to\Pic^{g+1}$ in the smooth case.
The locus $D_\Gamma\subset M_{0,n}$ defined above corresponds to the divisor $D$ in the smooth case.

We have to compactify the source and the target of the map~$v$.

\begin{Definition}
A nodal curve $\Sigma^s_\Gamma$, called a {\em stable hypertree curve},
is obtained by inserting a $\bP^1$ with $v_i$ markings 
instead of each singular point of $\Sigma_\Gamma$ with $v_i>2$.
\begin{figure}[htbp]
\includegraphics[width=4.5in]{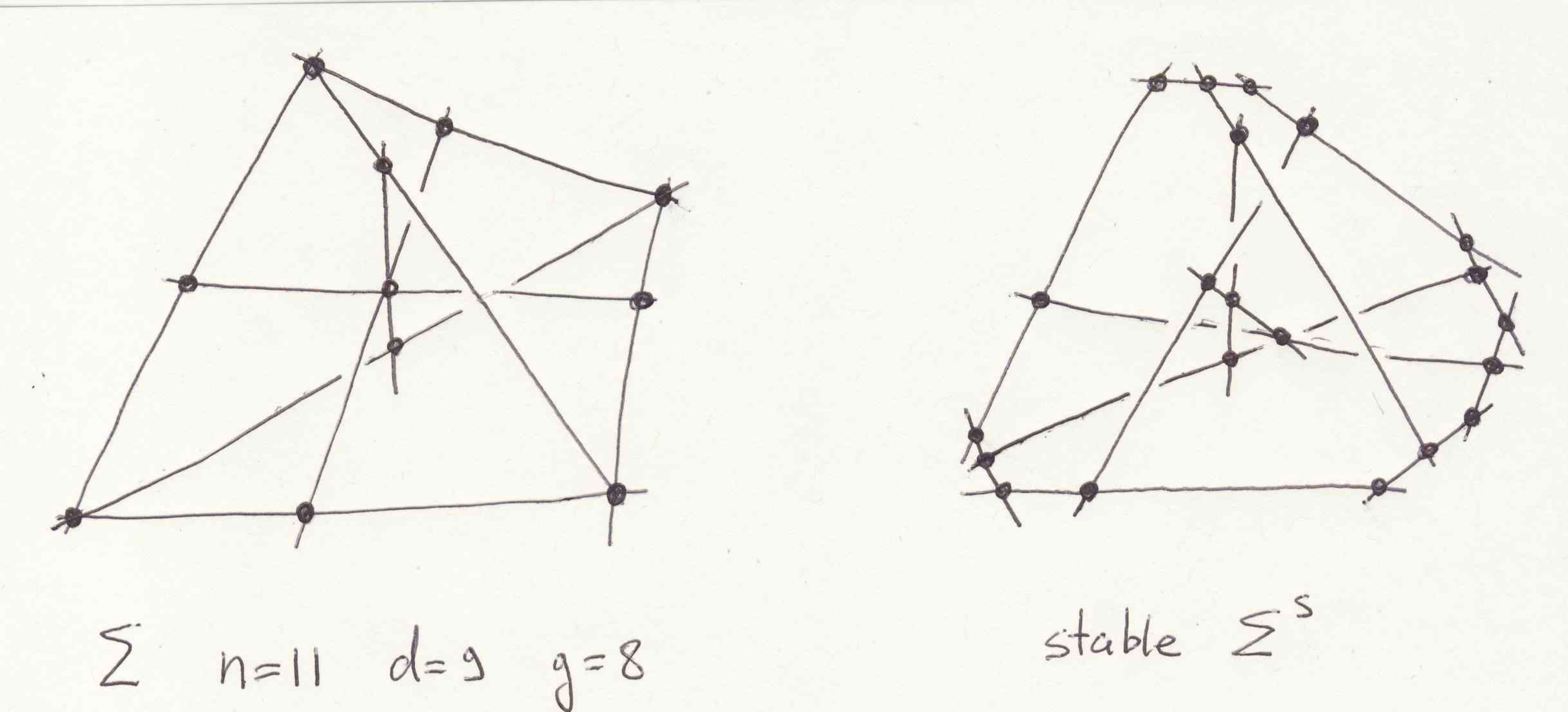}
\end{figure}
If $v_i>3$ then we do not allow extra moduli, instead we 
arbitrarily {\em fix} cross-ratios of marked points on inserted $\bP^1$'s.
Let $\Pic^{\underline{1}}$ be the 
Picard scheme of 
invertible sheaves on $\Sigma^s$ of degree $1$ on each irreducible component coming from 
$\Sigma$
and degree~$0$ on each component inserted at a non-nodal point of~$\Sigma$.
\end{Definition}

\begin{Theorem}\label{jacobianmain}
Let $\Gamma$ be an irreducible hypertree.
Any sheaf in  $\Pic^{\underline{1}}$ is Gieseker-stable w.r.t. the dualizing sheaf $\omega_{\Sigma^s}$.
Let $X_\Gamma$ be the normalization of the main component 
in the compactified Jacobian of $\Sigma^s$
relative over 
$$\oM_\Gamma=\prod_{j=1,\ldots,d}\oM_{0,\Gamma_j}$$ 
The map $v$ of \eqref{vmap}
induces a contracting birational map $v:\,\oM_{0,n}\dra X_\Gamma$
and $D_\Gamma$ is the only component of the exceptional locus that intersects $M_{0,n}$.
\end{Theorem}

\begin{Remarks}
(a) A stable hypertree curve is a special case of a graph curve
of Bayer and Eisenbud \cite{BE}.

(b) All irreducible hypertrees for small $n$ 
were
found by Scheidwasser \cite{Sch} using 
computer search.
Up to the action of $S_n$, there are $93$ hypertrees for $n=10$, $1027$ hypertrees for $n=11$, and so on. See 
Fig.~\ref{IlyaPics}
for all hypertrees for $n<10$.

(c) There are no irreducible hypertrees for $n=5$. This reflects the fact 
that the effective cone of 
$\oM_{0,5}\simeq\Bl_4\bP^2$
is generated by boundary divisors alone, i.e.,~by the ten $(-1)$-curves.

(d) The first proofs that 
$\Eff(\oM_{0,6})$ is not generated by boundary divisors
were found by Keel and Vermeire~\cite{V}. (In particular, this shows
that, for any $n\geq6$, $\Eff(\oM_{0,n})$ is not generated by boundary divisors.)
Their description of an extremal divisor is very different from ours,
which perhaps explains why it was not generalized to all $n$ before.
We will compare the two approaches in \ref{HistoricalRemarks}.

(e) Hassett and Tschinkel \cite{HT} proved that $\Eff(\oM_{0,6})$
is generated by boundary and Keel--Vermeire divisors.
So the Conjecture is true for $n=6$.
It was proved by the first author \cite{Ca} that in fact
the Cox ring of $\oM_{0,6}$ is generated by boundary and hypertree
divisors. A pipe dream would be to prove an analogous statement for any $n$.

(f) The existence of birational contractions $X_\Gamma$
supports the conjecture of Hu and Keel \cite{HK}
that $\oM_{0,n}$ is a Mori dream space. The map $v$ of \eqref{vmap}
is the first example of a birational contraction of $\oM_{0,n}$ whose exceptional locus 
intersects the interior $M_{0,n}$. Birational contractions whose exceptional locus 
lies in the boundary have been previously constructed by Hassett \cite{Hassett}.
In particular, the map $v$ gives a (hypothetical) new Mori chamber
of $\oM_{0,n}$. It would be interesting to factor $\oM_{0,n}\dra X_\Gamma$
through a small $\bQ$-factorial modification, which perhaps
has a functorial meaning. 

(g) We take only irreducible hypertrees in Theorem~\ref{MainThOne} because
if $\Gamma$ is not irreducible, then if we define $D_\Gamma$ as above, 
any component of $D_\Gamma$ will be equal to $\pi^{-1} (D_{\Gamma'})$,
where $\pi:\,\oM_{0,n}\to\oM_{0,k}$ is a forgetful map and $\Gamma'$ is an irreducible hypertree 
on a subset $K\subset N$ (see Lemma \ref{weakness}).

(h) As Fig.~\ref{IlyaPics} suggests, the number of new extremal rays grows rapidly with~$n$.
One reason for this is the existence of spherical hypertrees, another reason
is a ``Fibonaccian'' inductive construction (Theorem~\ref{FunnyConstruction})
that multiplies irreducible non-spherical hypertrees.

(i) Keel and McKernan~\cite{KM} proved that the effective cone of the symmetrization $\oM_{0,n}/S_n$ 
is generated by boundary divisors for any~$n$. So in some sense
our hypertree divisors reflect $S_n$-monodromy.

\end{Remarks}

Let us explain the layout of the paper. We start in \ref{BNSection} by introducing Brill--Noether loci of hypertree curves
and use a trick to show that a hypertree divisor (if non-empty) is an extremal ray of the effective cone of~$\oM_{0,n}$.
In~\ref{capacity} we introduce {\em capacity}, which measures how far is a collection
of subsets from being a hypertree. We relate capacity to the dimension of the image of a product of linear projections.
In \ref{Irred and Class} we use calculations with discrepancies to show that a hypertree divisor
is non-empty and irreducible. We also (partially) compute its class. In \ref{smaps} we study a compactified Jacobian
of a hypertree curve and show that $\oM_{0,n}$ is birationally contracted to~it. In \ref{planarity}
we prove the characterization of $D_\Gamma$ via projections of points given in the Introduction: in the previous sections
we define $D_\Gamma$ in a somewhat weaker fashion as a Brill--Noether locus.
In \ref{spherical} we study spherical  and \emph{generic} hypertrees. In particular, we show that if 
a hypertree is generic then the hypertree divisor uniquely determines the hypertree, except in the case 
when the hypertree is spherical (in which case the divisor uniquely determines the triangulation). We also give an 
inductive construction of many  non-spherical generic hypertrees.
Section \ref{DetEqs} is very elementary: we use basic linear algebra to deduce
determinantal equations for hypertree divisors. As a corollary, we show that black and white
hypertrees of a triangulated sphere give the same divisors on $\oM_{0,n}$.
In Section \ref{HistoricalRemarks} we relate hypertree divisors to gonality divisors on $\oM_g$
via various gluing maps $\oM_{0,n}\to\oM_g$. Finally, in Section \ref{pull-backs} we use the program Macaulay to give 
several examples of moving divisors on $\oM_{0,n}$ which are pull-backs of extremal divisors on $\oM_{g,k}$ via maps 
$\oM_{0,n}\to\oM_{g,k}$ ($n=2g+k$) obtained by gluing pairs of markings. These divisors on $\oM_{0,n}$
are linearly equivalent to sums of boundary (thus, at least in our examples, this construction does not lead to 
any new interesting divisors on $\oM_{0,n}$).

\medskip
\noindent\textsc{Acknowledgements.}
We are grateful to Sean Keel for teaching us $\oM_g$,
to Valery Alexeev and Lucia Caporaso
for answering our questions about compactified Jacobians, 
to Igor Dolgachev, Gabi Farkas, Janos Koll\`ar, and Bernd Sturmfels for useful discussions. 

Hypertrees for $n\le11$ were classified by Ilya Scheidwasser during an REU
directed by the second author. He also performed the most difficult combinatorial calculations in \ref{smaps}.
We are grateful to Ilya for the permission to reproduce his results
and for the beautiful pictures he made. The ``Fibonacci'' construction of Theorem~\ref{FunnyConstruction} was suggested to us by 
Anna Kazanova.

Parts of this paper were written while the first author was visiting the Max-Planck Institute in Bonn, Germany
and while both authors were members at MSRI, Berkeley.
The first author was partially supported by the NSF grant  DMS-1001157.
The second author was partially supported by the NSF grants 
DMS-0701191 and DMS-1001344  and by the Sloan research fellowship.

\tableofcontents

\section{Brill--Noether Loci of Hypertree Curves}\label{BNSection}\label{DivSection}\label{fibonacci}

We fix  a hypertree $\Gamma=\{\Gamma_1,\ldots,\Gamma_d\}$
and consider a hypertree curve $\Sigma$.

\begin{Definition}
A linear system on $\Sigma$ is called {\em admissible} if it 
is globally generated and the corresponding morphism $\Sigma\to\bP^k$
sends singular points of $\Sigma$ to different points.
An invertible sheaf~$L$ is called admissible if the complete linear system $|L|$ is admissible.
\end{Definition}

We define the {\em  Brill--Noether loci} $W^r$ and $G^r$ \cite{ACGH} as follows.
First suppose that $\Gamma$ consists of triples.
Then $\Sigma$ has genus $g=n-3$ and the Picard scheme
$\Pic^{\uone}$ of line bundles of degree $1$ on each irreducible component is isomorphic to~$\bG_m^g$
(not canonically).
The Brill--Noether locus $W^r$ parametrizes admissible line bundles $L\in \Pic^{\uone}$ 
such that 
$$h^0(\Sigma,L)\ge r+1.$$
The locus $G^r$ parametrizes admissible pencils on $\Sigma_\Gamma$
such that the corresponding line bundle is in $W^r$.
So we have a natural forgetful map
$$G^r \arrow^v W^r.$$

If $\Gamma$ contains not just triples, things get a little bit more complicated.
Let's give a functorial definition that works in general.
The~space $M_\Gamma$ defined in the Introduction represents a 
functor 
$$\cM_\Gamma:\,Schemes\to Sets$$
that sends a scheme $S$ to the set of isomorphism classes of flat families $\bSi\to S$
with reduced geometric fibers isomorphic to hypertree  curves.
A~hypertree curve is connected and it is easy to compute its genus
\begin{equation}\label{genusformula}
g=n-3-\dim M_\Gamma=d-1.
\cooltag
\end{equation}
Consider the relative Picard functor  
$$\cPic^{\uone}:Schemes\to Sets$$ 
that sends a scheme $S$
to an object $\bSi$ of $\cM_\Gamma(S)$  equipped with 
an invertible sheaf on $\bSi$ of multi-degree $(1,\ldots,1)$
modulo pull-backs of invertible sheaves on~$S$. 
This functor is represented by a $\bG_m^g$-torsor 
over $M_\Gamma$. This torsor is in fact trivial.
Notice that the dimension of $\Pic^{\uone}$ is always equal to $n-3$.
Let 
$$\cG^1:\,Schemes\to Sets$$ 
be a functor that sends a scheme $S$ to the set of isomorphism classes
of 
\begin{enumerate}
\item a family $\{p:\,\bSi\to S\}$ in $\cM_\Gamma(S)$;
\item a morphism  $f:\,\bSi\to\bP^1_S$
such that (a) images of irreducible components of $\bSi^{sing}$ are disjoint
and (b) each irreducible component of $\bSi$ maps isomorphically onto $\bP^1_S$.
\end{enumerate}
Here two morphisms are considered isomorphic if they differ by isomorphisms of $S$-schemes 
both on the source and the target.
Let 
$$v:\,\cG^1\to\cPic^{\uone}$$ be the natural transformation such that
$$(\bSi\to S,f:\,\bSi\to\bP^1_S)\mapsto(\bSi\to S,f^*\cO_{\bP^1_S}(1)).$$
We will see below that $\cG^1$ is represented by $M_{0,n}$.
For any $r\ge2$, let $G^r\subset G^1$ be a closed subset
(with an induced reduced scheme structure) of points where 
$p_*(f^*\cO_{\bP^1_{G^1}}(1))$ has rank at least $r+1$ (where $(p,f)$ is the universal family of $\cG^1$).
We define $W^r\subset \Pic^{\uone}$ as a scheme-theoretic image of $G^r$.

\begin{Definition}
Let $f:\,X\to Y$ be a quasiprojective morphism  of Noetherian schemes.
The  {\em exceptional locus} $\Exc(f)$ is  a complement to the union of points in $X$ isolated in their fibers.
$\Exc(f)$ is closed \cite[4.4.3]{EGA3}.
\end{Definition}

\begin{Definition}
An extremal ray $R$ of a closed convex cone $\cC \subset \bR^s$
is called an {\em edge} if the vectorspace
$R^{\perp} \subset (\bR^s)^*$ (of linear forms that vanish
on~$R$) is generated by supporting hyperplanes for $\cC$. 
This technical condition means that $\cC$ is ``not rounded'' at $R$.
\end{Definition}

\begin{Theorem}\label{ID}\label{r-DimlFibers}\label{coolcondition}
The functor $\cG^1$ is represented by $M_{0,n}$.
The map 
$$v:\,M_{0,n}\simeq G^1\to \Pic^{\underline{1}}$$
is birational.
Its exceptional locus is~$G^2$.
The map~$v$ induces an isomorphism 
\begin{equation}\label{opossum}
M_{0,n}\setminus G^2\simeq v(M_{0,n}\setminus G^2)=W^1\setminus W^2\subset\Pic^{\underline{1}}.
\cooltag\end{equation}
Any irreducible component of $G^2$
is a divisor whose closure in $\oM_{0,n}$ generates an edge of $\oEff(\oM_{0,n})$.
The closure of the pre-image of $G^2$ in $\oM_{0,n+1}$ with respect to the forgetful map
$M_{0,n+1}\to M_{0,n}$ is contracted by a birational morphism
\begin{equation}\label{Mapppi}
\prod\limits_{j=1}^d\pi_{\Gamma_j\cup\{n+1\}}:\,
\oM_{0,n+1}\to\prod\limits_{j=1}^d\oM_{0,\Gamma_j\cup\{n+1\}}.
\cooltag\end{equation}
All other exceptional divisors of this morphism belong to the boundary.
\end{Theorem}

\begin{Remark}
In subsequent sections we will show that if the hypertree is irreducible, then 
$G^2$ is non-empty and irreducible. 
By definition, a point in $G^2$ can be obtained by mapping a hypertree curve to $\bP^2$
and projecting its singular vertices from a point.
The definition of the divisor $D_{\Gamma}$ in the Introduction
is stronger, but eventually we will show that $D_\Gamma=\bar G^2$.
\end{Remark}

\begin{Remark}\label{more general 2}
If we consider collections $\Gamma$ satisfying only the first two conditions in 
Definition \ref{hypertree def} (i.e., the convexity and normality axioms may fail), one may similarly define 
hypertree curves and Brill-Noether loci. The map $v$ may not be birational anymore, but
parts of Theorem \ref{} still hold: the functor $\cG^1$ is represented by $M_{0,n}$ and the 
exceptional locus of the map $v$ is $G^2$. In Section \ref{capacity} we will give conditions under which
the map $v$ is birational onto its image (see Remark \ref{more general}).
\end{Remark}

\begin{proof}[Proof of Theorem~\ref{coolcondition}] 
We proceed in several steps.

\begin{Review}
Each datum $(\bSi\to S,\ f:\,\bSi\to\bP^1_S)\in\cG^1(S)$ gives rise
to an isomorphism class of a flat family over $S$ with reduced geometric fibers
given by~$\bP^1$ and with $n$ disjoint sections given by images of irreducible components of $\bSi^{sing}$.
This gives a natural transformation $\cG^1\to\cM_{0,n}$ which is in fact a natural isomorphism:
given a flat
family of marked $\bP^1$'s, we can just push-out $d$ copies of $\bP^1_S$ along sections in each $\Gamma_i$.
This gives a flat family of hypertree curves over $S$ and its map to $\bP^1_S$, i.e.,~a datum in $\cG^1(S)$.
\end{Review}

\begin{Review}
Next we define two auxilliary Brill--Noether loci, $C^r$ and $\tG^r$.
We call an effective Cartier divisor {\em admissible} if it does not contain singular points.
On the level of geometric points, 
$$C^r=\{\hbox{\rm a curve $\Sigma$, an admissible divisor}\ D\ \hbox{\rm on $\Sigma$ such that}\  \cO(D)\in W^r\},$$
$$\tilde G^r=\{(L,V)\in G^r,\ \hbox{\rm an admissible}\ D\in|V|\}.$$
These loci fit in the natural  commutative diagram of forgetful maps 
$$\begin{CD}
\tG^r @>>> C^r\\
@VVV @VV{D\mapsto\cO(D)}V \\
G^r @>>> W^r
\end{CD}$$

On the scheme-theoretic level, 
let $\bSi^{sm}$ be the smooth locus of the universal family $\bSi\to\cM_\Gamma$
with irreducible components $\bSi^{sm}_1,\ldots,\bSi^{sm}_d$.
Let 
$$\cC^0=\bSi^{sm}_1\times_{\cM_\Gamma}\ldots\times_{\cM_\Gamma}\bSi^{sm}_d$$
and let  
$$u:\,\cC^0\to\cPic^{\uone}$$
be the Abel map that sends $(p_1,\ldots,p_d)\in \cC^0(S)$
to $\cO_{\bSi}(p_1+\ldots+p_d)$. 
Geometric fibers of $u$ are open subsets of admissible divisors
in complete linear systems on $\Pic^{\uone}(\Sigma_k)$.
Let 
$$\cC^r:=u^{-1}(\cW^r)\subset\cC^0.$$

Finally, we define $\tilde\cG^1$ as a functor $Schemes\to Sets$
that sends $S$ to the datum $(\bSi\to S,\ f:\,\bSi\to\bP^1_S)\in\cG^1(S)$ and a section $s:\,S\to \bP^1_S$
disjoint from images of irreducible components of $\bSi\setminus\bSi^{sm}$.
We define $\tilde\cG^r$ as the preimage of $\cG^r$ for the forgetful map $\tilde\cG^r\to\cG^r$.
We also have the natural transformation $\tilde\cG^r\to\cC^0$
that sends $(\bSi\to S,\ f:\,\bSi\to\bP^1_S,s)$ to $f^{-1}(s(S))$.
It factors through $\cC^r$.
The same argument as above shows that $\tilde\cG^1$ is isomorphic to $\cM_{0,n+1}$
and that $\cC^0$ is isomorphic to $\prod\limits_{j=1}^d\cM_{0,\Gamma_j\cup\{n+1\}}$.
\end{Review}

\begin{Review}\label{bigdiagram}
To summarize, , we have the following commutative diagram
$$\xymatrix{\tilde\cG^1 \ar@{=}[dd] \ar[rd] \ar[rr]^V && \cC^1\ar[rr]\ar[rd] && \cC^0 \ar@{=}'[d][dd] \ar[rd]^u \\
& \cG^1\ar@{=}[dd] \ar[rr]^v && \cW^1 \ar[rr] && \cPic^1\ar[dd] \\
\cM_{0,n+1} \ar'[r][rrrr]^{\prod\limits_{j=1}^d\pi_{\Gamma_j\cup\{n+1\}}} \ar[rd]_{\pi_N} &&&&  \prod\limits_{j=1}^d\cM_{0,\Gamma_j\cup\{n+1\}} 
\ar[rd]^{\prod\limits_{j=1}^d\pi_{\Gamma_j}} \\
& \cM_{0,n} \ar[rrrr]^{\prod\limits_{j=1}^d\pi_{\Gamma_j}}  &&&&  \cM_\Gamma }
$$
where, for any subset $I\subset N$ with $|I|\ge4$,  
$$\pi_I:\,\oM_{0,n}\to\oM_{0,I}$$
is  the morphism given by dropping the points of $N\setminus I$
(and stabilizing). 
\end{Review}

\begin{Review}\label{}
It is clear from the definition that the exceptional locus of $v$ is exactly $G^2$ and that $v$ is birational
if and only if $G^1\ne G^2$. This is equivalent to $\tG^1\ne\tG^2$, which is equivalent
to $V$ being birational. This is proved in Theorem~\ref{size/explicit}.
\end{Review}

\begin{Review}
Finally, we note that $\tG^2$ is the preimage of $G^2$.
Since the closure of $\tG^2$ is in the exceptional locus of the regular morphism
\eqref{Mapppi}, 
Lemma~\ref{basicobs} below shows that 
the closure of any irreducible component of $G^2$ in $\oM_{0,n}$
is a divisor that generates an edge of $\oEff(\oM_{0,n})$.
\end{Review}
This finishes the proof of the Theorem.
\end{proof}

\begin{Lemma}\label{basicobs}
Consider the diagram of morphisms 
$$\begin{CD}
X @>f>> Y\\
@VpVV\\
Z
\end{CD}$$
of projective $\bQ$-factorial varieties.
Suppose that $f$ is birational and that $p$ is faithfully flat.
Let $D$  be an irreducible  component of $\Exc(f)$.
If $p(D)\ne Z$ and a generic fiber of $p$ along $p(D)$ is irreducible 
then $p(D)$ is a divisor that generates an extremal ray (in fact an edge) of~$\oEff(Z)$.
\end{Lemma}

\begin{proof}
$D$ is a divisor by van der W\"arden's purity theorem~\cite[21.12.12]{EGA4}.
It is well-known that it generates an edge of $\oEff(X)$.
Since $p$ is flat and $p(D)\ne Z$, $p(D)$ is an irreducible
divisor. Since $p^{-1}(p(D))$ is irreducible (e.g.~by \cite[Lem.~2.6]{T}), $p^{-1}(p(D))=D$.
It follows that $p(D)$ generates an edge of $\oEff(Z)$ because 
$\oEff(Z)$ injects in  $\oEff(X)$ by the pull-back $p^*$.
\end{proof}

\begin{Remark}
An interesting feature of this argument is that 
we study divisors $\ov{G^2}\subset\oM_{0,n}$ by pulling them to $\oM_{0,n+1}$
and then contracting the preimage by a birational morphism.
This gives a method of proving extremality of divisors by a flat base change.
In~\ref{smaps} we will contract $\ov{G^2}$ 
by a contracting birational map (but not a morphism) from $\oM_{0,n}$
to the compactified Jacobian of the stable hypertree curve $\Sigma^s$.  
\end{Remark}

\section{Capacity and Product of Linear Projections}\label{capacity}

\begin{Definition}
Let $\Gamma=\{\Gamma_\alpha\}$ be an arbitrary collection of subsets of the set $N=\{1,\ldots,n\}$
such that each subset has at least three elements.
We define its {\em capacity} as $$\Cap(\Gamma)=\max_{\Gamma'}\left\{\sum_{\beta}\big(|\Gamma'_{\beta}|-2\big)\right\},$$
where $\Gamma'$ runs through all sub-collections of $\Gamma$
that satisfy the convexity axiom~\eqref{CondS}.
Here $\Gamma'=\{\Gamma'_\beta\}$ is a sub-collection of $\Gamma$ if each 
$\Gamma'_\beta$ is a subset of some $\Gamma'_\alpha$.
For example, if $\Gamma$ is a hypertree then 
$$\Cap(\Gamma)=n-2$$ 
by the convexity and normalization axioms.
\end{Definition}

\begin{Theorem}\label{size/explicit}
Let $\Gamma=\{\Gamma_\alpha\}$ be an arbitrary collection of subsets of the set $N=\{1,\ldots,n\}$
such that each subset has at least three elements and $\Gamma_\alpha\not\subset\Gamma_\beta$ if $\alpha\neq\beta$.
The capacity of $\Gamma$ is equal to 
the dimension of the image of the map
\begin{equation}\label{MHVHMVJHVJ}
\pi_{\Gamma\cup\{n+1\}}:=\prod\limits_{j=1}^d\pi_{\Gamma_j\cup\{n+1\}}:\qquad
\oM_{0,n+1}\to\prod\limits_{j=1}^d\oM_{0,\Gamma_j\cup\{n+1\}}.
\cooltag\end{equation}
Moreover, $\pi_{\Gamma\cup\{n+1\}}$ is birational onto its image if and only if $\Gamma$ has maximum capacity 
$n-2$. In particular, 
$\pi_{\Gamma\cup\{n+1\}}$  if and only if $\Gamma$ satisfies \eqref{CondS} and \eqref{normalizat}.
\end{Theorem}

\begin{Remark}\label{more general}
Let $\Gamma$ be an arbitrary collection of subsets of the set $N=\{1,\ldots, n\}$,  
satisfying the first two conditions in Definition \ref{hypertree def} (see Remark \ref{more general 2}).
If $\Gamma$ has  maximum capacity $n-2$, by Theorem \ref{size/explicit}  
the map $\pi_{\Gamma\cup\{n+1\}}$ is birational onto its image and it follows that 
$G^2$ is a proper subset of $G^1$.
\end{Remark}

To prove Theorem \ref{size/explicit} we need two lemmas on linear projections.

\begin{Definition}
For a projective subspace $U\subset\bP^r$, let 
$$\pi_U:\,\PP^r\dra\PP^{l(U)}$$ 
be a linear projection from $U$, where $l(U)=\codim U-1$.
\end{Definition}

\begin{Lemma}\label{linalg}
Let $U_1,\ldots,U_s\subset\PP^r$ be subspaces such that $U_i\not\subset U_j$ when $i\ne j$.
Then (a) the rational map 
$$\pi=\pi_{U_1}\times\ldots\times\pi_{U_s}:\,\PP^r\dra\PP^{l(U_1)}\times\ldots\times
\PP^{l(U_s)}$$ is dominant
if and only if 
\begin{equation}\label{dim inters}
l\left(\bigcap_{i\in S} U_i\right)\ge \sum_{i\in S}l(U_i)\quad\hbox{for any $S\subset \{1,\ldots,s\}$.}
\cooltag\end{equation}
(b) If $r=l(U_1)+\ldots+l(U_s)$ and $\pi$ is dominant then $\pi$ is birational.
\end{Lemma}

\begin{proof}
Let $l_i:=l(U_i)$.
The scheme-theoretic fibers of the morphism $\PP^r\setminus\bigcup_iU_i\to\PP^{l_1}\times\ldots\times\PP^{l_s}$
are open subsets of projective subspaces. This implies~(b).
Now assume that $\pi$ is dominant but (\ref{dim inters}) is not satisfied, for
example we may assume that $W=U_1\cap\ldots\cap U_m$ has dimension
\begin{equation}\label{hohoho}
w\geq r-(l_1+\ldots+l_m).
\cooltag\end{equation}  
The projections $\pi_{U_i}$ for $i=1,\ldots,m$ factor through
the projection $\pi_W:\,\PP^r\dra\PP^{r-w-1}$.
It follows that the map:
$$\pi'=\pi_{U_1}\times\ldots\times\pi_{U_m}:\,\PP^r\dra\PP^{l_1}\times\ldots\times\PP^{l_m}$$ factors through $\pi_W$.
If $\pi$ is dominant, then so is $\pi'$, and therefore
the induced map $\PP^{r-w-1}\dra\PP^{l_1}\times\ldots\times\PP^{l_m}$ is dominant, which contradicts~\eqref{hohoho}.

Assume (\ref{dim inters}). We'll show that $\pi$ is dominant. 
We argue by induction on $r$.
Let $H$ be a general hyperplane containing $U_s$.
It suffices to prove that the restriction of $\pi_{U_1}\times\ldots\times\pi_{U_{s-1}}$ on $H$
is dominant.
Subspaces $U'_i:=U_i\cap H$ have codimension~$l_i+1$ in $H$ and, therefore, by induction assumption,
it suffices to prove that 
\begin{equation}\label{dim intersprime}
\dim\bigcap_{i\in S} U'_i<(r-1)-\sum_{i\in S}l_i\quad\hbox{for any $S\subset\{1,\ldots,r-1\}$.}
\cooltag\end{equation}
Let $W:=\bigcap\limits_{i\in S} U_i$. Let $L:=\sum\limits_{i\in S}l_i$.
By   \eqref{dim inters}, $\dim W<r-L$ and, therefore,  
$\dim H\cap W<r-L-1$ (i.e.,~we have \eqref{dim intersprime})  
unless  $W\subset U_s$. But in the latter case
$\dim\bigcap\limits_{i\in S} U'_i=\dim(U_s\cap W)<r-(l_s+L)$ by   \eqref{dim inters}.
\end{proof}

We would like to work out the case when all subspaces $U_1,\ldots,U_s$ are intersections
of subspaces spanned by subsets of points $p_1,\ldots,p_n\in\bP^{n-2}$
in linearly general position.
Let $N=\{1,\ldots,n\}$.
For any non-empty subset $I\subset N$, let $H_I=\langle p_i\rangle_{i\not\in I}$.

\begin{Lemma}\label{crazystuff}
The rational map 
$$\pi=\pi_{H_{\Gamma_1}}\times\ldots\times\pi_{H_{\Gamma_l}}:\,\bP^{n-2}\dra\bP^{|\Gamma_1|-2}\times\ldots\times\bP^{|\Gamma_l|-2}$$
is dominant if and only if \eqref{CondS} holds.
It is birational if and only if \eqref{CondS} and \eqref{normalizat} hold.
\end{Lemma}

\begin{proof}
For any $S\subset\{1,\ldots,l\}$, let 
$e_S$ be the number of connected components of $N$ (with respect to $\{\Gamma_i\}_{i\in S}$)
that have at least two elements.
Let 
$$\cH_S=\bigcap\limits_{i\in S}H_{\Gamma_i}.$$

Let $W\subset\bA^{n}_{x_1,\ldots,x_{n}}$ be a hyperplane $\sum x_i=0$.
In appropriate coordinates,
$\bP(W)$ is a projective space dual to $\bP^{n-2}$ and subspaces $H_I\subset\bP^{n-2}$ 
are projectively dual to projectivizations of linear subspaces $\langle x_i-x_j\rangle_{i,j\in I}$.
It follows that $\cH_S$ is projectively dual 
to a subspace $\langle e_i-e_j\rangle_{\exists k\in S:\ i,j\in\Gamma_k}$, which implies 
that 
$$
l(\cH_S)=|\bigcup_{i\in S}\Gamma_i|-e_S-1.
$$
By~Lemma~\ref{linalg}, it follows that $\pi$ is dominant if and only if
\begin{equation}\label{CondC}
|\bigcup_{i\in S}\Gamma_i|-e_S-1\ge\sum_{i\in S}(|\Gamma_i|-2)
\quad\hbox{\rm for any $S\subset\{1,\ldots,l\}$}.
\cooltag\end{equation} 

It remains to check that \eqref{CondC} and \eqref{CondS} are equivalent. 
It is clear that \eqref{CondC} implies \eqref{CondS}.
Now assume  \eqref{CondS}.
Let $I_1,\ldots,I_{e_S}$ be connected components of~$N$ (with respect to $\{\Gamma_i\}_{i\in S}$)
that have at least two elements.
This gives a partition $S=S_1\sqcup\ldots\sqcup S_{e_S}$ such that 
$I_k=\bigcup\limits_{j\in S_k}\Gamma_j$ for any~$k$. 
Applying  \eqref{CondS} for each $S_k$ gives
$$|\bigcup_{i\in S}\Gamma_i|-e_S-1\ge\sum_k\bigl(|\bigcup_{i\in S_k}\Gamma_i|-2\bigr)\ge\sum_k\sum_{i\in S_k}(|\Gamma_i|-2)=\sum_{i\in S}(|\Gamma_i|-2)$$
and this is nothing but \eqref{CondC}.
\end{proof}

\begin{proof}[Proof of Theorem~\ref{size/explicit}] 
Let $p_1,\ldots,p_n\in\bP^{n-2}$ be general points.
We have a birational morphism
$$\Psi:\,\oM_{0,n+1}\to\bP^{n-2}$$
(the Kapranov blow-up model), which is 
an iterated blow-up of $\PP^{n-2}$ along the points $p_1,\ldots, p_n$,
the proper transforms of lines connecting these points, and so on.
Moreover, we have a commutative diagram of rational maps
$$
\begin{CD}
\oM_{0,n+1} @>{\Psi}>> \bP^{n-2}\\
@V{\pi_{S\cup\{n+1\}}}VV                 @VV{p_S}V\\
\oM_{0,k+1}  @>{\Psi}>> \bP^{k-2}\\
\end{CD}$$
for each subset $S\subset N$ with $k$ elements,
where $p_S$ is a linear projection away from 
the linear span of points $p_i$ for $i\not\in S$, see \cite{Ka}.
It follows that the ``moreover'' part of the theorem is just 
a reformulation of Lemma~\ref{crazystuff}.

Let 
$$Z\subset \oM_{\Gamma\cup\{n+1\}}:=\prod\limits_{j=1}^d\oM_{0,\Gamma_j\cup\{n+1\}}$$
 be the image of $\pi_{\Gamma\cup\{n+1\}}$.
Notice that $\pi_{\Gamma'\cup\{n+1\}}$ factors through $\pi_{\Gamma\cup\{n+1\}}$
for any sub-collection $\Gamma'$. So it follows from Lemma~\ref{crazystuff}
that 
$$\dim Z\ge\Cap(\Gamma)$$
and that, to prove an opposite inequality, it suffices to show the following.
Suppose that $Z\ne \oM_{\Gamma\cup\{n+1\}}$. We claim that 
one can choose a proper sub-collection $\Gamma'$ such that $\dim p(Z)=\dim Z$, where
$$p:\,  \oM_{\Gamma\cup\{n+1\}}\to \oM_{\Gamma'\cup\{n+1\}}$$
is an obvious projection. Consider all possible maximal sub-collections, 
i.e.,~let $J$ be an indexing set obtained by taking $|\Gamma_\alpha|$ for each $\Gamma_\alpha$.
For each $j\in J$, let $\Gamma'_j$ be a sub-collection obtained by removing the corresponding index 
from the corresponding $\Gamma_\alpha$. Let $z\in Z$ be a general smooth point.
Notice that $z$ projects into $M_{0,\Gamma_\alpha\cup\{n+1\}}$ for each $\alpha$,
and so for each $j\in J$, the fiber of $p_j:\,  \oM_{\Gamma\cup\{n+1\}}\to \oM_{\Gamma'_j\cup\{n+1\}}$
passing through $z$ is a smooth rational curve. Moreover, it is easy to see that tangent vectors
to these rational curves at $z$ generate the tangent space to $\oM_{\Gamma\cup\{n+1\}}$ at $z$.
Since $Z$ is smooth at $z$, it follows that $p_j|_Z$ is generically finite for one of the projections.
\end{proof}

We have to refine Theorem~\ref{size/explicit}
to see how the map 
\begin{equation}\label{mappi}
\pi:\,\oM_{0,n+1}\to\M_{\Gamma\cup\{n+1\}}:=\prod_{\Gamma_{\alpha}}\M_{0,\Gamma_{\alpha}\cup\{n+1\}}
\cooltag\end{equation} 
affects the divisors of $\oM_{0,n+1}$. We borrow a definition
from matroid theory. 

\begin{Definition}
Let $I\subset N$ be any subset.
We define the {\em contracted collection} $\Gamma_I$ to be the collection of subsets of $I\cup\{p\}$ 
obtained from $\Gamma$ by replacing all the indices in $I^c$ with $p$ 
(and removing all subsets with less than three elements).
We define the {\em restricted collection} $\Gamma'_I$ to be the collection of subsets in $I^c$ given by 
intersecting each $\Gamma_\alpha$ with~$I^c$ (and removing subsets with less than three elements).
\end{Definition}

\begin{Lemma}\label{firstrow} 
For any hypertree $\Gamma$ we have
$$\codim\pi\big(\de_{I\cup\{n+1\}}\big)-1=n-3-\Cap(\Gamma_I)-\Cap(\Gamma'_I).$$
\end{Lemma}

\begin{proof}
For $I\subset N$, consider the products of forgetful maps:

\begin{equation*}
\pi_I:\M_{0,I\cup\{p,n+1\}}\ra\prod_{\Gamma_{\alpha}\subset I}\M_{0,\Gamma_{\alpha}\cup\{n+1\}}\times 
\prod_{\Gamma_{\alpha}\cap I^c\neq \eset, |\Gamma_{\alpha}\cap I|\geq2}\M_{0,(\Gamma_{\alpha}\cap I)\cup\{p, n+1\}}.
\end{equation*}

\begin{equation*}
\pi'_{I}:\M_{0,I^c\cup\{p\}}\ra\prod_{|\Gamma_{\alpha}\cap I^c|\geq3}\M_{0,(\Gamma_{\alpha}\cap I^c)\cup\{p\}},
\end{equation*}
By Theorem~\ref{size/explicit}, we have
$$
\dim\Im(\pi_I)=\Cap(\Gamma_I)
\quad\hbox{\rm and}\quad
\dim\Im(\pi'_I)=\Cap(\Gamma'_I).
$$
Note that 
$$\de_{I\cup\{n+1\}}\simeq\M_{0,I\cup\{p,n+1\}}\times\M_{0,I^c\cup\{p\}}$$
and the restriction of the map $\pi$ to 
$\de_{I\cup\{n+1\}}$
factors as the product $\pi_I\times\pi'_I$ followed by a closed embedding. 
\end{proof}

\begin{Lemma}\label{max size}
Let $\Gamma$ be an irreducible hypertree and let $I\subset N$ be a subset such that $2\leq|I|\leq n-2$
and either $I^c\subseteq\Gamma_{\beta}$ for some 
$\beta$, or $|I^c|=2$. Then 
$$\Cap(\Gamma_I)=|I\cup\{p\}|-2.$$
\end{Lemma}

\begin{proof}
We construct a sub-collection $\Gamma'$ of $\Gamma_I$ that satisfies the convexity axiom and 
$\sum_{\alpha}\big(|\Gamma'_{\alpha}|-2\big)=|I|-1$. Without loss of generality, we may assume $I^c=\{1,\ldots,l\}$. 
We define $\Gamma'$ as follows: 
\begin{itemize}
\item[(i) ] If $I^c\subsetneq\Gamma_{\beta}$, let $\Gamma'=\Gamma_I$; 
\item[(ii) ] If $I^c=\Gamma_{\beta}$ or if $|I^c|=l=2$ and $I^c\nsubseteq\Gamma_{\alpha}$ for any $\alpha$: we may assume that
$1\in\Gamma_1$ (note that $\Gamma_1\cap I^c=\{1\}$). Let $\Gamma'_1=\Gamma_1\setminus\{1\}$ if $|\Gamma_1|\geq4$ 
(omit $\Gamma'_1$ otherwise), $\Gamma'_{\alpha}=(\Gamma_I)_{\alpha}$ for all $\alpha\neq1$. 
\end{itemize}

Note that in all the cases $\sum_{\alpha}\big(|\Gamma'_{\alpha}|-2\big)=|I|-1$. Hence, the condition~(\ddag) holds for the set of all indices $\alpha$ that appear in $\Gamma'$. 
Assume that (\ddag) fails for a proper subset $T$ of indices $\alpha$: 
\begin{equation}\cooltag\label{1}
|\bigcup_{\alpha\in T}\Gamma'_{\alpha}|\leq\sum_{\alpha\in T}\big(|\Gamma'_{\alpha}|-2\big)+1\leq\sum_{\alpha\in T}\big(|\Gamma_{\alpha}|-2\big)+1.
\end{equation}
Let $k=|\big(\bigcup_{\alpha\in T}\Gamma_{\alpha}\big)\cap I^c|\leq l$. Then we have 
\begin{equation}\cooltag\label{2}
|\bigcup_{\alpha\in T}\Gamma'_{\alpha}|\geq |\bigcup_{\alpha\in T}\Gamma_{\alpha}|-k+1.
\end{equation}
Since $\Gamma$ is an irreducible hypertree, we have:
\begin{equation}\cooltag\label{3}
|\bigcup_{\alpha\in T}\Gamma_{\alpha}|\geq\sum_{\alpha\in T}\big(|\Gamma_{\alpha}|-2\big)+3.
\end{equation}
By (\ref{1}), (\ref{2}), (\ref{3}) we have $k\geq3$. This is a contradiction if $|I^c|=2$. 

Assume now that $I^c\subset\Gamma_{\beta}$. Let $k'=|\big(\bigcup_{\alpha\in T}\Gamma_{\alpha}\big)\cap\Gamma_{\beta}|$.
Then $k\leq k'$.  Consider the case when $\beta\notin T$. Since $\Gamma$ is an irreducible hypertree, we have:
\begin{equation}\cooltag\label{4}
|\bigcup_{\alpha\in T}\Gamma_{\alpha}|+|\Gamma_{\beta}|-k'=
|\bigcup_{\alpha\in T}\Gamma_{\alpha}\cup\Gamma_{\beta}|\geq\sum_{\alpha\in T}\big(|\Gamma_{\alpha}|-2\big)+\big(|\Gamma_{\beta}|-2\big)+3.
\end{equation}
By (\ref{1}), (\ref{2}), (\ref{4}) it follows that 
$$\sum_{\alpha\in T}\big(|\Gamma_{\alpha}|-2\big)+1\geq \sum_{\alpha\in T}\big(|\Gamma_{\alpha}|-2\big)+k'-k+2,$$
which is a contradiction, since $k'-k\geq0$.

Consider the case when $\beta\in T$ (only possible in case (i)). As (\ddag) fails, 
\begin{equation}\cooltag\label{1'}
|\bigcup_{\alpha\in T}\Gamma'_{\alpha}|\leq\sum_{\alpha\in T}\big(|\Gamma_{\alpha}|-2\big)+1\leq\sum_{\alpha\in T, \alpha\neq\beta}\big(|\Gamma_{\alpha}|-2\big)
+\big(|\Gamma_{\beta}|-l-1\big)+1.
\end{equation}
It follows from (\ref{1'}), (\ref{2}), (\ref{3}) that 
$$\sum_{\alpha\in T}\big(|\Gamma_{\alpha}|-2\big)-l+2\geq \sum_{\alpha\in T}\big(|\Gamma_{\alpha}|-2\big)-l+4,$$
which is a contradiction. This finishes the proof.
\end{proof}

\begin{Lemma}\label{capacity I'}
For any hypertree $\Gamma$ (not necessarily irreducible), 
the collection $\Gamma'_I$  satisfies the convexity axiom~\eqref{CondS}. In particular,
\begin{equation}
\Cap(\Gamma'_I)=
\sum_{|\Gamma_{\alpha}\cap I^c|\geq3}\big(|\Gamma_{\alpha}\cap I^c|-2\big).
\cooltag\label{bratstvo}\end{equation}
If moreover, $\Gamma$ is an irreducible hypertree and if $|I^c|=2$ or if $I^c\subseteq\Gamma_\alpha$ for some $\alpha$, then $\Cap(\Gamma'_I)=|I^c|-2$.
Otherwise,
$$
\Cap(\Gamma'_I)<|I^c|-2.
$$
\end{Lemma}

\begin{proof}
Arguing by contradiction, let $S\subset \Gamma'_I$ be a subset such that 
$$
\bigl|\bigcup_{j\in S}(\Gamma'_I)_j\bigr|-2<\sum_{j\in S}(|(\Gamma'_I)_j|-2)
$$
Let $I_l=\{1,\ldots,l\}$. After renumbering, we can assume that $I=I_k$.
Since $\Gamma'_{I_0}=\Gamma'_\emptyset=\Gamma$, which satisfies~\eqref{CondS}, there exists $l$ such that 
\begin{equation}
\bigl|\bigcup_{j\in S}(\Gamma'_{I_l})_j\bigr|-2\ge\sum_{j\in S}(|(\Gamma'_{I_l})_j|-2)
\cooltag\label{zuzu1}
\end{equation}
but 
\begin{equation}
\bigl|\bigcup_{j\in S}(\Gamma'_{I_{l+1}})_j\bigr|-2<\sum_{j\in S}(|(\Gamma'_{I_{l+1}})_j|-2)
\cooltag\label{zuzu2}
\end{equation}
It follows that some subsets $(\Gamma'_{I_l})_j$ contain $l+1$, and so the LHS in \eqref{zuzu2}
is equal to the LHS in \eqref{zuzu1} minus $1$.
However, the RHS in \eqref{zuzu2}
is equal to the RHS in \eqref{zuzu1} minus the number of subsets $(\Gamma'_{I_l})_j$ that contain $l+1$.
This is a contradiction. This proves that $\Gamma'_I$  satisfies the convexity axiom~\eqref{CondS}.

Clearly, if $|I^c|=2$ or if $I^c\subseteq\Gamma_\alpha$ for some $\alpha$, then $\Cap(\Gamma'_I)=|I^c|-2$.
Assume now that $|I^c|>2$ and $I^c\nsubseteq\Gamma_\alpha$ for any $\alpha$. 
We argue by contradiction: assume that $|I^c|-2=\Cap(\Gamma'_I)$. If $|\Gamma'_I|=0$, then it follows that $|I^c|=2$.
Similarly, if $|\Gamma'_I|=1$, it follows that $I^c\subset\Gamma_\alpha$, for the unique $\alpha$ giving $\Gamma'_I$. 
Hence, we can assume that $|\Gamma'_I|>1$. 

If $|S|\neq0$, $1$, $d$, the same proof as above shows that we have
$$
\bigl|\bigcup_{j\in S}(\Gamma'_I)_j\bigr|-2>\sum_{j\in S}(|(\Gamma'_I)_j|-2)
$$

Hence, if $|\Gamma'_I|\neq d$, then $\Cap(\Gamma'_I)<|I^c|-2$. 

Assume now that $|\Gamma'_I|=d$. We have:
$$|I^c|-2=\Cap(\Gamma'_I)=\sum_{\alpha=1}^d\big(|\Gamma_{\alpha}\cap I^c|-2\big).$$
It follows that
$$|I|=\sum_{\alpha=1}^d\big(|\Gamma_{\alpha}\cap I|\big).$$

It follows that the subsets $\Gamma_{\alpha}\cap I$, for all $\alpha$, are disjoint. This is a contradiction since 
every $i\in I$ belongs to at least two subsets $\Gamma_\alpha$.
\end{proof}

\begin{Lemma}\label{contracted boundary}
The following conditions are equivalent:
\begin{itemize}
\item A boundary divisor $\de_{I\cup\{n+1\}}$ is not contracted by $\pi$.
\item $n-3=\Cap(\Gamma_I)+\Cap(\Gamma'_I)$.
\item $|I^c|=2$ or $I^c\subset\Gamma_{\alpha}$ for some $\alpha$. 
\end{itemize}
\end{Lemma}

\begin{proof}
The equivalence of the first two conditions follows from Lemma~\ref{firstrow}.

If  $|I^c|=2$ or $I^c\subset\Gamma_{\alpha}$ for some $\alpha$ then
$\Cap(\Gamma_I)=|I|-1$ by Lemma~\ref{max size} and 
$\Cap(\Gamma'_I)=|I^c|-2$  by Lemma \ref{capacity I'}. It follows that $\codim\pi\big(\de_{I\cup\{n+1\}}\big)=1$ by Lemma~\ref{firstrow} .

Assume that $|I^c|>2$ and $I^c\nsubseteq\Gamma_\alpha$ for any $\alpha$. Then 
 $\Cap(\Gamma'_I)<|I^c|-2$ by Lemma \ref{capacity I'}. Since
$\Cap(\Gamma_I)\leq|I|-1$, it follows by Lemma~\ref{firstrow}  that $\codim\pi\big(\de_{I\cup\{n+1\}}\big)>1$.
\end{proof}

\section{Irreducibility of $D_\Gamma$ and its Class}\label{Irred and Class}

In this section we define $D_\Gamma$ as the closure 
of $G^2\subset M_{0,n}$ in $\M_{0,n}$. 
We will show in Section \ref{planarity} that this coincides with a stronger definition of $D_{\Gamma}$ given in 
the Introduction. 
Rather  than computing the class of $D_\Gamma$ directly, we (partially) compute the class of its pull-back $\pi_N^*D_\Gamma$, 
where $\pi_N:\M_{0,n+1}\ra\M_{0,n}$ is the forgetful map.
We will use the fact $\pi_N^*D_\Gamma$ is one of the divisors in the exceptional locus of 
the map $\pi$ of \eqref{mappi}
with other possible exceptional divisors all listed in Lemma~\ref{contracted boundary}.

\begin{Notation}
One advantage of $\M_{0,n+1}$ over $\oM_{0,n}$ is that $\Pic\oM_{0,n+1}$ has an equivariant basis 
with respect to permutations of the first $n$ indices.
Let $\Psi:\M_{0,n+1}\ra\PP^{n-2}$ be the Kapranov iterated blow-up of $\PP^{n-2}$ along points $p_1,\ldots, p_n$ 
and proper transforms of subspaces $\langle p_i\rangle_{i\in I}$ for $|I|\leq n-3$. 
Let~$E_I$ be the exceptional divisor over this subspace. 
Recall that $\Pic\oM_{0,n+1}$ is freely generated by $H:=\Psi^*\cO(1)$ and by the classes $E_I$.
\end{Notation}

We denote as usual by $v_i$  the valence of $i\in N$.

\begin{Theorem} \label{class}\label{first coeff}\label{Irreducibility}
Let $\Gamma=\{\Gamma_1,\ldots,\Gamma_d\}$ be 
an irreducible hypertree on $N$.
Then $D_\Gamma$ is non-empty, irreducible, and $v(D_\Gamma)=W^2\subset\Pic^1$ has codimension~$3$.
We have
$$\pi_N^*D_\Gamma\sim(d-1)H-\sum_{I\subset N\atop 1\le |I|\le n-3}m_IE_I,$$
where 
\begin{equation}
m_I\ge|I|-1+\left|\{\Gamma_\alpha\,|\,\Gamma_\alpha\subset I^c\}\right|-\Cap(\Gamma_I),
\cooltag\label{coefficientm_I}
\end{equation} 
\begin{equation}
m_{\{i\}}=d-v_i,
\cooltag\label{hikitty}
\end{equation}
\begin{equation}
m_{N\setminus\Gamma_\alpha}=1,
\cooltag\label{kitty1}
\end{equation}
\begin{equation}
m_{\Gamma_\alpha}=d+|\Gamma_\alpha|-\sum_{i\in \Gamma_\alpha}v_i.
\cooltag\label{himama}
\end{equation}
If $I$ is properly contained in $\Gamma_\alpha$ then
\begin{equation}
m_{N\setminus I}=0,
\cooltag\label{kitty2}
\end{equation}
\begin{equation}
m_I=d+|I|-1-\sum_{i\in I}v_i.
\cooltag\label{hipapa}
\end{equation}
\end{Theorem}

\begin{proof}
By Theorem \ref{size/explicit}, the map $\pi$ of \eqref{mappi}
is a birational morphism. By~Theorem~\ref{coolcondition}, its 
exceptional locus consists of $\cE:=\pi_N^*D_\Gamma$
and the boundary divisors $\de_{I\cup\{n+1\}}$ contracted by~$\pi$
(where $I\subset N$, $1\le |I|\le n-2$). 

\begin{Lemma}\label{irreducible}
$D_\Gamma$ is non-empty and irreducible.
\end{Lemma}

\begin{proof}
It suffices to show that $\pi_N^*D_\Gamma$ is non-empty and irreducible.
We compare ranks of the Neron--Severi groups and use the fact that 
$$\rho(\oM_{0,n+1})-\rho\left(\prod_{\Gamma_{\alpha}}\M_{0,\Gamma_{\alpha}\cup\{n+1\}}\right)$$
is equal to the number of irreducible components in $\Exc(\pi)$.
We have
$$\rho(\oM_{0,n+1})=2^n-1-{n(n+1)\over 2}$$
and
$$\rho\left(\prod_{\Gamma_{\alpha}}\M_{0,\Gamma_{\alpha}\cup\{n+1\}}\right)
=\sum_{\alpha=1}^d \left(2^{|\Gamma_\alpha|}-1-{|\Gamma_\alpha|(|\Gamma_\alpha|+1)\over 2}\right)
$$
The total number of boundary divisors of $\oM_{0,n+1}$ is $2^n-n-2$. By Lemma~\ref{contracted boundary}, the number of 
boundary components \emph{not} contracted by $\pi$ is
$$\frac{n(n-1)}{2}+\sum_{\alpha=1}^d \left(2^{|\Gamma_\alpha|}-1-|\Gamma_\alpha|-{|\Gamma_\alpha|(|\Gamma_\alpha|-1)\over 2}\right)$$
It follows after some simple manipulations that the number of irreducible components of $\cE$
is exactly one.
\end{proof}

\begin{Review}\label{canonical computations}
Next we compare the canonical classes. We have
\begin{equation}
K_{\M_{0,n+1}}-\pi^*K_{\M_{\Gamma\cup\{n+1\}}}=c\cE+\sum_{\de_{I\cup\{n+1\}}\in\Exc(\pi)}a_I\de_{I\cup\{n+1\}}
=c\cE+\sum_{I\subset N\atop 1\le |I|\le n-3} a_IE_I,\cooltag\label{discrrr}
\end{equation}
for some positive integers $a_I$ and $c$, see \cite[p. 53]{KoM}.
Here we use the fact that if $|I|=n-2$ then $\de_{I\cup\{n+1\}}$
is not an exceptional divisor in the Kapranov model, but a proper transform of the hyperplane in $\bP^{n-2}$
that passes through all $p_i$, $i\in I$.
These divisors are not in $\Exc(\pi)$ by Lemma~\ref{contracted boundary}.
\end{Review}

\begin{Review}
We use the following basic property of discrepancies:
$$c\ge\codim \pi(\cE)-1\quad\hbox{\rm and}\quad
a_I\ge\codim\pi\big(\de_{I\cup\{n+1\}}\big)-1.$$
By Lemma~\ref{firstrow}, it follows that
\begin{equation}
a_I\ge n-3-\Cap(\Gamma_I)-\Cap(\Gamma'_I).
\cooltag\label{dududu}
\end{equation}
\end{Review}

\begin{Review}\label{canonical computations}
Next we compute the canonical classes. Standard calculations give
$$K_{\M_{0,n+1}}=-(n-1)H+\sum_I \big(n-2-|I|\big)E_I$$
and 
\begin{gather*}
\pi^*_{\Gamma_\alpha\cup\{n+1\}}K_{\M_{\Gamma_{\alpha}\cup\{n+1\}}}=-(|\Gamma_{\alpha}|-1)\left(H-\sum_{I\cap\Gamma_{\alpha}=\eset}E_I\right)+\\
+\sum_{I'\subset \Gamma_{\alpha}\atop 1\leq|I'|\leq|\Gamma_{\alpha}|-3}\big(|\Gamma_{\alpha}|-2-|I'|\big)\sum_{I''\subset N\setminus\Gamma_{\alpha}} E_{I'\cup I''}.
\end{gather*}
Combining these formulas together gives 
$$m_I\ge|I|-1-\Cap(\Gamma_I)-\Cap(\Gamma'_I)+$$
$$
+\sum_{I\cap\Gamma_{\alpha}=\eset}(|\Gamma_{\alpha}|-1)+
\sum_{1\leq|\Gamma_{\alpha}\cap I|\leq|\Gamma_{\alpha}|-3}
\big(|\Gamma_{\alpha}|-2-|\Gamma_{\alpha}\cap I|\big).
$$
This formula along with \eqref{bratstvo} imply formula \eqref{coefficientm_I}.
\end{Review}


\begin{Lemma}\label{calculationsss1}\label{calculationsssss}
Formulas  \eqref{hikitty}--\eqref{hipapa} hold.
\end{Lemma}

\begin{proof}
Using \eqref{coefficientm_I} it is easy to see that
the LHS of any of these formulas is greater than or equal to the RHS.

Since boundary divisors $\de_{N\cup\{n+1\}\setminus\Gamma_\alpha}$
and $\de_{N\cup\{n+1\}\setminus I}$ are not in the exceptional locus of $\pi$,
formulas \eqref{kitty1} and \eqref{kitty2} follow from \eqref{discrrr} and from the calculations of the canonical classes above.

By projection formula, $\pi_N^*D_\Gamma\cdot C=0$ for any curve $C$ in the fiber of $\pi_N$.
Let $C$ be a general fiber. Then $\Psi(C)$ is a rational normal curve in $\bP^{n-2}$,
and therefore $H\cdot C=n-2$. We have $E_i\cdot C=\delta_{i,n+1}\cdot C=1$ and any other boundary
divisor $E_I$ intersects $C$ trivially.
It follows that 
$$0=\pi_N^*D_\Gamma\cdot C=(n-2)(d-1)-\sum_{i=1}^nm_i\le (n-2)(d-1)-\sum_{i=1}^n(d-v_i)=$$
$$(n-2)(d-1)-nd+\sum_{i=1}^nv_i=(n-2)(d-1)-nd+(n-2)+2d=0.$$
It follows that $m_i=d-v_i$ for any $i$.

Now let $C$ be the curve in the fiber of $\pi_N$ over a general point in $\de_{\Gamma_\alpha}$
such that the $(n+1)$-st marked point moves along the component with points marked by $I\subset\Gamma_\alpha$.
Then $H\cdot C=|I|-1$, $\de_{i,n+1}\cdot C=1$ if $i\in I$ and~$0$ otherwise,
$\de_{I\cup\{n+1\}}\cdot C=-1$, $\de_{N\cup\{n+1\}\setminus I}\cdot C=1$, and other boundary divisors
intersect $C$ trivially. Since we already know that $m_i=d-v_i$ by the above, 
and that $m_{N\setminus I}=1$ (if $I=\Gamma_\alpha$) and $0$ otherwise by \eqref{kitty1} and \eqref{kitty2},
a simple calculation gives $m_I$.
\end{proof}

Finally, we claim that 
$$c=1,\quad \codim\pi(\cE)=2,\quad \hbox{\rm and}\quad \codim v(D_\Gamma)=3.$$
Indeed, $c$ obviously divides all coefficients $m_I$ 
but some of them are equal to $1$ by \eqref{kitty1}. So $c=1$.
Since $\cE$ is exceptional and $c\ge\codim\pi(\cE)-1$, we have $\codim\pi(\cE)=2$.
It follows by Theorem~\ref{ID} that the map $\cE\to\pi(\cE)$ is generically a $\bP^1$-bundle,
i.e.~$W_2\ne W_3$. Since $G^2\to W^2$ has $2$-dimensional fibers
outside of $W^3$, we have the formula  
$\codim v(D_\Gamma)=3$.
\end{proof}

The reader is perhaps disappointed that we do not give a closed formula for the class of 
a hypertree divisor $D_\Gamma$.
The difficulty of computing this class stems from the fact
that $\pi$ has (exponentially) many exceptional boundary divisors
and the discrepancy of a boundary divisor $\delta_{I}\subset\oM_{0,n+1}$ for the map $\pi$ \eqref{mappi}
is not always equal to $\codim\pi(\delta_{I})-1$.
However, there is one case when they are equal, namely
when $\codim\pi(\delta_{I})=2$. 
This happens quite often: see for example Lemma~\ref{qqqcheck}, which is used in Theorem~\ref{amazing}
to recover a hypertree from its class.

\begin{Definition}\label{wheel}
A triple $\{i,j,k\}\subset N$ is called a {\em wheel} of an irreducible hypertree $\Gamma$, if it is not contained in 
any hyperedge, but there are hyperedges $\Gamma_{\alpha_1}$, $\Gamma_{\alpha_2}$, $\Gamma_{\alpha_3}$ of $\Gamma$
such that $\{i,j\}\subset\Gamma_{\alpha_1}$, $\{j,k\}\subset\Gamma_{\alpha_2}$, $\{i,k\}\subset\Gamma_{\alpha_3}$.
\end{Definition}


\begin{Lemma}\label{nextfix}
Suppose $\Gamma$ contains only triples, with $\{i,j,k\}$ not one of them and not a wheel, with the property that
$$\Cap(\Gamma_{N\setminus\{i,j,k\}})=n-4$$
(which is equivalent to $\codim\pi(\delta_{i,j,k})=2$).
Then we have equality in \eqref{coefficientm_I}.
\end{Lemma}

\begin{proof}
We know that $\codim\pi(\delta_{i,j,k})=2$ and we are claiming that the discrepancy of $\pi$ at $\delta_{i,j,k}$ is equal to $1$.
It will be enough to show that that no other divisor of $\oM_{0,n+1}$ has the same image as $\delta_{ijk}$ under $\pi$.
Indeed, then we can cut by hypersurfaces in a very ample linear system on the target
of $\pi$ to reduce the discrepancy calculation to the case of a birational
morphism of smooth surfaces with a unique exceptional divisor over a point, in which case the discrepancy is equal to $1$
by standard factorization results for birational morphisms of smooth surfaces \cite[5.3]{Ha}.
 
We write $a\lra b$ if vertices $a,b\in N$ belong to some $\Gamma_\alpha$.
Up to symmetries, there are three possible cases.

\begin{itemize}
\item[{\bf (X)}] $i\lra j$ and $j\lra k$.
\item[{\bf (Y)}] $i\lra j$ but $i\not\lra k$ and $j\not\lra k$.
\item[{\bf (Z)}] $i\not\lra j$, $j\not\lra k$, and $i\not\lra k$.
\end{itemize}

Notice that $\pi(\delta_{i,j,k})$ belongs to the boundary $\oM_{\Gamma\cup\{n+1\}}\setminus M_{\Gamma\cup\{n+1\}}$ in cases (X) and (Y).
So in these cases $\pi(\delta_{i,j,k})$ can not be equal to $\pi(\cE)$, where $\cE=\pi_N^*(D_\Gamma)$
is the only exceptional divisor of $\pi$ intersecting the interior $M_{0,n+1}$. 
In case (Z), the image of  $\delta_{i,j,k}$ intersects the interior $M_\Gamma$, but we claim that
in this case $\pi(\delta_{i,j,k})\ne\pi(\cE)$ as well.
Arguing by contradiction, suppose $\pi(\delta_{i,j,k})=\pi(\cE)$. 
Notice that the rational map $v:\,\oM_{0,n}\dra\bG_m^{n-3}$ of Theorem~\ref {ID}
is defined at the generic point of $\delta_{ijk}\subset\oM_{0,n}$ in case Z:
just take a map $\Sigma\to\bP^1$ that collapses points $i,j,k$ to the same point
and pull-back $\cO_{\bP^1}(1)$ (a similar analysis will be given below, see Lemma \ref{moonwalk2}).
Since $W^2\subset \bG_m^{n-3}$ has codimension $3$, 
a generic line bundle in $W^2$ has $h^0=3$ (see Thm. \ref{class}).
Passing to an open subset in $\oM_{0,n}$
containing the generic point of $\delta_{ijk}$, we have that $v(\delta_{ijk})=W^2$.
But this implies that, in a planar realization that corresponds to a generic line bundle in $W^2$,
points $i,j,k$ are collinear.
We will show in Theorem~\ref{PlanarTh} that this is not the case.

So it remains to check the statement for boundary divisors only, i.e.,~to 
show that if $\pi(\delta_{ijk})=\pi(\delta_I)$ and $n+1\not\in I$ then $I=\{i,j,k\}$. 
Let $\Gamma'\subset\Gamma$ be a subset of all triples other 
than the triple containing $\{i,j\}$ (in cases (X) and (Y)) and the triple containing $\{j,k\}$ (in case (X)).
Consider the morphism
$$\pi':\oM_{0,n+1}\to\oM_{\Gamma'\cup\{n+1\}}:=\prod_{\Gamma_\alpha\in\Gamma'}\oM_{\Gamma_\alpha\cup\{n+1\}}.$$
Then we have $\pi'(\delta_{ijk})=\pi'(\delta_I)$ has dimension $n-4$ and intersects the interior $M_{\Gamma'}$.
So $I$ has the following properties: 
\begin{itemize}
\item $i,j,k\in I$ in case (X);  $i,j\in I$ in case (Y).
\item $I$ contains $s$ whole triples from $\Gamma'$ and $q$ ``separate'' points in $N$ not related by $\lra$ to any other point in $I$.
\end{itemize}

From $\delta_I\simeq\oM_{0,|I|+1}\times\oM_{0,(n+1)-|I|+1}$
we have the following easy estimate
$$n-4=\dim\pi'(\delta_I)\le s+\dim\oM_{0,(n+1)-|I|+1}=n-|I|+s-1,$$
and therefore 
\begin{equation}\cooltag\label{moonwalk}
|I|\leq s+3.
\end{equation}

Consider the case (X). Since $$n-4=\dim\pi(\de_{i,j,k})=\Cap((\Gamma')_{N\setminus\{i,j,k\}})$$ it follows that
for any subset $T$ ($|T|\geq2$) of triples from $\Gamma'$, 
$$|\bigcup_{\alpha\in T}\Gamma_\alpha|\geq |T|+4,$$
if $i,j,k\in \bigcup_{\alpha\in T}\Gamma_{\alpha}$. 

Assume $s\geq2$. If $i,j,k\in \bigcup_{\Gamma_\alpha\subset I}\Gamma_{\alpha}$ then
we have:
$$|I|-q\geq|\bigcup_{\Gamma_\alpha\subset I}\Gamma_\alpha|\geq s+4,$$
which contradicts (\ref{moonwalk}).

If one of $i,j,k$, say $i$, is not in $\bigcup_{\Gamma_\alpha\subset I}\Gamma_{\alpha}$ then 
$$|I\setminus\{i\}|-q\geq|\bigcup_{\Gamma_\alpha\subset I}\Gamma_\alpha|\geq s+3,$$
which again contradicts (\ref{moonwalk}).

Assume $s=1$. Let $\Gamma_1$ be the unique triple in $\Gamma'$ contained in $I$. We have 
$|I|=q+3$ and by (\ref{moonwalk}) $|I|\leq 4$. It follows that $q=0$ or $1$. Since $i,j,k\in I$ it follows 
that at least two of the indices $i,j,k$  are in $\Gamma_1$, which is a contradiction.

Consider now the cases (Y),  (Z).
We have a usual diagram of morphisms
$$\begin{CD}
M_{0,n+1} @>{\pi'}>> M_{\Gamma'\cup\{n+1\}}\\
@V{\pi_N}VV @VV{D\mapsto\cO(D)}V \\
M_{0,n} @>v>> \Pic^{\uone}(\Sigma').
\end{CD}$$

\begin{Lemma}\label{moonwalk2}
The morphism $v$ can be extended to generic points of $\delta_{ijk}$ and~$\delta_I$ as follows:
Let $C$ be a fiber of the universal family over a general point of $\delta_{ijk}$ (resp., $\delta_I$).
On one component $C_1$ we have points $i,j,k$ (resp.~$I$) and the attaching point $p$, while on the other component $C_2$
we have points $N\setminus\{i,j,k\}$ (resp.~$N\setminus I$) and the attaching point $q$.
This gives a morphism $f:\,\Sigma'\to\bP^1$ obtained by sending points in 
$N\setminus\{i,j,k\}$ (resp.~$N\setminus I$) to the corresponding points of the second component of $C$
and by sending points in $i,j,k$ (resp.~$I$) to the point $q$. Consider the line bundle $L=f^*\cO_{\bP^1}(1)$.
The line bundle $L$ has degree $0$ on the components $\Gamma_\alpha\subset I$. Each such component $\Gamma_\alpha$
can be identified with $C_1$, thus we can twist $L$ by 
$\cO_{\Sigma'}(p)$, which gives a line bundle in $\tilde L\in\Pic^{\uone}(\Sigma)$. 
\end{Lemma}

\begin{proof}
Take $|\Gamma'|$ copies of of the universal family $\M_{0,n+1}$ over $\M_{0,n}$, indexed by triples in $\Gamma'$. Let $\bf X$ be the push-out 
of these families, glued along sections, as prescribed by $\Gamma'$. (The fiber of $\bf X$ over a point in $M_{0,n}$ is $\Sigma$.) 
Let $U$ be the open in $\M_{0,n}$ which is the union of $M_{0,n}$  and $\de_{i,j,k}$ (resp. $\de_I$), not containing any other boundary strata.
Let ${\bf X}^0$ be the preimage of $U$ in ${\bf X}$.

There are maps $u: {\bf X}^0\ra U\times\Sigma$ (given by stabilization) and  $f:{\bf X}^0\ra U\times\PP^1$ (obtained by contracting the points in $I$).
Let $M=f^*\O_{\PP^1}(1)$ and $L=u_*(M)$. It follows from a local calculation in \cite{F} that 
$L$ is invertible and for $m\in U$ we have $L_m\in \Pic^{\uone}(\Sigma)$ satisfying the Lemma. 
\end{proof}

After shrinking $\oM_{0,n}$ to an open subset containing generic points of $\delta_{ijk}$ and $\delta_I$, 
this gives 
$$v(\delta_{ijk})=v(\delta_I).$$

In case (Z), $v(\delta_{ijk})\not\subset W^2$,
i.e. a general line bundle $L$ in $v(\delta_{ijk})$ has $h^0=2$ and it induces a map $f:\Sigma\ra\PP^1$ that collapses only the points
$i, j, k$ to the point $q$. Since $v(\delta_{ijk})=v(\delta_I)$, the map $f$ collapses the points in $I$ to the point $q$. It follows that $I=\{i,j,k\}$. 
This finishes case (Z).

In case (Y), since $\pi'(\delta_{ijk})$ has codimension~$1$, the map
$\pi'|_{\delta_{ijk}}$ generically has $1$-dimensional fibers, this implies that
$v(\delta_{ijk})\not\subset W^3$,
i.e. a general line bundle in $v(\delta_{ijk})$ has $h^0=3$ and gives an admissible map $f:\Sigma'\to\bP^2$
such that points $i,j,k$ belong to a line $H\subset\bP^2$. 
The corresponding point of $\oM_{0,n}$ is obtained by projecting $\Sigma'$ from a general point of $H$.
Note that the points  in $N\setminus\{i,j,k\}$ will be mapped to distinct points via this projection, hence no points in 
$N\setminus\{i,j,k\}$ will lie on the line $H$.

The same analysis for $\delta_I$ combined with the fact that $v(\delta_{ijk})=v(\delta_I)$ 
shows that via the map $f$ the points in $I$ are collinear. Since in Case (Y) $i,j\in I$, it follows that
the points in $I$ lie on $H$. This implies $I=\{i,j,k\}$.
\end{proof}

Finally, we analyze hypertrees that are not irreducible. Recall that we denote by $D_{\Gamma}$ the closure of $G^2(\Gamma)$ in $\M_{0,n}$.
 
\begin{Lemma}\label{weakness}
If $\Gamma$ is not irreducible and $D_\Gamma\neq\emptyset$, then for every irreducible $D$ component of $D_\Gamma$
there exists an irreducible hypertree $\Gamma'$
on  a subset $N'\subset N$ such that 
$$D=\pi^{-1} (D_{\Gamma'}),$$
where $\pi:\,\oM_{0,n}\to\oM_{0,N'}$ is a forgetful map.
\end{Lemma}

\begin{proof}
If $\Gamma'$ is an irreducible hypertree, then $D_{\Gamma'}$ is an irreducible divisor in $\oM_{0,N'}$
intersecting the interior. Since $\pi$ is flat with irreducible fibers along points in $M_{0,N'}$, $\pi^{-1}(D_\Gamma')$ is irreducible.
Hence, it is enough to prove $D\subset\pi^{-1} (D_{\Gamma'})$. Note, since $D_{\Gamma}\neq\emptyset$, we have $G^2(\Gamma)\neq\emptyset$.

We argue by induction on $d$. Let $S\subset\{1,\ldots,d\}$ be a subset such that \eqref{CondS} is an equality.
We may assume that $S$ is minimal with this property. Let $d'=|S|$, let $\Gamma'$ be a collection of $\Gamma_i$ for $i\in S$.
Let $N'=\cup_{i\in S}\Gamma_i$. Then $\Gamma'$ is almost a hypertree: all axioms are satisfied except possibly for the second axiom:
it could happen that there exists an index $i\in N'$ that belongs to only one subset $\Gamma'_j$.
In this case we can remove $i$ from $N'$ (and remove $\Gamma'_j$ from $\Gamma'$ if $|\Gamma'_j|=3$).
Continuing in this fashion, we get a subset $N'\subset N$ and a hypertree $\Gamma'$ on it. By minimality of $S$, 
$\Gamma'$ is irreducible.

Let $D$ be a component of $D_{\Gamma}$ (i.e., the closure of a component of $G^2(\Gamma)$). If 
$D\subseteq\pi^{-1} (D_{\Gamma'})$, then we are done. Assume now that 
$D$ is not contained in $\pi^{-1} (D_{\Gamma'})$. Then a dense open in $D$ is disjoint 
from $\pi^{-1} (G^2(\Gamma'))$; hence, a general element in $D\cap G^2(\Gamma)$ is obtained via projection from a map
$\Sigma\ra\PP^r$ ($r\geq2$) that maps $\Sigma'$ to a line. 
Let $$\Gamma''=(\Gamma\setminus\Gamma')\cup\{\Gamma_0\},\quad\hbox{where }
\Gamma_0=\bigcup_{\Gamma_i\in\Gamma'}\Gamma_i.$$

If there exists an index $i\in N$ that belongs to only one subset $\Gamma''_i$, we remove it. Let $N''$ be the remaining 
set if indices. It is easy to check that $\Gamma''$ is a hypertree on $N''$. Moreover, our assumptions imply that 
$D\subseteq\pi^{-1} (D_{\Gamma''})$. By our induction assumption, any component of $D_\Gamma''$ is the preimage by a 
forgetful map of some $D_{\tilde{\Gamma}}$ for some irreducible hypertree  $\tilde{\Gamma}$. 
\end{proof}

\section{Compactified Jacobians of Hypertree Curves}\label{smaps}

Our goal in this section is to prove Theorem~\ref{jacobianmain}: 
if $\Gamma$ is an irreducible hypertree then the hypertree divisor $D_\Gamma\subset\oM_{0,n}$
is contracted by a contracting birational map to the compactified Jacobian.

We start by considering any hypertree, not necessarily irreducible.
We extend the universal stable hypergraph curve $\Sigma^s/M_\Gamma$ to a curve over
$\oM_\Gamma$ in an obvious way. Let $\Sigma^s_\Omega$ be one of the geometric fibers.

\begin{Definition}
A coherent sheaf on $\Sigma^s_\Omega$ is called {\em Gieseker semi-stable} (resp. {\em Gieseker stable})
if it is torsion-free, has rank~$1$ at generic points of $\Sigma^s_\Omega$,  
and is semi-stable (resp. stable) with respect to the canonical polarization~$\omega_{\Sigma^s_\Omega}$.
\end{Definition}

The {\em compactified Jacobian} \cite{OS,Ca} $\oPic/\oM_\Gamma$
parametrizes  gr-equivalence classes of Gieseker semi-stable sheaves.
By \cite{Si}, it is functorial: consider the functor 
$$\ocPic:\, Schemes\to Sets$$
that assigns to a scheme $S$ the set of coherent sheaves on $\Sigma^s_S$ flat over $S$
and such that its restriction to any geometric fiber $\Sigma^s_\Omega$ is Gieseker semi-stable.
Then there exists a natural transformation 
$\ocPic\to h_{\oPic}$ 
which has the universal property: for any scheme $T$, any natural 
transformation     
$\ocPic\to h_T$ factors through a unique morphism $\oPic\to T$.

Over each geometric point of $\oM_\Gamma$,  
$\oPic$ is a stable toric variety of $\Pic^{\uzero}(\Sigma_\Omega)$ and its normalization
is a disjoint union of toric varieties.

\begin{Proposition}\label{stabledegree}\label{GiesekerThm}
A pull-back of an invertible sheaf in $\Pic^{\uone}(\Sigma_\Omega)$ is Gieseker stable on~$\Sigma_\Omega^s$.
\end{Proposition}

\begin{proof}
Let $X=\Sigma^s$ be a stable hypertree curve.
We call an irreducible component of $X$ black if it is a proper transform
of a component of~$\Sigma$. Otherwise we call it white.
It is well-known that slope stability on reducible curves
reduces to the following {\em Gieseker's basic inequality}.
For any proper subcurve $Y\subset X$, we have
\begin{equation}\label{Gieseker_Ineq}
\left|b(Y)-b(X){m(Y)\over m(X)}\right|<{1\over2}\#Y.
\cooltag\end{equation}
Here, 
$$b(S)=\deg L|_S,\quad m(S)=\deg\omega_{\Sigma^s}|_S,\quad\hbox{\rm and}\quad 
\#Y:=|Y\cap\overline{ X\setminus Y}|.$$
In our case, $b(S)$  is just  the number of black components in $S$, and we have
$$m(X)=2g-2=2d-4.$$
We denote $m:=m(Y)$, $b=b(Y)$, and have to show that 
\begin{equation}\label{Gieseker_Ineq_2}
|(2d-4)b-dm|<(d-2)\#Y.
\cooltag\end{equation}

\begin{Review} 
It is easy to see that the complementary subcurve $Y^c:=\overline{ X\setminus Y}$
satisfies \eqref{Gieseker_Ineq} if and only if $Y$ does.
Hence, by interchanging $Y$ with $Y^c$, we can assume that 
\begin{equation}\label{condition_yyy}
dm-(2d-4)b\geq0
\cooltag\end{equation}
and try to show that
\begin{equation}\label{Gieseker_Ineq_3}
dm-(2d-4)b-(d-2)\#Y<0
\cooltag\end{equation}
\end{Review}

\begin{Review}
Consider a white component $w_1$ of $\Sigma^s$ which is not in $Y$ 
but such that at least one adjacent black component is in $Y$. 
Enumerate the black components in $Y$ 
intersecting $w_1$ as $b_1, b_2, ..., b_i$, and the rest as $b_{i+1}, ..., b_k$, with $1\leq i\leq k$ (and $k\geq 3$).  
We claim that adding $w_1$ to $Y$ does not decrease the left hand side in \eqref{Gieseker_Ineq_3}
and increases the left hand side in \eqref{condition_yyy}.
We~only need to show that $dm-(d-2)\#Y$ increases. Adding $w_1$ to $Y$ increases 
$m$ by $k-2$. If the original value of $\#Y$ is $x+i$, where $i$ is the contribution from $w_1$ intersecting $b_1$ through $b_i$, then the value after adding $w_1$ to $Y$ is $x+k-i$. Hence, $\#Y$ increases by $k-2i$.
Then the difference of values of the left hand side is
$$d(k-2)-(d-2)(k-2i)=(d-2)i+k-d\ge (d-2)+3-d>0.$$
Hence we can assume that all white lines hit by a black component in $Y$ are also in $Y$:
by showing \eqref{Gieseker_Ineq_3} in this situation, we show \eqref{Gieseker_Ineq_3} in general.
\end{Review}

\begin{Review}
Let $P_i$ be the number of singular points of $\Sigma$ of valence~$i$.
Then 
$$\sum P_i=n\quad\hbox{\rm and}\quad\sum iP_i=2d+n-2.$$
This is because $\sum iP_i$ is the total number of times a singular point is hit by a
component in $\Sigma$. This is equal to $\sum |\Gamma_\alpha|$, which by the 
normalization axiom equals $2d+n-2$. So we have
\begin{equation}\label{Gieseker_Ineq_6}
\sum(i-1)P_i=2d+n-2-n=2d-2.
\cooltag\end{equation}
\end{Review}

Let $p_i$ be the number of singular points of $\Sigma$ 
of valence $i$ hit by the image of $Y$.
Then \eqref{Gieseker_Ineq_6} implies that
\begin{equation}\label{Gieseker_Ineq_5}
\sum_i(i-1)p_i\leq2d-2.
\cooltag\end{equation}
with strict inequality if $Y$ does not cover all the points in $N$.

Let $b_i$ be the number of black components in $Y$ with $i$ singular points.
By the convexity axiom, we have $\sum_ip_i\geq \sum_i(i-2)b_i+2$. If $Y$ covers all the points
in $N$ then we claim that this inequality is strict: otherwise, as $Y$ is a proper subcurve of $X$,
the convexity axiom would be violated when we consider the components of $Y$ and one extra 
component that is not in $Y$. This inequality together with \eqref{Gieseker_Ineq_5} (at least one being strict) 
implies that 
\begin{equation}\label{trururururu}
\sum_i (i-d)p_i+(d-1)\sum_i(i-2)b_i<0.
\cooltag\end{equation}

Let $l'_i$ be the number of isolated white components with $i$ singularities 
(i.e., those not hit by any black components in $Y$). Since we obviously have
$\sum(i-d)l'_i\le 0$, \eqref{trururururu} implies that
\begin{equation}\label{Gieseker_Ineq_4}
\sum_i (i-d)p_i+\sum_i(i-d)l'_i+(d-1)\sum_i(i-2)b_i<0.
\cooltag\end{equation}

We claim that this inequality is equivalent to \eqref{Gieseker_Ineq_3}.
Let $l_i$ be the number of white components in $Y$ with $i$ singular points.
Then 
$$m=\sum(i-2)b_i+\sum(i-2)l_i=\sum(i-2)b_i+\sum(i-2)l_i'+\sum(i-2)p_i,$$
since $\sum(i-2)l_i'$ is the contribution to $\sum(i-2)l_i$ by isolated white components in $Y$ and 
$\sum(i-2)p_i$ is the contribution by white components hit by black components in $Y$,
which we can assume are all in $Y$.
We also have
\begin{equation}\label{PndY_Eq}
\#Y=\sum il'_i+\sum ip_i-\sum ib_i,
\cooltag\end{equation}
where 
$\sum il'_i$ is the contribution to $\#Y$ by isolated white components in $Y$,
$\sum ip_i$ is the total number of times a point in the image of $Y$ in $\Sigma$
is hit by a black component (not necessarily in $Y$) and $\sum ib_i$ is the total number of times a black component
 in $Y$ hits one of these points, so their difference is the contribution to $\#Y$ by everything except isolated white components.

So we have
$$dm-(2d-4)b-(d-2)\#Y=$$
$$
2\sum(i-d)p_i+2\sum(i-d)l_i'+2(d-1)\sum (i-2)b_i<0$$
by \eqref{Gieseker_Ineq_4}.
\end{proof}

\begin{Corollary}
$\Pic^{\uone}\subset\oPic$.
\end{Corollary}

\begin{Review}\label{CaporasoStuff}
Let $\oPicone$ be the normalization of the closure of $\Pic^{\uone}$ in $\oPic$.
It compactifies the $\bG_m^g$-torsor $\Pic^{\uone}$ over $M_\Gamma$ by adding boundary divisors
of two sorts, {\em vertical} and {\em horizontal}.
Vertical boundary divisors are divisors over the boundary of $\oM_\Gamma$.
The boundary divisors of $\oM_\Gamma$ are parametrized by subsets
$I\subset \Gamma_\alpha$ with $|I|,|\Gamma_\alpha\setminus I|>1$.
The corresponding hypertree curve $\Sigma'$ generically has $d+1$ irreducible component, 
with the $\alpha$'s component broken into a nodal curve with two components, 
$C_\alpha^1$ (with singular points indexed by $I$) and $C_\alpha^2$ 
(with singular points indexed by $\Gamma^\alpha\setminus I$).
There could be two corresponding vertical boundary divisors. Generically
they parametrize line bundles on $\Sigma'$ that have degree $1$ on $C_\alpha^1$ and degree $0$
on $C_\alpha^2$ (resp.~degree $0$ on $C_\alpha^1$ and degree $1$
on $C_\alpha^2$) and degree $1$ on the remaining components.
Notice that apriori it is not clear that these loci are non-empty divisors:
one has to check that these line bundles are Gieseker semi-stable.

Horizontal boundary divisors are toric (over a geometric point of $M_\Gamma$) and can be described as follows.
Choose a node in $\Sigma^s$ and let $\hat\Sigma$ be a curve obtained from $\Sigma^s$
by inserting a strictly semistable $\bP^1$ at the node.
Start with the multidegree $\uone$ and 
choose a multidegree $\hat d$ on $\hat\Sigma$ such that the degree on
the extra $\bP^1$ is $1$ and the degree on one of the neighboring black components 
is lowered from $1$ to $0$ (lowering the degree on a white component would lead to an unstable sheaf).
The corresponding Gieseker semi-stable sheaves on $\Sigma^s$ are push-forwards of invertible sheaves $\hat F$ on $\hat\Sigma$
of a given Gieseker semi-stable multidegree with respect to the stabilization morphism $\hat\Sigma\to\Sigma^s$.
Note that this creates a sheaf which is not invertible at the node.
An easy count shows that potentially this gives as many as $2d-2+n$ horizontal divisors.
\end{Review}

\begin{Lemma}
If $\Gamma$ is an irreducible hypertree then $\oPicone$ has a maximal possible number of horizontal ($2d-2+n$)
and vertical boundary divisors.
\end{Lemma}

\begin{proof}
This is a numerical question: one has to check that the corresponding multidegrees are Gieseker-stable.
The proof is parallel to the proof of Proposition~\ref{GiesekerThm}: a stronger (by~$1$)
inequality satisfied by an irreducible hypertree compensates for the difference (by $1$) 
in the multidegree. We omit this calculation.
\end{proof}

\begin{Example}
The papers \cite{OS} and \cite{Al} contain a recipe for presenting
the polytope of $\oPicone$ as a slice of the hypercube. We won't go into the details here but let us give our favorite example.
Let $\Sigma$ be the Keel--Vermeire curve with $4$ components
indexed by $\{1,2,3,4\}$.
Then the polytope is the rhombic dodecahedron of  Fig.~\ref{garnet}.
\begin{figure}[htbp]
\includegraphics[width=2in]{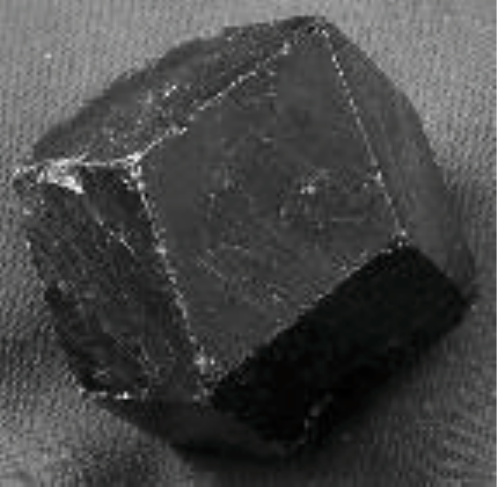}
\caption{\small Compactified Jacobian of the Keel--Vermeire curve.
}\label{garnet}
\end{figure}
The normals to its faces are given by roots $\alpha_{ij}=\{e_i-e_j\}$ of the root system $A_3$, where $i,j\in\{1,2,3,4\}$, 
$i\ne j$.
To describe a pure sheaf from the corresponding toric  codimension~$1$ stratum,
consider a quasi-stable curve $\Sigma_{ij}$
obtained by inserting a $\bP^1$ at the node of $\Sigma$ where $i$-th and $j$-th 
components intersect. Now just pushforward to $\Sigma$ an invertible sheaf
that has degree $1$ on this $\bP^1$ and at any component of $\Sigma_{ij}$ other than the proper transform
of the $i$-th component of $\Sigma$ (where the degree is $0$). 
\end{Example}

Now we can prove Theorem~\ref{jacobianmain}.

\begin{proof}
Our proof is parallel to the proof of irreducibility of $D_\Gamma$ in Lemma~\ref{irreducible}.
Consider the birational map $v:\,\oM_{0,n}\dra\oPicone$. Note that $\oPicone$ is in general not
$\bQ$-factorial. The map $v$ contracts only one divisor intersecting $M_{0,n}$, namely $D_\Gamma$.
The map $v$ is necessarily  contracting if
$$\rho(\oM_{0,n})-\rho'(\oPicone)=1+|\{\hbox{\rm boundary divisors contracted by $v$}\}|.$$
(Here $\rho'(X)$ denotes the rank of the class group $\Cl(X)$). 
Computation of this number shows that it suffices to check that the following boundary divisors
are {\em not} contracted by $v$: 
\begin{itemize}
\item $\delta_{ij}$ for $\{i,j\}\not\subset\Gamma_\alpha$;
\item $\delta_{I}$ for $I\subset \Gamma_\beta$.
\end{itemize}
We use the commutative diagram of rational maps
(with $v$ and the Abel map not everywhere defined)

$$\begin{CD}
\oM_{0,n+1} @>{\pi}>> \prod_\alpha\oM_{0,\Gamma_\alpha\cup\{n+1\}}\\
@V{\pi_N}VV @VV{u}V \\
\oM_{0,n} @>v>> \oPicone
\end{CD}$$

We lift boundary divisors of $\oM_{0,n}$ defined above to boundary divisors
$\delta_{ij}$ and $\delta_I$ of $\oM_{0,n+1}$, respectively.
By Lemma~\ref{contracted boundary}, these divisors are not contracted by $\pi$.
Notice that $\pi(\delta_{ij})$ and $\pi(\delta_{\Gamma_\beta})$ are not boundary
divisors of $\oM_{\Gamma_\cup\{n+1\}}$ and are therefore mapped to $\Pic^{\uone}$.
We will prove in Lemma \ref{dominance} that $v(\delta_{ij})$ and $v(\delta_{\Gamma_\beta})$
are divisors in $\oPicone$ (but not boundary divisors).

Next we consider $\delta_I$ such that $1<|I|<|\Gamma_\alpha|-1$.
This divisor is mapped to a divisor 
$$\{\delta_I\subset\oM_{0,\Gamma_\alpha\cup\{n+1\}}\}\times\prod_{\beta\ne\alpha}\oM_{0,\Gamma_\beta\cup\{n+1\}}\subset
\oM_{\Gamma_\cup\{n+1\}}.$$
Note that the Abel map can be extended to the interior of this divisor and maps it 
to the corresponding vertical boundary divisor of  $\oPicone$. By Lemma \ref{dominance} (iii), this map is dominant.

Finally, consider $\delta_{\Gamma_\alpha\setminus\{i\}}$.
This divisor is mapped to a divisor 
$$\{\delta_{i,n+1}\subset\oM_{0,\Gamma_\alpha\cup\{n+1\}}\}\times\prod_{\beta\ne\alpha}\oM_{0,\Gamma_\beta\cup\{n+1\}}\subset
\oM_{\Gamma_\cup\{n+1\}}.$$
This divisor maps onto the horizontal boundary divisor
that corresponds to the $i$-th node of the $\alpha$-th irreducible component (use Lemma \ref{dominance} (iv)).

\begin{Lemma}\label{dominance}
The Abel map restricted to $\pi(\delta_I)$ generically has one-dimensional fibers if $I$ is either contained in some $\Gamma_\beta$ or if 
$|I|=2$.

\end{Lemma}

\begin{proof}
Let $\ov{N}=(N\setminus I)\cup\{p\}$. Denote by $\delta^o_I$ the interior of the boundary $\delta_I$:
$$\delta^o_I\cong M_{0,I\cup\{p\}}\times M_{0,\ov N\cup\{p\}}.$$

Let  $\ov\Gamma$ be the collection of subsets of 
$\ov{N}$ obtained by identifying the points in $I$ with $p$ (and throwing away any subsets with fewer than three elements). 
Note that $\ov\Gamma$ has maximum capacity (moreover, in case (ii), (iii), (iv) $\ov\Gamma$ is a hypertree on $\ov N$)
and therefore, Remark \ref{more general} applies. 

We denote by $\ov\Sigma$ the corresponding hypertree curve and by $\Pic^{\uone}_{\ov\Gamma}$ 
the relative Picard scheme of line bundles of multi-degree $(1,\ldots,1)$. Similarly, we let $G^1_{\ov\Gamma}$, 
$\tilde{G}^1_{\ov\Gamma}$, etc be the corresponding Brill-Noether loci. 
We will use the usual commutative diagram of morphisms (with $\ov\pi$ the product of forgetful maps
and $\ov{u}$ the Abel map corresponding to $\ov\Gamma$):

$$\begin{CD}
M_{0,\ov N\cup\{p\}} @>{\ov\pi}>> \prod_\alpha M_{0,\ov\Gamma_\alpha\cup\{p\}}\\
@V{\ov\pi_N}VV @VV{\ov{u}}V \\
\oM_{0,\ov N} @>>> \Pic^{\uone}_{\ov\Gamma}
\end{CD}.$$

The main observation that we will use is that generically along the image of $\ov\pi$, the Abel map $\ov{u}$ has one-dimensional fibers.

Consider first the case when $|I|=2$ for $I\not\subset\Gamma_\alpha$ for any $\alpha$. We have:
$$\de^o_I\cong M_{\ov N\cup\{p,n+1\}}\cong \tilde{G}^1_{\ov\Gamma}.$$

A point  $[m]$ in $M_{0,\ov N\cup\{p,n+1\}}$ corresponds via the above isomorphism to a morphism
$\ov{f}:\ov\Sigma\ra\PP^1$ and an admissible section $\ov{s}$ of $\ov{f}^*\O(1)$. By abuse of notation
we consider $[m]$ as a point of $\de^o_I$. Then $\pi([m])\in M_{\Gamma\cup\{n+1\}}$ corresponds 
to a pair $(\Sigma,s)$, where $s$ is an admissible section such that
$s=\rho^*\ov{s}$, where $\rho:\Sigma\ra \ov\Sigma$ is the map that collapses $i, j$ to $p$.
Case (i) now follows from the following commutative diagram:
$$\begin{CD}
M_{0,\ov N\cup\{p\}} @>{\ov{\pi}}>> \ov\pi(M_{0,\ov N\cup\{p\}}) @>{\ov{u}}>> \ov{u}(\ov\pi(M_{0,\ov N\cup\{p\}}))\\
@V{\cong}VV @V{\cong}V{\rho^*}V  @VV{\rho^*}V\\
\de^0_I @>{\pi}>>\pi(\de^0_I) @>{u}>>u(\pi(\de^0_I)).
\end{CD}$$
as the vertical maps (which are given by pull-back by $\rho$) are bijective.

Consider now case when $I=\Gamma_\beta$. A point  $[m]$ in $M_{0,\ov N\cup\{p\}}\cong \tilde{G}^1_{\ov\Gamma}$ corresponds to a morphism
$\ov{f}:\ov\Sigma\ra\PP^1$ and an admissible section $\ov{s}$ of $\ov{f}^*\O(1)$. We have:
$$\de^o_I\cong M_{\ov N\cup\{p,n+1\}}\times M_{\Gamma_\beta\cup\{p\}}\cong \tilde{G}^1_{\ov\Gamma}\times M_{\Gamma_\beta\cup\{p\}}.$$

If $([m],[m'])\in\de^o_I$ then the point $\pi([m],[m'])$ in $M_{\Gamma\cup\{n+1\}}$ corresponds 
to a pair $(\Sigma,s)$ with the following properties: there is a morphism  
$$\rho:\Sigma\ra\ov\Sigma$$ that collapses the component $\Sigma_\beta$ to $p$, 
and if $q\in\Sigma_\beta$ is determined by $[m']\in M_{\Gamma_\beta\cup\{p\}}$, then 
$s=\rho^*\ov{s}+q$ is the corresponding admissible section. Note, if we fix a point in 
the image of $\pi(\de^o_I)$ via the Abel map, this fixes the element $[m']$, and thus $q$. 
Case (ii) now follows from a similar commutative diagram (in the diagram above, take products with 
$M_{\Gamma_\beta\cup\{p\}}$ in the first row).

The remaining cases are similar. 
\end{proof}
\end{proof}

\section{Planar Realizations of Hypertrees}\label{planarity}

To distinguish between the Brill-Noether loci of different collections of subsets, we denote by $G^2(\Gamma)$ the Brill-Noether locus 
$G^2$ corresponding to a collection of subsets $\Gamma=\{\Gamma_1,\ldots, \Gamma_d\}$.
Recall that an element of $G^2(\Gamma)$ can be obtained by composing a morphism $\Sigma\ra\PP^2$
with a linear projection $\bP^2\dra\bP^1$, such that the morphism 
\begin{itemize}
\item has degree $1$ on each component of $\Sigma$,
\item separates points in $N$.
\end{itemize}
So basically we choose $N$ different points in $\bP^2$  
such that for each $\alpha$, the points in 
$\Gamma_{\alpha}$ are collinear. By Theorem \ref{Irreducibility}, 
if $\Gamma$ is an irreducible hypertree,  $G^2(\Gamma)$ is an irreducible 
subvariety of codimension $1$ in $\M_{0,n}$. Recall that a hypertree $\Gamma$ has a 
\emph{planar realization} if there exists a map $\Sigma\ra\PP^2$ such that 
all points in $N$ are distinct and the points in a subset $S\subset N$ 
with  are collinear {\em if 
and only if} $S\subset\Gamma_{\alpha}$ for some $\alpha$. 
Clearly, this is an open condition on $G^2(\Gamma)$. We prove that
this open set is non-empty:

\begin{Theorem}\label{PlanarTh}
Any  irreducible hypertree has a planar realization. 
\end{Theorem}

\begin{proof}
Let $\Gamma=\{\Gamma_1,\ldots,\Gamma_d\}$ be an irreducible hypertree. Assume $\Gamma$ does not have a planar realization. 
It follows that there is a triple 
$$\Gamma_0=\{a,b,c\}\subset N,$$ 
not contained in any $\Gamma_{\alpha}$, 
such that the points in $\Gamma_0$ are collinear for any $\Sigma\ra\PP^2$ 
that gives a point of $G^2(\Gamma)$. 

Let $\tilde{\Gamma}=\Gamma\cup\{\Gamma_0\}$. 
By our assumption, $G^2(\Gamma)=G^2(\tilde{\Gamma})$. 
We may assume $a\in\Gamma_1$. Since $\Gamma_1$ does not contain $\Gamma_0$, we may assume 
$b\notin\Gamma_1$. Let $\Gamma'_1=\Gamma_1\setminus\{a\}$.  
Construct a new collection of subsets 
$\Gamma'$:
\begin{itemize}
\item[(A) ] If $|\Gamma_1|=3$, let $\Gamma'=\{\Gamma_2,\ldots,\Gamma_d,\Gamma_0\}$
\item[(B) ] If $|\Gamma_1|\geq4$, let $\Gamma'=\{\Gamma'_1,\Gamma_2,\ldots,\Gamma_d,\Gamma_0\}$.
\end{itemize}

\begin{Claim}\label{Gamma'}
The collection of subsets $\Gamma'$ is a hypertree.
\end{Claim}

\begin{proof}[Proof of Claim \ref{Gamma'}]
We prove that $\Gamma'$ satisfies the convexity axiom \eqref{CondS}.  As $\Gamma$ is an irreducible hypertree, for any 
$S\subset\{2,\ldots,d\}$ we have:
\begin{gather*}
\bigl|\Gamma_0\cup\bigcup_{j\in S}\Gamma_j\bigr|\geq \bigl|\bigcup_{j\in S}\Gamma_j\bigr|\geq
\sum_{j\in S}(|\Gamma_j|-2)+3=\\=\sum_{j\in S}(|\Gamma_j|-2)+(|\Gamma_0|-2)+2.
\end{gather*}

Similarly, if $S\subsetneq\{2,\ldots,d\}$ we have
\begin{gather*}
\bigl|\Gamma'_1\cup\Gamma_0\cup\bigcup_{j\in S}\Gamma_j\bigr|\geq \bigl|\Gamma'_1\cup\bigcup_{j\in S}\Gamma_j\bigr|\geq\\
\geq \bigl|\Gamma_1\cup\bigcup_{j\in S}\Gamma_j\bigr|-1\geq
\sum_{j\in S}(|\Gamma_j|-2)+(|\Gamma_1|-2)+3-1=\\=\sum_{j\in S}(|\Gamma_j|-2)+(|\Gamma'_1|-2)+(|\Gamma_0|-2)+2.
\end{gather*}

It is easy to see that $\Gamma'$ satisfies the normalization axiom \eqref{normalizat}. It follows that
$\Gamma'$ is a hypertree (possibly not irreducible).
\end{proof}

We use our working definition of $D_{\Gamma}$, 
namely $D_{\Gamma}=\ov{G^2(\Gamma)}$. Similarly, $D_{\Gamma'}=\ov{G^2(\Gamma')}$.
By Theorem \ref{coolcondition} and Theorem \ref{size/explicit} $D_\Gamma'$ is a divisor in $\M_{0,n}$ (possibly reducible) and the map
$$\pi':\,\oM_{0,n+1}\to\M_{\Gamma'\cup\{n+1\}}:=\prod_{\Gamma'_{\alpha}}\M_{0,\Gamma'_{\alpha}\cup\{n+1\}}$$ 
is a birational morphism whose exceptional locus consists of $\pi^{-1}_N(D_{\Gamma'})$ and boundary divisors in $\M_{0,n+1}$ 
contracted by $\pi'$.

By Theorem \ref{Irreducibility}, $D_{\Gamma}=\ov{G^2(\Gamma)}$ is an irreducible divisor in $\M_{0,n}$. 
In addition, we have $G^2(\tilde{\Gamma})\subseteq G^2(\Gamma')$ and by assumption $G^2(\tilde{\Gamma})=G^2(\Gamma)$. 
It follows that  $D_{\Gamma}$ is an irreducible component of $D_{\Gamma'}=\ov{G^2(\Gamma')}$. 

Let $\cE_1,\ldots,\cE_s$ be the irreducible components of ${\pi^{-1}_N}D_{\Gamma'}$. We may assume $\cE_1={\pi^{-1}_N}D_{\Gamma}$. 
By Theorem  \ref{Irreducibility}, we have:
$$\cE_1={\pi^{-1}_N}D_{\Gamma}=(d-1)H-\sum_{I\subset N\atop 1\le |I|\le n-3} m_IE_I,$$
where $m_I$  satisfies the inequality (\ref{coefficientm_I}). By (\ref{hikitty}) we have $m_{\{i\}}=d-v_i$.

\begin{Notation}
Let $d'$ be the number of hyperedges in $\Gamma'$. (Hence, $d'=d$ in Case (A) and $d'=d+1$ in Case (B).)
Denote by $v'_i$ the valence of $i\in N$ in $\Gamma'$. 
\end{Notation}

\begin{Lemma}\label{weir}
The classes of the divisors $\cE_i$ are subject to the following relation:
\begin{equation}\cooltag\label{johnny}
\sum_{i=1}^sc_i\cE_i=(d'-1)H-\sum_{I\subset N\atop 1\le |I|\le n-3} m'_IE_I-\sum_{|I|=n-2\atop \de_{I\cup\{n+1\}}\in\Exc(\pi')} a'_I\de_{I\cup\{n+1\}},
\end{equation}
where  $c_1,\ldots,c_s$ are positive integers, $a'_I$ is the discrepancy of the divisor 
$\de_{I\cup\{n+1\}}$ with respect to the map $\pi'$ and the integers $m'_I\geq0$ satisfy the following inequality:
\begin{equation}\cooltag\label{m'_I}
m'_I\geq |I|-1-\Cap((\Gamma')_I)+\left|\{\Gamma'_\alpha\,|\,\Gamma'_\alpha\subset I^c\}\right|.
\end{equation}

In particular, we have: $$m'_i\geq d'-v'_i.$$
\end{Lemma}

\begin{proof}
Note that formula (\ref{discrrr}) still holds (the map $\pi'$ is birational):
\begin{equation}
\sum_{i=1}^sc_i\cE_i=K_{\M_{0,n+1}}-{\pi'}^*K_{\M_{\Gamma\cup\{n+1\}}}-\sum_{\de_{I\cup\{n+1\}}\in\Exc(\pi')}a'_I\de_{I\cup\{n+1\}}.
\cooltag\label{discrrr part2}
\end{equation}

For the purpose of the Lemma, we ignore the terms $a'_I\de_{I\cup\{n+1\}}$ for $|I|=n-2$ in the above formula.
Then the lemma follows from \ref{canonical computations}, combined with the inequality $a'_I\geq\codim\pi'(\de_{I\cup\{n+1\}})-1$ 
and Lemma \ref{capacity I'}.
\end{proof}

We compare the coefficient of $H$ in both sides of the equation (\ref{johnny}). Recall that the coefficient of $H$ in $\cE_1$ is $d-1$.

Consider first Case (A). Since the degree of $H$ is at least $d-1$ in the left hand-side, and at most $d-1$ on the right, it follows that 
$a'_I=0$ for all $|I|=n-2$ and  moreover $s=1$, $c_1=1$, i.e., we have:
$$\cE_1={\pi^{-1}_N}D_{\Gamma}=(d-1)H-\sum_{I\subset N\atop 1\le |I|\le n-3} m'_IE_I.$$

It follows that $m_i=m'_i$ for all $i$. By Lemma \ref{weir}, $m'_i\geq d-v'_i$. By (\ref{hikitty}) $m_i=d-v_i$. This leads to a contradiction,
since $v'_i<v_i$ for all $i\in \Gamma_1\setminus\{a,b,c\}$ (we use here the assumption that $\Gamma_1\setminus\{a,b,c\}\neq\emptyset$). 

Consider now Case (B). The coefficient of $H$ on the right hand-side of (\ref{johnny}) is at most $d$, while the  
the coefficient of $H$ in $\cE'_1$ is $d-1$. If $s>1$, it follows that $s=2$ and $\cE_2$ is an irreducible divisor that has $H$-degree $1$. 
From the Kapranov blow-up model of $\M_{0,n+1}$ one can see that either $\cE_2$  is a boundary divisor or $h^0(\cE_2)>1$. 
This is a contradiction, since $\cE_2$ is a divisor that intersects the interior of $\M_{0,n+1}$ and moreover, it
is an exceptional divisor for the birational map $\pi'$. The same argument shows that $c_1=1$. 

Moreover, we must have $a'_{I_0}=1$, for some $|I_0|=n-2$ (with $a'_I=0$ for all $I\neq I_0$). Let $\{u,v\}=I_0^c$.
We have:
$$\cE_1={\pi^{-1}_N}D_{\Gamma}=dH-\sum_{I\subset N\atop 1\le |I|\le n-3} m'_IE_I-\de_{u,v}.$$

In particular, $m_i=m'_i-1$ for all $i\neq u,v$ and $m_i=m'_i$ if $i\in\{u,v\}$. 
Note that $v'_b=v_b+1$, $v'_c=v_c+1$, while $v'_i=v_i$ for all $i\neq b,c$. 
By Lemma \ref{weir}, $m'_i\geq d'-v'_i=d+1-v'_i$ for all $i$. Since $m_i=d-v_i$ for all $i$, it follows that 
if $i\neq\{b,c\}$ then  $i\neq\{u,v\}$, i.e., $\{u,v\}=\{b,c\}$. We have:
\begin{equation}\cooltag\label{lady}
\cE_1={\pi^{-1}_N}D_{\Gamma}=dH-\sum_{I\subset N\atop 1\le |I|\le n-3} m'_IE_I-\de_{b,c}.
\end{equation}

We consider the coefficients $m_I$ and $m'_I$ for $I=\Gamma'_1$. By (\ref{hipapa}), we have
\begin{equation}\cooltag\label{hipapa2}
m_I=d+|I|-1-\sum_{i\in I}v_i.
\end{equation}
By Lemma \ref{weir}, we have: 
\begin{equation}\cooltag\label{gaga}
m'_I\geq |I|-1-\Cap((\Gamma')_I)+\left|\{\Gamma'_\alpha\,|\,\Gamma'_\alpha\subset I^c\}\right|.
\end{equation}
Recall that we assume $b\notin\Gamma_1$. Note that 
\begin{equation}\cooltag\label{poker}
\left|\{\Gamma_\alpha\,|\,\Gamma_\alpha\subset I^c\}\right|=d-\sum_{i\in I} v_i+|I|-1.
\end{equation}
We will compare $\left|\{\Gamma_\alpha\,|\,\Gamma_\alpha\subset I^c\}\right|$ with
$\left|\{\Gamma'_\alpha\,|\,\Gamma'_\alpha\subset I^c\}\right|$.

We consider two cases. First, assume $c\notin\Gamma_1$.  Then $(\Gamma')_I=\{I\}$. Hence, 
$\Cap((\Gamma')_I)=|I|-2$. Since $\Gamma_{0}=\{a,b,c\}\subset I^c$ it follows that:
$$m'_I\geq 1+\left|\{\Gamma'_\alpha\,|\,\Gamma'_\alpha\subset I^c\}\right|=2+\left|\{\Gamma_\alpha\,|\,\Gamma_\alpha\subset I^c\}\right|=
d+|I|+1-\sum_{i\in I} v_i,$$
which contradicts (\ref{hipapa2}) since by (\ref{lady}) $m_I=m'_I-1$.

Now assume $c\in\Gamma_1$. By (\ref{poker}) and from $\Gamma_{0}\nsubseteq I^c$ it follows that
$$m'_I\geq 1+\left|\{\Gamma'_\alpha\,|\,\Gamma'_\alpha\subset I^c\}\right|=1+\left|\{\Gamma_\alpha\,|\,\Gamma_\alpha\subset I^c\}\right|=
d+|I|-\sum_{i\in I} v_i.$$ This contradicts (\ref{hipapa2}), since by (\ref{lady}), $m_I=m'_I$.
\end{proof}

\section{Spherical and not so Spherical Hypertrees}\label{spherical}

\begin{Theorem}\label{heroic}
Let $\cK$ be an even (i.e.,~bicolored) triangulation of a sphere with $n$ vertices.
Then  its collection of black (resp.~white) triangles $\Gamma$ (resp.~$\Gamma'$)
is a hypertree. It is irreducible
if and only if $\cK$ is not a connected sum of two triangulations.
\end{Theorem}

\begin{proof}  
Let $d$ (resp.~$d'$) be the number of triangles in $\Gamma$ (resp.~$\Gamma'$). 
Since $\cK$ has $3d=3d'$ edges, we have  $d=d'$. By Euler's formula,
$$n-3d+2d=2,$$
and therefore $d=n-2$.

\begin{Review}\label{slick} 
Take any $k$ black triangles $\Gamma_1,\ldots,\Gamma_k$ and let $\Delta\subset S^2$ be their
union.
As~a simplical complex, $\Delta$ has $k$ faces,
$3k$ edges, and $|\Gamma_1\cup\ldots\cup\Gamma_k|$ vertices.
Since $h_2(\Delta)=0$, we have
$$\chi(\Delta)=h_0(\Delta)-h_1(\Delta)=\bigl|\bigcup_{i=1}^k\Gamma_i\bigr|-2k.$$
Abusing notatation, let $S^2\setminus \Delta$ denote the simplicial complex
obtained by removing interiors of triangles in $\Delta$. Let $D$ be the closure of a 
connected component of the set $S^2\setminus \Delta$ with vertices removed. Note that
$D$ is not necessarily a polygon  (it is not necessarily simply connected), 
but its boundary edges are well-defined. Their number $e(D)$
is equal to three times the number of white triangles inside $D$
minus three times the number of black triangles inside $D$.
It follows that $3|e(D)$. Then the number of edges in 
$\partial(S^2\setminus \Delta)=\partial\Delta$ equals
\begin{equation}\cooltag\label{meyers}
3k=\sum e(D_i)\geq 3h_0(S^2\setminus \Delta).
\end{equation}
This implies that 
\begin{equation}\hbox{\rm 
$h_0(S^2\setminus \Delta)\le k$
for any union $\Delta$ of $k$ black (resp.~white) faces,}
\cooltag\label{Lenin}\end{equation}
\end{Review}

\begin{Review}[$\Gamma$ satisfies the convexity axiom \eqref{CondS}]

By Alexander duality, 
$$h_1(\Delta)=h_0(S^2\setminus\Delta)-1\le k-1,$$
and we have
$$\bigl|\bigcup_{i=1}^k\Gamma_i\bigr|-2
=2k+h_0(\Delta)-h_1(\Delta)-2\ge$$
$$\ge
2k+h_0(\Delta)+1-k-2\ge k.
$$
It follows that $\Gamma$  is a hypertree.
\end{Review}

\begin{Review}
Suppose that $\Gamma$ is not irreducible. 
Then one can find a subset of $k$ black triangles $\Gamma_1,\ldots,\Gamma_k$
as above with $1<k<n-2$ such that all inequalities above are equalities, i.e.
$$h_0(S^2\setminus\Delta)=k,\quad h_0(\Delta)=1.$$

Hence, $S^2\setminus\Delta$ has $k$ connected components $D_1,\ldots,D_k$. Moreover,
using (\ref{meyers}) we have $e(D_i)=3$ for all $i$. Some (but not all) of the $D_i$'s are just white triangles $\cK$, 
others are unions of black and white triangles. But all of them are simpy-connected polygons, since 
$h_1(S^2\setminus\Delta)=0$ by Alexander duality (hence, $S^2\setminus\Delta$ is simply connected).

Now it is clear that we are done:
Let $D$ be one of the connected components of $S^2\setminus\Delta$
which is not a white triangle. But the boundary of $D$ is a triangle, and it is clear that $\cK$
is a connected sum of two triangulations  $\cK_1$ and $\cK_2$
glued along the boundary of $D$.
Namely, $\cK_1$ is formed by removing all triangles inside $D$ and gluing a white triangle along the boundary of $D$ instead.
And $\cK_2$ is formed by removing all triangles not in $D$ and gluing a black triangle along the boundary of $D$ instead.
\end{Review}

And the other way around, if $\cK$ is a connected sum of  
triangulations $\cK_1$ and $\cK_2$ then $\Gamma$ is not irreducible:
just take the set $S$ to be the set of all black triangles of $\cK_1$.
\end{proof}

The proof of Theorem \ref{heroic} shows the following:
\begin{Lemma}\label{e(D)} If $\cK$ is an even triangulation of a sphere and $D$ is 
a polygon such that any triangle inside $D$ adjacent to a boundary edge of $D$ is white (equivalently,  
$D$ is one of the connected components of the complement to a union of black triangles in $\cK$) 
then the number of edges of $D$ is divisible by $3$. Moreover, $\cK$ is irreducible if and only if 
whenever the number of edges of $D$ is three, then $D$ is a white triangle or the complement of a black triangle.
\end{Lemma}

We will prove in Corollary~\ref{coinciwb} that white and black hypertrees of any irreducible even triangulation
give the same divisor on $\oM_{0,n}$. We will now show that
under a mild genericity assumption there are no other hypertrees that give the same divisor.

\begin{Definition}\label{generic hypertree}
Let $\Gamma$ be an irreducible hypertree composed of  triples. We call $\Gamma$ 
\emph{generic} if for any triple $\{i,j,k\}\subset N$ that is not a hyperedge or a wheel (see~\ref{wheel}), we have
\begin{equation}\label{sjxzvbhkjzdfbk}
\Cap(\Gamma_{N\setminus\{i,j,k\}})=n-4,
\cooltag\end{equation}
where $\Gamma_{N\setminus\{i,j,k\}}$ is the collection of triples obtained from $\Gamma$
by identifying vertices $i$, $j$, and $k$ (and removing triples which contain two of the points $i,j,k$).
\end{Definition}

\begin{Theorem}\label{amazing}
Let $\Gamma$, $\Gamma'$ be generic hypertrees. 
If  $D_\Gamma=D_{\Gamma'}$ then $\Gamma'$ is irreducible and there exists a bicolored triangulation $\cK$ of~$S^2$
such that $\Gamma$ is its collection of black faces and~$\Gamma'$ is its collection of white faces.
In this case $\Gamma$ uniquely determines the triangulation.
\end{Theorem}

Theorem \ref{amazing} and Lemmas \ref{pull-back}, \ref{new} give a lower bound on the number of extremal rays of the effective cone of
$\oM_{0,n}$, namely, the number of generic non-spherical irreducible hypertrees plus half of the number of generic spherical 
irreducible hypertrees (on all subsets of $N$).

\begin{Lemma}\label{pull-back}
Let $\Gamma$ be an irreducible hypertree on a subset $K$ of $N$ and consider the forgetful map $\pi_K: \oM_{0,n}\ra\oM_{0,K}$.
Then $\pi_K^{-1}D_\Gamma$ generates an extremal ray of $\oEff(\oM_{0,n})$.
\end{Lemma}

\begin{proof}
We may assume without loss of generality that $K=\{1,\ldots,k\}$ for $k\leq n$. By Theorem \ref{irreducible}, the divisor $D_\Gamma$ is irreducible. 
Therefore, the divisor $\pi_K^{-1}D_\Gamma$ is irreducible. Moreover, $\pi_{\Gamma\cup\{n+1\}}(D_\Gamma)$ has codimension
at least two in $\oM_{\Gamma\cup\{n+1\}}$. It is enough to construct a hypertree $\tilde{\Gamma}$ on the set $N$ such that 
$\pi_K^{-1}D_\Gamma$ is in the exceptional locus of  $\pi_{\tilde{\Gamma}\cup\{n+1\}}$. This will be the case if for example
$\Gamma\subseteq\tilde{\Gamma}$. Let $\tilde{\Gamma}=\Gamma\cup\{\Gamma'_1,\ldots, \Gamma'_{n-k}\}$ where
$$\Gamma'_1=\{k+1,1,2\},\quad\Gamma'_2=\{k+2,1,3\},\ldots, \Gamma'_{n-k}=\{n,1,n-k\}.$$
\end{proof}

\begin{Lemma}\label{new}
Let $\Gamma$ be an irreducible hypertree on $N$. If for some forgetful maps $\pi:\oM_{0,\tilde{N}}\ra\oM_{0,N}$ and  
$\pi':\oM_{0,\tilde{N}}\ra\oM_{0,N'}$ for subsets $N$ and $N'$ of $\tilde{N}$, we have 
$$\pi^{-1}(D_\Gamma)={\pi'}^{-1}(D_{\Gamma'}),$$ for some irreducible hypertree $\Gamma'$ on $N'\subseteq \tilde{N}$, then $N=N'$, 
$D_\Gamma=D_{\Gamma'}$.
\end{Lemma}

\begin{proof}
Consider the divisor class of the pull-back $D$ of $\pi^{-1}(D_\Gamma)$ to $\oM_{0,|\tilde{N}|+1}$ in the Kapranov model 
with respect to the $|\tilde{N}|+1$ marking. Using Theorem \ref{class}, we have
$$D=(d-1)H-\sum_{i\in \tilde{N}\setminus N}(d-1)E_i-\sum_{i\in N}(d-v_i)E_i\ldots,$$
where $v_i\geq2$ is the valence of $i$ in $\Gamma$. If $$\pi^{-1}(D_\Gamma)={\pi'}^{-1}(D_{\Gamma'}),$$
then $d=d'$ and by reading off the coefficients of $E_i$ that are equal to $d-1$, it follows that $N=N'$ and $D_\Gamma=
D_{\Gamma'}$. 
\end{proof}

\begin{proof}[Proof of Theorem~\ref{amazing}]
Comparing the classes of $D_\Gamma$ and $D_\Gamma'$ given in Theorem~\ref{class}, we see that
$d=d'=n-2$, i.e.,~$\Gamma'$ is also composed of triples, and for each $i\in N$, the hypertrees $\Gamma$ and $\Gamma'$
have the same valences $v_i$.

Let $\Xi$ (resp.~$\Xi'$) be the collection of wheels of $\Gamma$ (resp.~$\Gamma'$). We claim that
$$\Gamma\cup\Xi=\Gamma'\cup\Xi'.$$

Let $m$ (resp., ~$m'$) be coefficients of the class of $D_\Gamma$ (resp.~$D_{\Gamma'}$), as in 
Theorem~\ref{class}. Then by (\ref{kitty1}) and (\ref{coefficientm_I}) $m'_{N\setminus\{i,j,k\}}\geq1$ for any 
triple $\{i,j,k\}$ that is a hyperedge or a wheel in $\Gamma'$. But since $\Gamma$ is a generic hypertree,
using Lemma \ref{nextfix}, we have $m_{N\setminus\{i,j,k\}}=0$ for any triple $\{i,j,k\}$ which is not a hyperedge or a wheel. 
This proves that $\Gamma'\cup\Xi'\subset\Gamma\cup\Xi$. Since both $\Gamma$, $\Gamma'$ are generic hypertrees, this proves the claim.

Suppose $\Gamma\ne\Gamma'$.
Without loss of generality, we can assume that
$$\Gamma_1\in\Gamma\setminus\Gamma'.$$

We are going to construct a finite bi-colored $2$-dimensional polyhedral complex~$\cK$ inductively, 
as the union of complexes $\cK_1\subset\cK_2\subset\ldots$.
On each step, any black face of $\cK_i$ is going to be a hyperedge in $\Gamma$
and a wheel in $\Gamma'$, and vice versa for white faces.

Let's define $\cK_1$. Its vertices are indexed by $\Gamma_1$. Since $\Gamma_1$ is a wheel in $\Gamma'$,
it can be identified with a triangle in a unique way, where edges of the triangle are precisely intersections 
(with two elements) of $\Gamma_1$ with hyperedges of $\Gamma'$.
So we let $\cK_1$ be this polygon, colored black.

Next we define an inductive step. Suppose $\cK_n$ is given.
Take a face $X$. Then $X$ is either black or white. The construction is absolutely symmetric,
so let's suppose that $X$ is black. 
Then the set of vertices of $X$ is a hyperedge in $\Gamma$ and a wheel in $\Gamma'$.
Moreover, we will make sure that, in our inductive construction, edges of $X$ are exactly
intersections (with two elements) of $X$ with hyperedges of $\Gamma'$.
Notice that this holds for $\cK_1$. 

Let $\{a,b\}\subset X$ be an edge
that is not an edge of some white face.
If any edge of $X$ is also an edge of some white face then discard $X$, and try another face.
If we can not find a face with an edge that is not an edge of some face of an opposite color
then the algorithm stops.

Since $X$ is a wheel of $\Gamma'$, $\{a,b\}$ is the intersection of $X$ with a unique hyperedge $Y$ of $\Gamma'$.
This will be our next face. Since $a,b\in X$, $X$ is a unique hyperedge in $\Gamma$ containing $a,b$.
So $Y$ must be a wheel in $\Gamma$. Therefore, we can identify $Y$ with vertices of a triangle
such that its edges are identified with ($2$-pointed) 
intersections of $Y$ with hyperedges in $\Gamma$.
For example, $(a,b)$ will be one of these edges. We define $\cK_{n+1}$ as $\cK_n$ with $Y$ added 
as a new white polygon. 

We have to check that $\cK_{n+1}$ is a bi-colored polygonal complex, i.e.~that any two faces
of $\cK_{n+1}$ share at most two vertices, and if they share exactly 
two vertices, then in fact they share an edge and are colored differently.
So let $Z$ be a face of $\cK_n$ such that $|Z\cap Y|>1$ but $Z\ne Y$. 
Then $Z$ can not be a hyperedge of $\Gamma'$, so $Z$ is a black face.
Since $Z$ is a wheel of $\Gamma'$, $Z\cap Y$ is an edge of $Z$.
And since $Y$ is a wheel in $\Gamma$, $Z\cap Y$ is an edge of $Y$. 

At some point this algorithm stops. Let $\cK$ be the resulting polygonal complex.
Let $\{\Gamma_i\,|\,i\in S\}$ (resp.~ $\{\Gamma'_i\,|\,i\in S'\}$) be the collection of its black faces
(resp.~white faces) for some $S\subset\{1,\ldots,n-2\}$ (resp.~$S'\subset\{1,\ldots,n-2\}$).
Let 
$$e=\sum_{i\in S}|\Gamma_i|=\sum_{i\in S'}|\Gamma'_i|$$
be the number of edges of $\cK$ and let
$$v=|\cup_{i\in S}\Gamma_i|=|\cup_{i\in S'}\Gamma'_i|$$
be the number of its vertices.
Finally, let $f=f_b+f_w$ be the number of its faces, where $f_b=|S|$ (resp.~$f_w=|S'|$)
is the number of black faces (resp.~white faces).

Notice that apriori $\cK$ is not necessarily homeomorphic to a closed surface,
because at some vertices of $\cK$ several sheets can come together.
At these points, the link of $\cK$ is homeomorphic to the disjoint union of several circles.
Let $\bar\cK\to\cK$ be the ``normalization'' obtained by separating these sheets.
Then $\bar\cK$ is homeomorphic to a closed surface.
Let $\bar v\ge v$ be the number of vertices in $\bar\cK$.
We have
$$2(\bar v-e+f)\ge 2v-2e+2f_b+2f_w=$$
$$=|\cup_{i\in S}\Gamma_i|+|\cup_{i\in S'}\Gamma'_i|
-\sum_{i\in S}|\Gamma_i|-\sum_{i\in S'}|\Gamma'_i|+2|S|+2|S'|=$$
$$=|\cup_{i\in S}\Gamma_i|-\sum_{i\in S}(|\Gamma_i|-2)+|\cup_{i\in S'}\Gamma'_i|
-\sum_{i\in S}(|\Gamma_i|-2)\ge4$$
since $\Gamma$ and $\Gamma'$ are hypergraphs. Since they are strong hypergraphs,
the inequality is strict unless $S=S'=\{1,\ldots,n-2\}$.
It follows that 
$$\chi(\bar\cK)\ge2,$$
and the inequality is strict unless $S=S'=\{1,\ldots,n-2\}$.
But the Euler characteristic can not be bigger than~$2$, with the equality if and only if $\bar\cK$ is a sphere.
It follows that $\bar\cK=\cK$ is a bi-colored triangulation of a $2$-sphere which uses all hyperedges in $\Gamma$
as black faces and all hyperedges in $\Gamma'$ as white faces.

It remains to show that $\Gamma$ uniquely determines the triangulation. It is enough to show that
$\Gamma'$ is the set of all wheels $\Xi$ of $\Gamma$. Since $\Gamma'\subseteq\Xi$,
it is enough to show that there are no wheels in $\Gamma$ other than the set of white triangles in $\cK$.
Assume there exists a wheel $\{i,j,k\}$ which is not a white triangle. 
Let $D$ be one of the two polygons on the sphere bordered by this wheel. 
By switching between the two polygons, we may assume that at least two of the three edges of
$D$ are bordered by two black triangles which lie outside of $D$. If the remaining black triangle bordering 
$D$ lies on the outside of $D$, this contradicts the irreducibility 
of $\Gamma$. If the remaining black triangle lies on the inside of $D$, then we obtain a 
polygon bordered by white triangles and having $4$ edges, which contradicts Lemma \ref{e(D)}.
\end{proof}

We now construct both spherical and non-spherical generic hypertrees.

\begin{Definition}\label{generictriang}
Let $\cK$ be an even triangulation. Let $D$ be a polygon such that any triangle inside $D$
adjacent to a boundary edge of $D$ is white. 

We call $\cK$ a \emph{generic triangulation} if:
\begin{itemize}
\item $\cK$ is irreducible, i.e.,~if $D$ has three edges then $D$ is a white triangle
or the complement of a black triangle (see Lemma \ref{e(D)}).
\item If $D$ has $6$ edges then $D$ is either a hexagon $A$ or $B$ or the complement of a hexagon $A'$ or $B'$
from the following picture:
\begin{figure}[htbp]\label{kprime}
\includegraphics[width=2in]{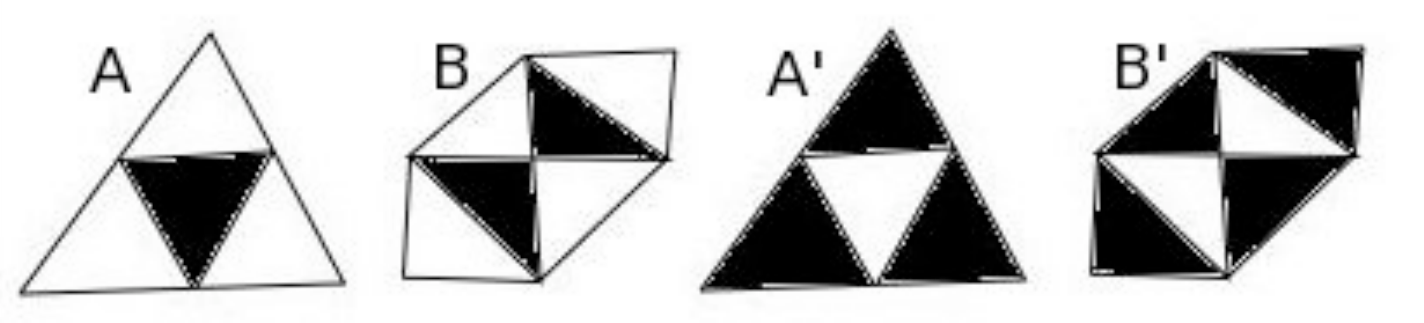}
\end{figure}
\end{itemize}
\end{Definition}

\begin{Remark}
Genericity means that vertices are sprinkled on the sphere sufficiently densely.
We didn't try to give a combinatorial classification 
of generic triangulations, although this is perhaps possible.
But to give a flavor of what's going on, suppose $\cK$ is any even triangulation
and let $\cK'$ be a ``quadrupled''  even triangulation obtained by the following procedure:
\begin{figure}[htbp]
\includegraphics[width=3in]{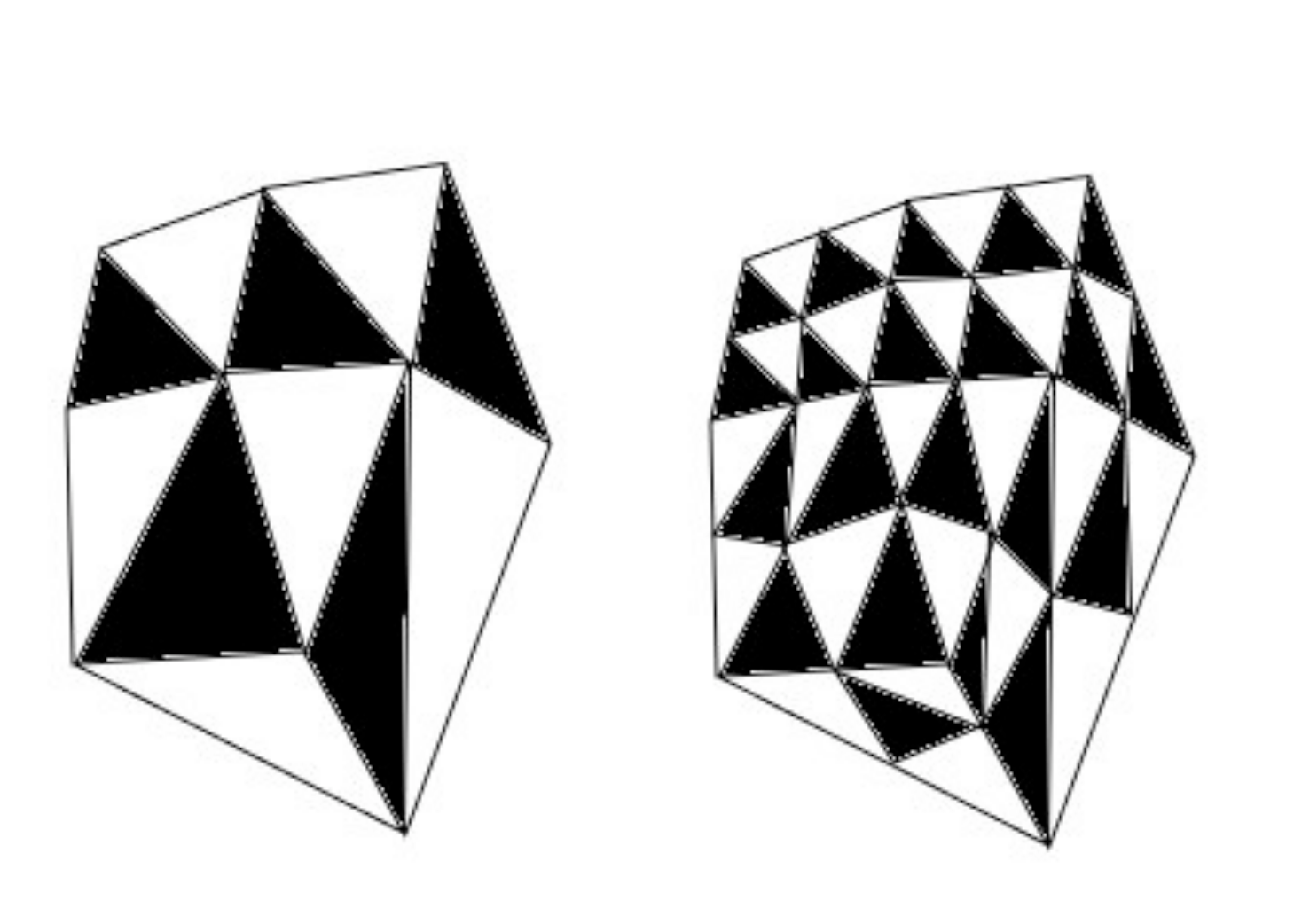}
\caption{Quadrupling a triangulation}\label{kprime}
\end{figure}
take all vertices in $\cK$ and add a midpoint of 
any edge of $\cK$  as a new point of~$\cK'$.
For any black (resp.~white) triangle $T=\{a,b,c\}$ of $\cK$, the triangulation $\cK'$ has black (resp.~white)
triangles $\{a,b',c'\}$, $\{a',b,c'\}$, and $\{a',b',c\}$
and a white (resp.~black) triangle $\{a',b',c'\}$, where $a',b',c'$ are new points in $T$ opposite to vertices $a,b,c$,
see Figure~\ref{kprime}.
It is not hard to see that after quadrupling $\cK$ several times the triangulation becomes generic.
Indeed, any closed path with $6$ edges will happen either in the region
of the triangulation that looks like a standard $A_2$-triangulated $\bR^2$, in which case~$D$ is a hexagon $A$ or $A'$,
or this path loops around a vertex of valence $\ne6$. In this case the valence must be equal to $4$, and we have a hexagon 
$B$ or $B'$. In fact, quadrupling just once is enough \cite{Har}. 
\end{Remark}

\begin{Lemma}\label{qqqcheck}
Let $\cK$ be a generic triangulation and let $\Gamma$ be its collection of black triangles.
Then $\Gamma$ is a generic hypertree, except when $n=8$ and $\cK$ is the triangulation 
given by the bipyramid (see Section \ref{HistoricalRemarks}).
\end{Lemma}

\begin{Remark}
The genericity assumption in Lemma \ref{qqqcheck} is necessary: the bipyramid 
is easily seen to not be a generic triangulation for $n>8$ (for example 
there are many loops with $6$ edges with black triangles on one side of it that pass once
through the north pole and once through the south pole).
We will show in \ref{HistoricalRemarks} that the corresponding divisor 
is a pull-back of the ``Brill--Noether divisor'' for a certain map $\oM_{0,n}\to\oM_{n-3}$,
and consequently its symmetry group is much larger than the dihedral group.
This divisor can be realized by various hypertrees $\Gamma'$
obtained from $\Gamma$ by permuting equatorial points.
By Lemma \ref{qqqcheck}, the bipyramid for $n=8$ is the only generic triangulation 
that does not correspond to a generic hypertree.  
\end{Remark}

\begin{proof}[Proof of Lemma~\ref{qqqcheck}]
We write $a\lra b$ if vertices $a$ and $b$ are connected by an edge. 
Up to symmetries, there are three possible cases.

{\bf Case X:} $ij$ and $jk$ are both edges of black triangles.
These triangles are removed in $\Gamma_{N\setminus\{i,j,k\}}$. 

{\bf Case Y:} $ij$ is an edge of a black triangle,
which will be removed in $\Gamma_{N\setminus\{i,j,k\}}$,
but $i\not\lra k$ and $j\not\lra k$. We also remove a black triangle adjacent to $k$ as follows:
if $i,j,k$ are vertices of the hexagon $A'$, then we remove the black triangle
inside the hexagon adjacent to $k$; in all other cases, we remove a random 
black triangle adjacent to $k$.

{\bf Case Z:} $i\not\lra j$, $j\not\lra k$, and $i\not\lra k$.
In this case we remove two black triangles adjacent to the same point (it could be $i$, $j$, or $k$) 
according to the following rules. If one of the points $i$, $j$, or $k$ has valence $2$ (see Definition \ref{valence}),
then we remove both triangles adjacent to this point (any of $i,j,k$ is going to work).
If each of the points $i,j,k$ has valence more than $2$, but these points are 
vertices of the hexagon $A'$, then we remove the black triangle
inside the hexagon adjacent to $i$ and any other black triangle adjacent to $i$.
In any other case we just remove two random black triangles adjacent to $i$.

We claim that the remaining $n-4$ triangles $\tilde\Gamma$ form a hypertree if we identify $i=j=k$.
Let $S\subset \tilde\Gamma$ be a proper subset of $s$ black triangles, with $1< s< n-4$.
It is enough to 
show that $S$ covers at least $s+2$ vertices (after we identify $i=j=k$).
Let $\Delta$ be the union of triangles in $S$ before the identification.
Since $\Gamma$ is irreducible, $\Delta$ contains at least $s+3$ vertices of $N$.
So~it suffices  to prove the following:

\begin{Claim}
If $i,j,k\in\Delta$ then $\Delta$ contains at least $s+4$ vertices of~$N$.
\end{Claim}

By \eqref{slick}, this claim is equivalent to the following more simple:

\begin{Claim}
The complement $S^2\setminus\Delta$ contains either  a connected component with at least $9$ sides
or at least two connected components with at least $6$ sides each (by~\ref{slick},  the number of sides
 is always divisible by $3$).
\end{Claim}

We argue by contradiction.
Note that we remove two triangles, and a connected component of $S^2\setminus\Delta$
that contains any of them has at least six edges. Therefore both removed triangles belong to the same connected component, call it $D$,
with six edges (and all other connected components are white triangles). Recall that $i,j,k\in\Delta$ and $\Delta\subset S^2\setminus D$.

We know how all hexagons look like: $D$ must either be the inside of a hexagon $A$ or $B$ or the ``outside" of a hexagons $A'$ or $B'$.
The hexagon~$A$ is excluded because it contains only one black triangle.

The hexagon $B$ contains two black triangles inside, so they must be the removed triangles. Since $i,j,k\in\Delta$, it 
must be that $i,j,k$ are on the boundary of the hexagon. In cases $X$ and $Y$ the removed triangles contain $i,j,k$; hence, two opposite
vertices of the hexagon $B$ are excluded. In this case it follows that one of $i,j,k$ is connected by an edge to the other two. This is only possible 
in case $X$. But in case $X$ the removed triangles have in common only $j$; hence, $j$ must be strictly inside the hexagon, which is a contradiction.
Finally, the case $Z$ is impossible because the removed triangles have in common $i$, therefore $i$ must be 
the point strictly inside the hexagon, which is a contradiction.

Suppose that $D$ is the outside of the hexagon $A'$ or~$B'$. Since the removed triangles are contained in $D$, it follows that
no two of $i,j,k$ can be connected by a black triangle inside the hexagon. 

Assume $D$ is the outside of the hexagon $A'$. Then $i,j,k$ are the three vertices of $A'$ with no two of them connected by an edge.
In cases $Y$ and $Z$ one of the removed triangles is inside $A'$, which is a contradiction. 

Assume we are in case $X$. Let $a$ (resp. $b$, resp. $c$) be the 
middle vertex (on the boundary of $A'$) between $i$, $j$ (resp. $j$, $k$, resp. $i$, $k$). We claim that $i,j,a$ and  $j,k,b$ must form white triangles. 
(Assume $\{i,j,a\}$ is not a white triangle. Consider the polygon $Q$ bordered by the white triangles that contain the edges $\{i,a\}$, $\{j,a\}$, $\{i,j\}$. 
Then  $Q$ has either $3$ or $4$ edges. This is a contradiction. The other case is identical.)

Consider the complement of the polygon $P$ bordered by black triangles $\{i,a,c\}$, $\{k,b,c\}$ and the two black triangles adjacent to the edges $\{i,j\}$ and $\{j,k\}$.
Since $\cK$ is a generic triangulation, it must be that $P$ is either the complement of one of the hexagons $A'$ or $B'$ (in which case $P$ contains two 
vertices strictly in its interior, while $A'$, $B'$ do not; hence, a contradiction) or $P$ is the inside of one of the hexagons $A$ or $B$. This completely determines the 
triangulation; in case $A$ we must have $n=8$, while in case $B$ we have $n=9$. It is easy to see that in case $A$ this gives the bipyramid triangulation, and that case
$B$ cannot happen for an irreducible hypertree.

Assume now that $D$ is the outside of the hexagon $B'$. Since no two of $i,j,k$ can be connected by a black triangle inside the hexagon and since $i,j,k$ are inside the 
hexagon, it follows that the point strictly in the interior of $B'$ must be one of $i,j,k$. In cases $X$, $Y$ since $i,j,k$ belong to the removed triangles, which are in $D$, 
it follows that $i,j,k$ are on the boundary of $B'$, which is a contradiction. In case $Z$, note that since the valence of the interior point is $2$, the removed triangles must be the two black triangles inside $B'$, which is a contradiction.

\end{proof}

\begin{Review}
W. Thurston \cite{Th} suggests an approach for the classification of triangulations of the sphere based on hyperbolic geometry. 
Moreover, he gives a complete classification for triangulations with $v_i=2$ or $3$ for all vertices $i$.
It would be interesting to see how irreducible and generic triangulations fit in his classification.
\end{Review}

We have not tried to classify all non-spherical hypertrees.
It is easy to see that just choosing a random collection of triples
is not going to work: one of the results in the theory of random hypergraphs
is that they are almost surely disconnected. 
It is easy to see that disconnected hypertrees do not satisfy 
\eqref{CondS}. But perhaps one can enumerate all hypertrees inductively, using some
simple ``add a vertex'' procedures. Here is an example of such a construction.
The number of irreducible hypertrees produced this way 
grows very rapidly as $n$ goes to infinity.

\begin{Review}\textsc{Construction.}\label{Fibonacci}
Suppose that $\Gamma'$ is an irreducible hypertree on $N$ with triples only.
After renumbering, we can assume that $n$ belongs to only two triples,
namely to $\Gamma'_{n-3}$ and $\Gamma'_{n-2}$.
Suppose also that $n-1\in\Gamma'_{n-2}$.
We~define $n-1$ triples for $k=n+1$ as follows: $\Gamma_i:=\Gamma'_i$ for $i=1,\ldots,n-3$;
if $\Gamma'_{n-2}=\{i,n-1,n\}$ then we define $\Gamma_{n-2}:=\{i,n-1,n+1\}$;
and we define $\Gamma_{n-1}:=\{a,n,n+1\}$, where $a$ is any index in $N\setminus(\Gamma_{n-2}\cup\Gamma_{n-3})$, see Fig.~\ref{constructionne}.
\begin{figure}[htbp]
\includegraphics[width=4in]{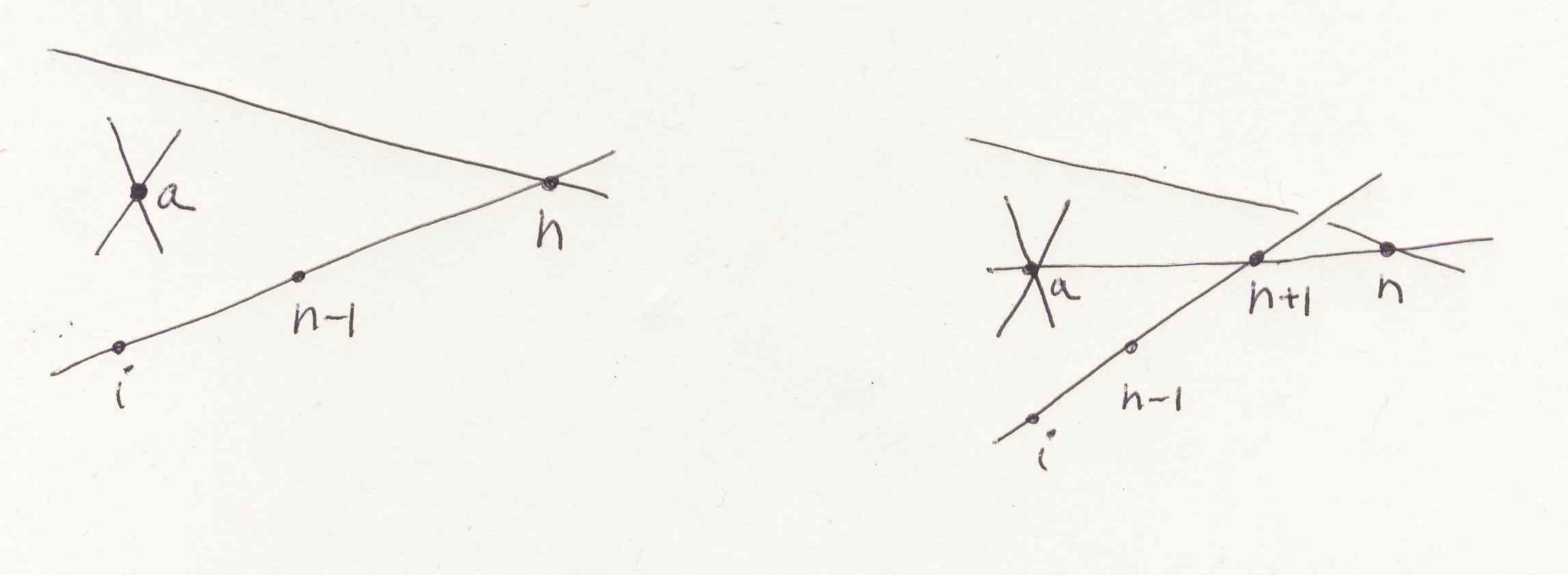}
\caption{}\label{constructionne}
\end{figure}
\end{Review}

\begin{Proposition}\label{FunnyConstruction} 
$\Gamma$ is an irreducible hypertree.
\end{Proposition}

\begin{proof}
Suppose $I\subset\{1,\ldots,n-1\}$, $1<|I|<n-1$.
Consider several cases. If $I\subset\{1,\ldots,n-3\}$ then 
$$|\bigcup_{i\in I}\Gamma_i|=|\bigcup_{i\in I}\Gamma'_i|\ge |I|+3,$$
and we are done.
If $I=I'\cup\{n-2\}$ (resp.~$I=I'\cup\{n-1\}$), where $I'\subset\{1,\ldots,n-3\}$, then
$$|\bigcup_{i\in I}\Gamma_i|\ge|\bigcup_{i\in I'}\Gamma'_i|+1,$$
because $n+1$ belongs to the first union but does not belong to the second union.
So again we are done  unless $|I'|=1$, in which case the claim is easy.

It remains to consider the case 
$I=I'\cup\{n-2, n-1\}$, where $I'\subset\{1,\ldots,n-3\}$ (and note that $|I'|<n-3$).
If $I'$ is empty then the claim is easy. Otherwise, let
$I''=I'\cup\{n-2\}$. Then $1<|I''|<n-2$ and so
$$|\bigcup_{i\in I''}\Gamma'_i|\ge|I''|+3.$$
But $$\bigcup_{i\in I}\Gamma_i\supseteqq\bigcup_{i\in I''}\Gamma'_i\sqcup\{n+1\}.$$
So $\Gamma$ is an irreducible hypertree.
\end{proof}
 
\begin{Lemma}
Let $\Gamma'$ be an irreducible hypertree composed of triples and let $\Gamma$ be the irreducible 
hypertree  obtained from $\Gamma'$ by the inductive construction \ref{Fibonacci}. If $\Gamma'$ is generic, then $\Gamma$ 
is generic.
\end{Lemma}

\begin{proof}
Let $\{i,j,k\}$ be a triple in $N\cup\{n+1\}$ that is not a triple, nor a wheel in $\Gamma$. Denote by $\tilde{\Gamma}_{\alpha}$
the triple $\Gamma_{\alpha}$ obtained by identifying $i,j,k$ with $p$, with the convention that we drop the $\tilde{\Gamma}_{\alpha}$'s
that are not triples. We prove that we can find $n-3$ triples $\tilde{\Gamma}_{\alpha}$ that satisfy (\ref{CondS}).

Consider the case when $n+1\in\{i,j,k\}$, say $k=n+1$: If $\{i,j\}$ is contained in some $\Gamma_{\alpha}$ 
for  $\alpha\in\{1,\ldots,n-3\}$, say $\Gamma_{n-3}$, then we take as our $n-3$ triples $\tilde{\Gamma}_1,\ldots,\tilde{\Gamma}_{n-4}$ and one of
$\tilde{\Gamma}_{n-2}$, $\tilde{\Gamma}_{n-1}$ (one that is a triple; for example, since  $\Gamma_{n-2}\cap\Gamma_{n-1}=\{n+1\}$ and
$\{i,j,n+1\}$ is not a wheel in $\Gamma$, one of $\Gamma_{n-2}$, 
$\Gamma_{n-1}$ does not contain $i,j$). If  $\{i,j\}$ is not contained in any of $\Gamma_1,\ldots, \Gamma_{n-3}$, we take 
$\tilde{\Gamma}_1,\ldots, \tilde{\Gamma}_{n-3}$ as our triples. The result follows from the fact that for any 
$T\subset\{1,\ldots,n-3\}$, since $n+1\notin \cup_{\alpha=1}^{n-3}\Gamma_\alpha$, we have
$$\big|\bigcup_{\alpha\in T}\tilde{\Gamma}_{\alpha}\big|\geq\big|\bigcup_{\alpha\in T}\Gamma_{\alpha}\big|-1\geq|T|+3-1=|T|+2.$$
Adding one of $\tilde{\Gamma}_{n-2}$, $\tilde{\Gamma}_{n-1}$ adds the index $n+1$ to the union, and condition (\ref{CondS}) is still satisfied.
The case when $n\in\{i,j,k\}$, $n+1\notin\{i,j,k\}$ is similar: we take $\tilde{\Gamma}_1,\ldots, \tilde{\Gamma}_{n-4}$ 
(note, $n\notin\cup_{\alpha=1}^{n-4}\Gamma_{\alpha}$) and $\tilde{\Gamma}_{n-1}$
as our triples.   

Consider now the case when $n, n+1\notin\{i,j,k\}$. Then $\{i,j,k\}$ is not a triple, nor a wheel in $\Gamma'$. Since $\Gamma'$ is a 
generic hypertree, there are $n-4$ triples from $\Gamma'$ which after identifying identifying $i,j,k$ with $p$ satisfy (\ref{CondS}).

If the $n-4$ triples are also triples in $\Gamma$ (i.e., $\Gamma'_{n-2}$ is not among them), then adding one $\Gamma_{n-1}$ to them
will do. Assume the contrary. Since $|\tilde{\Gamma'}_{n-2}|=3$, then $|\tilde{\Gamma}_{n-2}|=3$.
We claim that the remaining $n-5$ triples and $\tilde{\Gamma}_{n-2}$, $\tilde{\Gamma}_{n-1}$ will do the job. Let $T$ be a subset of
the $n-5$ remaining triples. Clearly, $\{\tilde{\Gamma}_{\alpha}\}_{\alpha\in T}$ satisfy  (\ref{CondS}). Adding one of 
$\tilde{\Gamma}_{n-2}$, $\tilde{\Gamma}_{n-1}$ to  $\{\tilde{\Gamma}_{\alpha}\}_{\alpha\in T}$ will not violate (\ref{CondS}). 
But $\{\tilde{\Gamma}_{\alpha}\}_{\alpha\in T}$, $\tilde{\Gamma}_{n-2}$, $\tilde{\Gamma}_{n-1}$ also satisfy (\ref{CondS}):
$$\big|\bigcup_{\alpha\in T}\tilde{\Gamma}_{\alpha}\cup\tilde{\Gamma}_{n-2}\cup\tilde{\Gamma}_{n-1}\big|\geq
\big|\bigcup_{\alpha\in T}\tilde{\Gamma}_{\alpha}\cup\tilde{\Gamma'}_{n-2}\cup\{n+1\}\big|\geq
(|T|+1)+2+1.$$
This finishes the proof.
\end{proof}

\section{Determinantal Equations} \label{DetEqs}

In this section we give simple determinantal equations of hypertree divisors in $\oM_{0,n}$
and then use them to show that black and white hypertrees of a spherical hypertree give the same divisor in $\oM_{0,n}$.

We consider only the case when hyperedges are triples. Fix a hypertree
$$\Gamma=\{\Gamma_1,\ldots,\Gamma_{n-2}\}$$
on the set $\{1,\ldots,n\}$. 
We work in ``homogeneous coordinates'' on $M_{0,n}$, i.e.,~we represent
a point of $M_{0,n}$ by $n$ roots $x_1,\ldots,x_n$ of a binary  $n$-form.

\begin{Proposition}
Let $A$ be an $(n-2)\times n$ matrix with the following rows\break (well-defined up to sign) : if $\Gamma_\alpha=\{i,j,k\}$ then 
$$A_{\alpha i}=x_j-x_k,\quad A_{\alpha j}=x_k-x_i,\quad A_{\alpha k}=x_i-x_j.$$
Then $D_\Gamma$ is given by the vanishing of any $(n-3)\times(n-3)$ minor of $A$
obtained by deleting a row and three columns with non-zero entries in that row.
\end{Proposition}


\begin{Example}
Consider the only hypertree for $n=7$ with hyperedges 
$$\Gamma=\{712,734,756,135,246\}.$$
Then we have
$$
A=\left[\begin{matrix}
x_2-x_7&x_7-x_1&0&0&0&0&x_1-x_2\cr
0&0&x_4-x_7&x_7-x_3&0&0&x_3-x_4\cr
0&0&0&0&x_6-x_7&x_7-x_5&x_5-x_6\cr
x_3-x_5&0&x_5-x_1&0&x_1-x_3&0&0\cr
0&x_4-x_6&0&x_6-x_2&0&x_2-x_4&0\cr
\end{matrix}\right]
$$
and $D_\Gamma$ is given by equation
$$-x_7^{2} {x_2} {x_3}+{x_7} {x_1} {x_2}
      {x_3}+x_7^{2} {x_1} {x_4}-{x_7} {x_1} {x_2}
      {x_4}-{x_7} {x_1} {x_3} {x_4}+{x_7} {x_2}
      {x_3} {x_4}+$$
$$x_7^{2} {x_2} {x_5}-{x_7} {x_1}
      {x_2} {x_5}-x_7^{2} {x_4} {x_5}+{x_1} {x_2}
      {x_4} {x_5}+{x_7} {x_3} {x_4} {x_5}-{x_2}
      {x_3} {x_4} {x_5}-x_7^{2} {x_1} {x_6}$$
$$+{x_7}
      {x_1} {x_2} {x_6}+x_7^{2} {x_3} {x_6}-{x_1}
      {x_2} {x_3} {x_6}-{x_7} {x_3} {x_4}
      {x_6}+{x_1} {x_3} {x_4} {x_6}+{x_7} {x_1}
      {x_5} {x_6}$$
$$-{x_7} {x_2} {x_5} {x_6}-{x_7}
      {x_3} {x_5} {x_6}+{x_2} {x_3} {x_5}
      {x_6}+{x_7} {x_4} {x_5} {x_6}-{x_1} {x_4}
      {x_5} {x_6}=0$$
\end{Example}

\begin{proof}
This is very simple. Fix different points $x_1,\ldots,x_n\in\bA^1$. The condition that these points
can be obtained by projecting a hypertree curve is as follows: there exist $y_1,\ldots,y_n\in\bA^1$
such that a triple of points 
$$(x_i,y_i),\ (x_j,y_j),\ (x_k,y_k)$$
lie on the line for any hyperedge $\Gamma_\alpha=\{i,j,k\}$.
This can be expressed by vanishing of the determinant
$$
-\det\left[\begin{matrix} 
	1	 & x_i  & y_i \cr
	1       & x_j  & y_j \cr
	1       & x_k & y_k \cr
\end{matrix}\right]=
y_i(x_j-x_k)+y_j(x_k-x_i)+y_k(x_i-x_j)=0.
$$
This gives a homogeneous system of linear equations on $y_i$'s with the matrix of coefficients $A$.
Notice that it has a $2$-dimensional subspace of trivial solutions (obtained by placing all points $p_i$ along some
line on the plane). Thus the condition that there exists a planar realization of $\Gamma$ with projections
$x_1,\ldots,x_n$ is given by vanishing of any non-trivial $(n-3)\times (n-3)$ minor.
For example, fix a row that corresponds to $\Gamma_\alpha=\{i,j,k\}$.
We can force $y_i=y_j=y_k=0$ as this points have to lie on a  line anyways.
Then we get a system of linear equations on the remaining $n-3$ variables and the condition
is that this system has a non-trivial solution. This gives a minor as in the statement of the theorem.
\end{proof}

These equations do not explain why black and white hypertrees of the spherical hypertree
yield the same divisor: their matrices $A_B$ and $A_W$ will be vastly different. For example, 
one can check that $\rk A_B\ne \rk A_W$ for some values of variables $x_1,\ldots,x_n$.      
Fortunately, $D_\Gamma$ has another determinantal equation.

\begin{Proposition}
Let $G$ be a $2$-dimensional simplicial complex
with simplices $\Gamma_1,\ldots,\Gamma_{n-2}$ (oriented arbitrarily).
Then $H_1(G,\bZ)=\bZ^{n-3}$. We choose a generating set of paths $P_1,\ldots,P_{m}$
(one can take $m=n-3$), where
$$P_i=(\alpha_i^1\to\alpha_i^2\ldots\to\alpha_i^r\to\alpha_i^{r+1}=\alpha_i^1)$$
is a path in  $\{1,\ldots,n-2\}$ such that $\Gamma_{\alpha_i^j}\cap\Gamma_{\alpha_i^{j+1}}\ne\emptyset$ for any $j$.
Consider an $m\times(n-2)$-matrix $B$ such that $B_{i\alpha}=0$ if $\alpha\not\in P_i$ and 
$B_{i\alpha}=x_k-x_l$ if $\alpha=\alpha_i^s\in P_i$, where $k=\Gamma_{\alpha_i^s}\cap\Gamma_{\alpha_i^{s+1}}$ and
$l=\Gamma_{\alpha_i^s}\cap\Gamma_{\alpha_i^{s-1}}$.
Then $D_\Gamma$ is given by vanishing of any non-trivial $(n-3)\times(n-3)$-minor of $B$.
\end{Proposition}

\begin{Corollary}\label{coinciwb}
Black and white hypertrees of a spherical hypertree give the same divisor on $\oM_{0,n}$.
\end{Corollary}

\begin{proof}[Proof of the Corollary]
Notice that $G$ is just the ``black'' part of the bi-colored triangulation of the sphere.
We orient all black and white triangles according to an orientation of the sphere.
We can choose a generating set of $H_1(G,\bZ)$ to be given by cycles around white triangles.
Then $B$ has the following very simple form. It has rows indexed by white triangles, columns
indexed by black triangles. The entry $B_{wb}$ is equal to zero if triangles $w$ and $b$
do not share an edge, and $B_{wb}=x_i-x_j$ if $w$ and $b$ intersect along the edge $[ij]$
(oriented according to the orientation of $b$). Now notice that the corresponding matrix
for the white hypertree is just the minus transposed matrix $-B^t$.
\end{proof}

\begin{Example}
Consider the spherical hypertree from Example~\ref{10vertexexample}.
The matrix $B$ looks as follows (we skip zero entries):
$$\left[\begin{matrix}
x_1-x_3&&&x_2-x_1&&&&x_3-x_2\cr
x_4-x_1&&x_8-x_4&x_1-x_8&&&&\cr
x_3-x_4&x_5-x_3&x_4-x_5&&&&&&\cr
&&x_5-x_8&&x_8-x_{10}&x_{10}-x_5&&&\cr
&x_6-x_5&&&&x_5-x_9&x_9-x_6&\cr
&&&&x_{10}-x_7&x_9-x_{10}&x_7-x_9&\cr
&x_3-x_6&&&&&x_6-x_7&x_7-x_3\cr
&&&x_8-x_2&x_7-x_8&&&x_2-x_7\cr
\end{matrix}
\right]$$

\end{Example}

\begin{proof}[Proof of the Proposition]
This is similar to the proof of the previous Proposition: we fix $x_1,\dots,x_n\in\bA^1$ but now instead of using
$y_1,\ldots,y_n$ as variables, we use slopes of the lines $k_1,\ldots,k_{n-2}$ as variables.
To reconstruct the planar realization, we choose a height $y_1$ arbitrarily and then consecutively compute the remaining ``heights'' $y_i$:
if $y_i$ is already known and $i$ and $j$ are connected by a line with slope $k_\alpha$ then of course
$$y_j=y_i+k_\alpha(x_j-x_i),$$
which gives $y_j$. We only have to show that heights of points thus obtained do not depend on a sequence of lines
that connects them to the first point. But this precisely means that for each cycle
of lines $$\alpha^1\to\alpha^2\ldots\to\alpha^r\to\alpha^{r+1}=\alpha^1,$$
the relative height of the last point with respect to the first point is equal to $0$,
i.e.,~that slopes satisfy the system of linear equations with matrix $B$.
Throwing away a trivial solution when all slopes are equal, we get the equation of $D_\Gamma$
as a minor of $B$ of codimension~$1$ (note that rows and columns of $B$ add to zero.)
\end{proof}

\section{Hypertrees and Brill--Noether Divisors on $\oM_g$}\label{HistoricalRemarks}

Consider the Keel--Vermeire divisor on $\oM_{0,6}$.
According to our description, $D_\Gamma$ is the locus of projections of vertices of the 
complete quadrilateral. This is a spherical hypertree with
the triangulation given by an octahedron.
There are two hypertrees (black and white)
that give the same divisor.
The total number of  Keel--Vermeire divisors on $\oM_{0,6}$ is $15$.
They are parameterized by markings of the octahedron, i.e.,~by tri-partitions of $\{1,\ldots,6\}$ into pairs.
For example, Figure~\ref{octahedron} corresponds to a $3$-partition $(12)(34)(56)$.

\begin{figure}[htbp]
\includegraphics[width=4in]{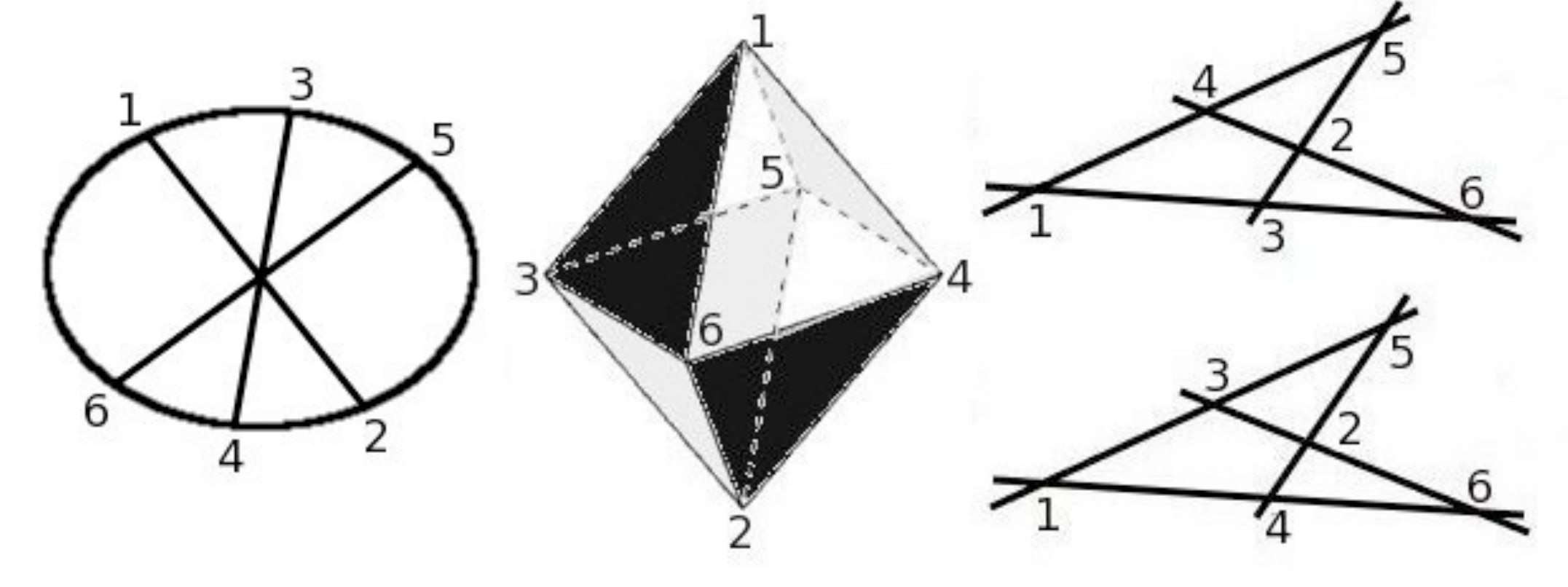}
\caption{The Keel-Vermeire divisor in $\oM_{0,6}$}
\label{octahedron}
\end{figure}

Now let us explain the left-hand-side of Figure~\ref{octahedron}.
For any tri-partition, consider the map $\oM_{0,6}\to\oM_3$ obtained by gluing
points in pairs
$$
\begin{matrix}
M_{0,6} & \to &\oM_3\cr
\includegraphics[height=0.3in]{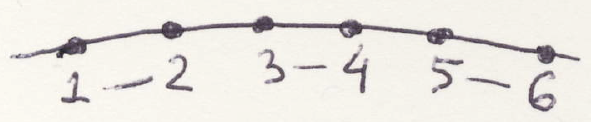}&\raise0.1in\hbox{$\mapsto$}&\includegraphics[height=0.3in]{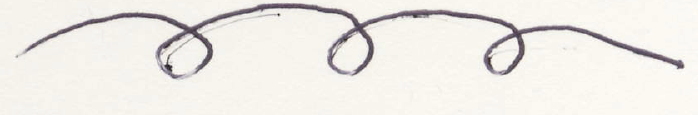}
\end{matrix}
$$

Keel defined a divisor $D_K\subset M_{0,6}$ as the pull-back of the 
hyperelliptic locus in~$\oM_3$. This locus is divisorial. 
By the theory of admissible covers~\cite{HM}, a hyperelliptic involution
on the general genus $3$ curve in the limit induces an involution 
of $\bP^1$ that exchanges points in the pairs $(12)$, $(34)$, and $(56)$.
Quotient by this involution is a degree $2$ map $\bP^1\to\bP^1$,
which can be realized by embedding $\bP^1$ in $\bP^2$
as a plane conic and projecting it from a point.
It follows that $D_K\subset M_{0,6}$ is the locus of $6$ points on a conic such that
chords connecting pairs of points $(12)$, $(34)$, and $(56)$ are concurrent. 

It is quite amazing that these two descriptions give the same divisor:

\begin{Proposition} $D_K=D_\Gamma$.
\end{Proposition}

\begin{proof}
Passing to the projectively dual picture, let
$A_1,A_2,A_3,A_4\in\bP^2$ be general points
and let $L\subset\bP^2$ be a general line. Let $\{L_{ij}\}$ be $6$ lines connecting pairs of points $A_i$, $A_j$. 
The claim is that there exists an involution of $L$ that permutes $L\cap L_{ij}$ and $L\cap L_{i'j'}$ if
$\{i,j\}\cup\{i',j'\}=\{1,2,3,4\}$. More precisely, we prove that $D_\Gamma\subset D_K$. Since $D_K$ 
is an irreducible divisor (this is easy to see by the above description), the Proposition follows. 

The proof is illustrated in Figure~\ref{CremonaProof}.
Let $T:\,\bP^2\dra\bP^2$ be the standard Cremona transformation with the base locus $\{A_1,A_2,A_3\}$.
Then $T$ contracts lines $L_{23}$, $L_{13}$, and $L_{12}$  to points $A_1',A_2',A_3'$. Let $A_4'=T(A_4)$.

\begin{figure}[htbp]
\includegraphics[width=4in]{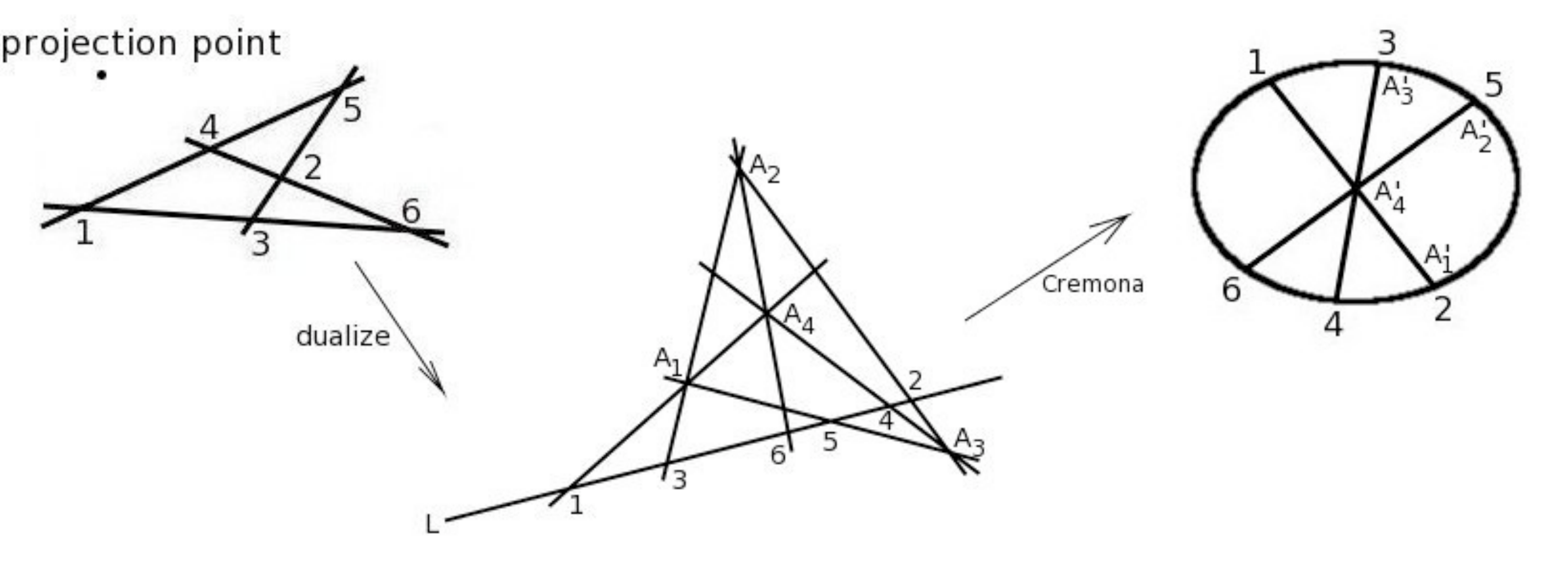}
\caption{$D_K=D_\Gamma$}
\label{CremonaProof}
\end{figure}
Notice that $T(L)=C$ is a conic that passes through $A_1',A_2',A_3'$.
These points are images of points $L\cap L_{23}$, $L\cap L_{13}$, and $L\cap L_{12}$, respectively.
For any $i=1,2,3$, the map $T$ sends the line $L_{i4}$ to the line that passes through $A_i'$ and $A'_4$.
So the diagonals connecting $A_i'$ to $T(L_{i4}\cap L)$ are concurrent.
\end{proof}

\begin{Remark}
The equation of $D_K$ was found by Joubert \cite[1867]{J}.
A point of $M_{0,6}$ is given by $6$ roots
$x_1,\ldots,x_6$ of a binary sextic. 
Put them on the Veronese conic. This gives 
$6$ points $p_i = (1,x_i,x_i^2)$.
The equation of the line $\langle p_i,p_j\rangle$ is
$${1\over x_j-x_i}\det\left[\begin{matrix} X & Y &  Z\cr
1 & x_i & x_i^2\cr 1 & x_j  & x_j^2\cr\end{matrix}\right]=X(x_ix_j)-Y(x_i+x_j)+Z = 0.$$
The condition that the three lines are concurrent is
\begin{equation}
\det\left[\begin{matrix} 
	x_1x_2	 & x_1+x_2 & 1 \cr
	x_3x_4  & x_3+x_4 & 1 \cr
	x_5x_6  & x_5+x_6 & 1 \cr
\end{matrix}\right]=0.
\cooltag\label{ddd}\end{equation}	
After some calculations, this gives
\begin{equation}
(14)(36)(25)+(16)(23)(45)=0,
\cooltag\label{trieq}
\end{equation}
where we use the classical bracket notation $(ij)=x_i-x_j$.
The equation for~$D_\Gamma$ is of course the same, see~\ref{DetEqs} and~\cite[p.~93]{St}.
\end{Remark}

\begin{Remark}
In fact $D_K$ was known earlier.
Cayley \cite[1856]{C} studied Hilbert functions
of graded algebras (using a different language) and computed the Hilbert function
of the algebra of invariants  of binary sextics:
$$h\left(k[\Sym^6k^2]^{\SL_2}\right)={1-x^{30}\over (1-x^2)(1-x^4)(1-x^6)(1-x^{10})(1-x^{15})}.$$ 
This lead him to the (correct) 
prediction that this algebra 
is generated by invariants $A,B,C,D,E$ of degrees $2,4,6,10,15$ with a single relation
$$E^2=f(A,B,C,D)$$ 
for some polynomial $f$. Salmon \cite[1866]{S} computed these invariants  
and proved (page 210) that $E$ has very simple meaning:
$E=0$ if and only if roots of the sextic are in involution!
We are not specifying a tri-partition here, so any of the 15 tri-partitions can occur.
Salmon computes (page~275) an expression of $E$ in
terms of roots of the sextic: $E$ is a product of $15$ determinants \eqref{ddd},
one determinant for each tri-partition $(ij)(kl)(mn)$.
\end{Remark}

One can ask if there are other hypertree divisors with similar ``dual'' meaning
as pull-backs of Brill--Noether (or perhaps Koszul) divisors on~$\oM_g$.
We will show that this is so
for the easiest spherical hypertree one can draw: the bipyramid.
We will leave it to the reader to find further examples.

Let $n=2k+2$. A hypertree curve is illustrated in Figure~\ref{bipyramid} (for $n=12$).
\begin{figure}[htbp]
\includegraphics[width=4in]{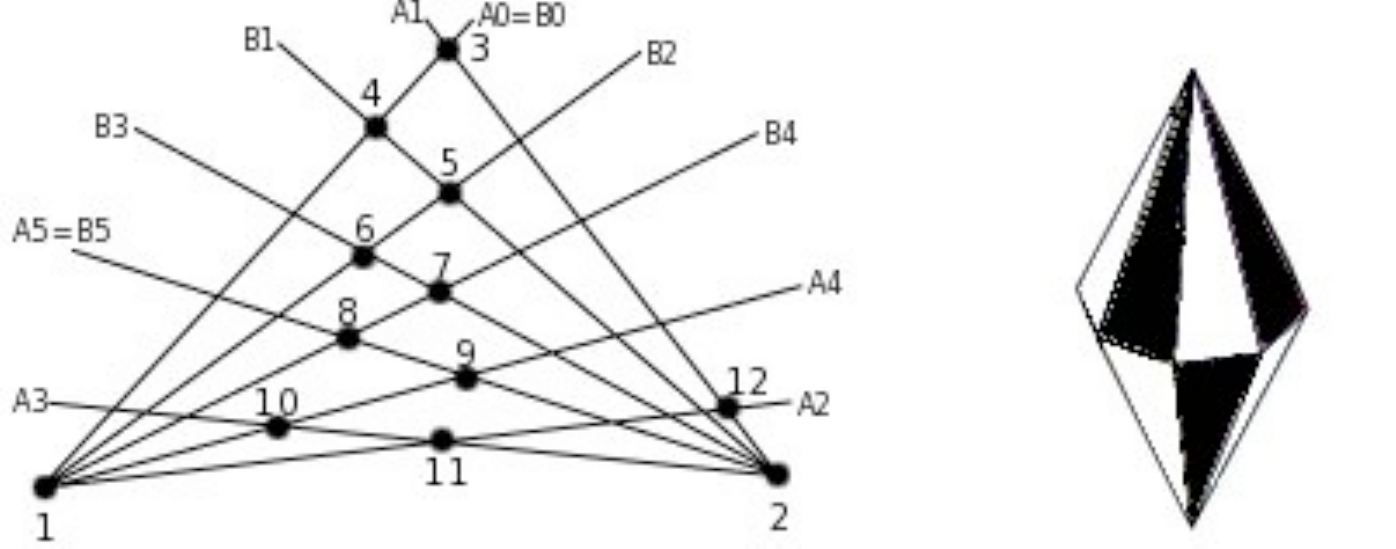}
\caption{Bipyramid hypertree for $n=12$}
\label{bipyramid}
\end{figure}
We label lines by $A_0=B_0$,\quad $A_i$, $B_i$ ($i=1,\ldots,k-1$),\quad and $A_k=B_k$.
The labels are chosen so that the point $1$ (resp., the point $2$) belongs to the lines $A_i$, $B_i$ for $i$
even (resp., odd) and such that the points $4,5,\ldots,3+(n-2/2)$ are obtained by intersecting lines 
$$A_0=B_0\to B_1\to\ldots\to B_k=A_k$$
while the points $4+(n-2/2), \ldots, n, 3$ 
are obtained by intersecting
$$B_k=A_k\to A_{k-1}\to\ldots\to A_0=B_0.$$

The bipyramid determines a tri-partition 
$$\{1,\ldots,n\}=\{1,2\}\cup X\cup Y$$ 
into two poles (in our example $1$ and $2$) and two ``alternating'' subsets $X$ and $Y$ of the equator
with $|X|=|Y|=k$.
In our example $k=5$,
$$X=\{3,5,7,9,11\}\quad\hbox{\rm and}\quad Y=\{4,6,8,10,12\}.$$

Let $D_\Gamma\subset\oM_{0,n}$ be the corresponding hypertree divisor.

\begin{Definition}
Consider the map
$$
\begin{matrix}
M_{0,2k+2} & \arrow^\psi &\oM_{2k-1}\cr
\includegraphics[height=0.5in]{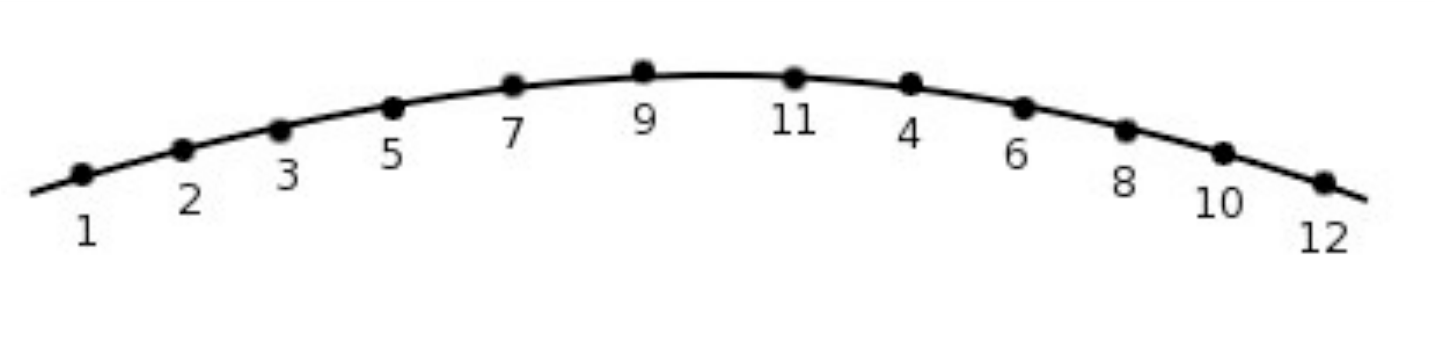}&\raise0.2in\hbox{$\mapsto$}&\includegraphics[height=0.5in]{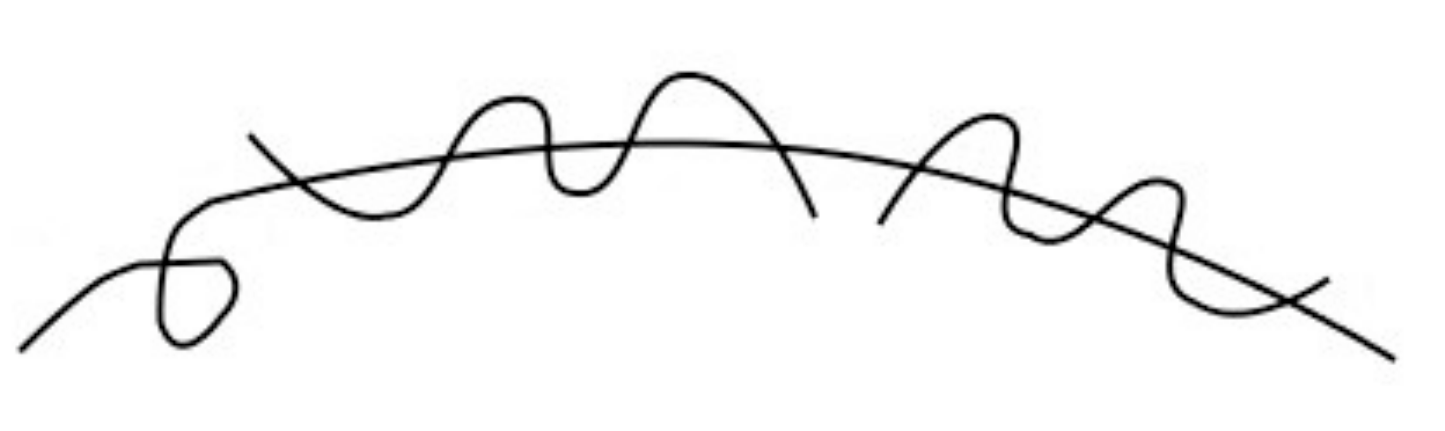}
\end{matrix}
$$
obtained by gluing the poles of $\bP^1$ and then gluing to it two copies of $\bP^1$ with $k$ marked points on each
one 
along points of parts $A$ and $B$ of the tri-partition.
Let $D_K\subset M_{0,2k+2}$ be the pull-back of the Brill--Noether divisor in $\oM_{2k-1}$ that parameterizes $k$-gonal 
curves\footnote{If $k\ge4$ then the attached $\bP^1$'s (and hence the map $\psi$) are not uniquely defined. However,
we will see that $D_K$ does not depend on any choices.}. 
\end{Definition}

\begin{Proposition} $D_K=D_\Gamma$.
\end{Proposition}

\begin{Remark}
It is well-known that the Brill-Noether divisor on $\oM_{2k-1}$ is extremal for $k\leq5$. However, for $k\geq6$ the 
Brill-Noether divisor is not extremal \cite{Gabi}. The Proposition shows that nevertheless, the proper transform of
the Brill-Noether divisor via the map $\psi$ is an extremal divisor on $\oM_{0,2k+2}$. See also Remark \ref{restrictions of contractions}
and the examples in Section \ref{pull-backs}.
\end{Remark}

\begin{proof}
Using theory of admissible covers, we can identify $D_K$ with a locus in $M_{0,n}$ such that the corresponding
$\bP^1$ with $n$ marked points admits a $g^1_k$ with members $X$, $Y$, and $Z$ such that $1,2\in Z$.
In other words, $D_K$ parameterizes $n$-tuples $\{p_1,\ldots,p_n\}$ of points on a rational normal curve 
$$C\subset\bP^k$$
such that 
\begin{equation}\label{tripartway}
\langle p_1,p_2\rangle\cap\langle p_i\rangle _{i\in X}\cap \langle p_i\rangle _{i\in Y}\ne\emptyset.
\cooltag\end{equation}

It is not hard to see that $D_K$ is irreducible ($D_K$ can be parametrized by an open
in $(\PP^1)^{n-1}$, as the markings $p_i$ for $i=1,\ldots, n-1$ determine $p_n$).
So it remains to show that $D_\Gamma\subset D_K$.
Consider $n$ points $p_1,\ldots,p_n\in\bP^1$ obtained by projecting vertices of a hypertree from Figure~\ref{bipyramid} 
(assume all lines $A_i, B_i$ are distinct).
We claim that if we put these points on a rational normal curve $C\subset\bP^k$, the condition
\eqref{tripartway} is going to be satisfied.

In a projectively dual plane, we get a configuration of $n$ lines in $\bP^2$ depicted in Figure~\ref{baraban}.
\begin{figure}[htbp]
\includegraphics[width=4in]{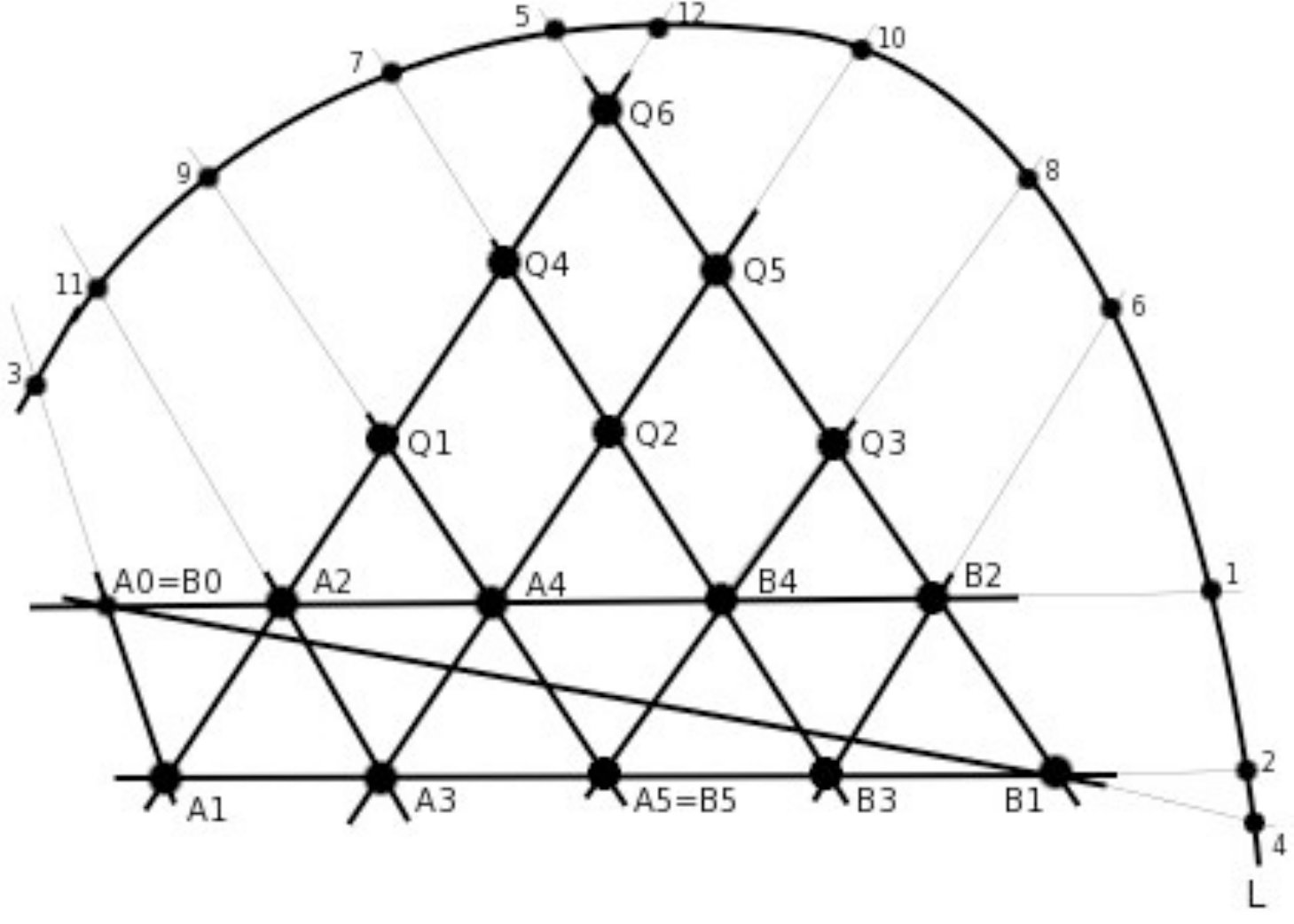}
\caption{Dual configuration of a bipyramid ($n=12$)}
\label{baraban}
\end{figure}
Let us explain what's new in this picture. The line $L$ is projectively dual to the focus of projection
(we draw $L$ as a curve because we are about to identify it with a rational normal curve in $\bP^k$).
The definition of points $Q_1,\ldots,Q_{(k-1)(k-2)\over 2}$ is clear from the picture.

Let $S$ be the blow-up of $\bP^2$ in points $A_i$ ($i>0$), $B_i$ ($i>0$), and all $Q_i$. 
We do not blow-up $A_0=B_0$, so basically we blow-up a ``triangular number'' of points arranged in the triangular grid.
One has to be slightly careful though because this arrangement of lines has moduli,
and in particular there are no ``horizontal'' lines containing points in the grid other than lines $1$ and $2$.
For example, there is in general no line containing $Q_1$, $Q_2$, and $Q_3$.

Consider a divisor
$$D=kH-A_1-\ldots-A_k-B_1-\ldots-B_{k-1}-Q_1-\ldots-Q_{(k-1)(k-2)\over 2}$$
on $S$. The following effective divisors are linearly equivalent to $D$:
$$D_1=L_3\cup\ldots\cup L_{2k+1}\quad\hbox{\rm and}\quad D_2=L_4\cup\ldots\cup L_{2k+2},$$
where $L_i$ is a proper transform of a line number~$i$.
It follows that the linear system $|D|$ has no  fixed components, and therefore it defines 
a rational map
$$\Psi:\,S\dra\bP^k$$
regular outside of points of intersection of $D_1$ and $D_2$.
In fact $\Psi$ is regular at $A_0=B_0$ because $|D|$ also contains
$$
L_2\cup L_6\cup L_8\ldots\cup L_{2k+2}
$$
which does not contain $A_0=B_0$. The following argument proves that the dimension of the linear system $|D|$ is $k$ and that 
the restriction of $|D|$ to $L$ cuts a complete linear system: If $|D|$
contains a member $\tilde D$ that contains $L$ as a component, a simple analysis using Bezout theorem 
shows that $\tilde D\supset D_1\cup D_2$, which is impossible
(see a similar analysis below). 

We see that $\Psi(L)=C\subset\bP^k$ is a rational normal curve.
Notice that hyperplanes $\langle p_i\rangle _{i\in X}$ and $\langle p_i\rangle _{i\in Y}$ are cut out by divisors $D_1$ and $D_2$
and that these divisors have another point in common on $S$, namely $A_0=B_0$.
Finally, let's consider the line $\langle p_1,p_2\rangle$. Any hyperplane containing this line
corresponds to a divisor $\tilde D$ in $|D|$ that contains points $1$ and $2$. By Bezout theorem,
$\tilde D$ contains the line $L_2$, and then the residual divisor $\tilde D- L_2$ contains the line $L_1$
(again by Bezout theorem).
It follows that $\tilde D$ also contains a point $A_0=B_0$ and therefore
$$
\Psi(A_0)\in\langle p_1,p_2\rangle\cap\langle p_i\rangle _{i\in A}\cap \langle p_i\rangle _{i\in B}.
$$
QED
\end{proof}

The bipyramid divisor
is very exceptional for its symmetries: the symmetry group of a bipyramid 
is a binary dihedral group $\tilde D_k$ but the corresponding 
divisor in $\oM_{0,2k+2}$ is a pull-back from $\oM_{0,2k+2}/S_2\times S_k\times S_k$.

\section{Pull-backs of divisors from $\oM_g$}\label{pull-backs}

We will consider pull-backs of several ``geometric" divisors on $\oM_{g,k}$ (for special values of $g$ and $k$) via 
maps $$\rho:\oM_{0,n}\ra\oM_{g,k}, \quad n=2g+k$$ obtained by identifying $g$ pairs of markings on a rational stable curve. 
We give some evidence that in general this does not lead to any new interesting divisors on $\oM_{0,n}$. This is in 
contrast with the case $n=6$ (when we obtain the Keel-Vermeire divisors) 
and the cases in Section \ref{HistoricalRemarks} (where we pull-back via different types of maps).

\begin{Remark}\label{restrictions of contractions}
In our examples we will pull back divisors $E$ on $\oM_{g,k}$ which are extremal in $\Eff(\oM_{g,k})$. 
Moreover, in the examples \ref{ex1} and \ref{ex2} the divisor $E$ is contracted to a point by a birational 
contraction $\oM_{g,k}\dra Y$ \cite{Jensen}. It is natural to ask whether the restriction
to $\oM_{0,n}$ is still a birational contraction onto the image, which would imply that the components of $\rho^*E$  are
extremal in $\Eff(\oM_{0,n})$. This turns out not to be  true in general, as in our examples we prove that the pull-back of $E$ to 
$\oM_{0,n}$ is not extremal.
\end{Remark}

\begin{Review}
We will consider proper transforms of divisors $E$ in $\oM_{0,n}$ via $\rho$:
\begin{equation}\cooltag\label{D}
D=\ov{\rho^{-1}(E)\cap M_{0,n}}.
\end{equation}

The main reason for considering the proper transform is to avoid having boundary components contained in the pull-back $\rho^*E$. 
(This actually happens: see Example \ref{ex2} for an instance of this.)

To compute the class of the divisor $D$ we will consider its pull-back to the Fulton-MacPherson configuration space \cite{FM}. 
Denote by $\PP^1[n]$ (resp., $\bA^1[n]$), the Fulton-MacPherson space of $n$ points on $\PP^1$ (resp., on $\bA^1$).
The space $\PP^1[n]$ is isomorphic to the Kontsevich moduli space \cite{FP} of stable maps $\oM_{0,n}(\PP^1,1)$. 
There are forgetful morphisms: 
$$\tilde{\phi}:\PP^1[n]\ra\oM_{0,n},\quad \phi:=\tilde{\phi}_{|\bA^1[n]}: \bA^1[n]\ra\oM_{0,n}.$$
\end{Review}

\begin{Notation}
Let $D_I$ ($|I|\geq2$) be the boundary divisor whose general point corresponds to a stable map $f: C\ra\PP^1$, with $C$ a 
rational curve with two components $C_1$, $C_2$, with markings from $I$ on $C_1$ and markings from $I^c$ on 
$C_2$ and such that $f$ has degree $0$ on $C_1$ and  degree $1$ on $C_2$.
\end{Notation}

\begin{Lemma}\label{FM space}
A divisor class $D$ on $\oM_{0,n}$ is determined by its pull-back $\phi^*D$ as follows:
If  in the Kapranov model of $\oM_{0,n}$ with respect to the $n$-th marking we have
$D=dH-\sum m_I E_I$, then 
$$\phi^*D=-dD_{\{1,\ldots, n-1\}}-\sum_{3\leq|I|<n-1, n\not\in I} m_{\{1,\ldots,n-1\}\setminus I}D_I+ (\hbox{ sum of } D_I \hbox{ with } n\in I)$$
\end{Lemma}

\begin{proof}
This is an easy calculation using the relations between the 
boundary divisors $D_I$ (see \cite{FM}) and the fact that $\tilde{\phi}^*\de_I=D_I+D_{I^c}$.  
Since the divisors $$\{D_I\}_{|I|\geq3, n\not\in I}, \quad \{D_I\}_{|I|\geq2, n\in I}$$
are linearly independent, it follows that $D$ is determined by $\phi^*D$.
\end{proof}

\begin{Review}\label{E'}
Let $x_1,\ldots, x_n$ be coordinates on $\bA^n$ and let:
$$U=\bA^n\setminus \bigcup_{i\neq j}(x_i=x_j).$$
Clearly, we have $\phi(U)\subset M_{0,n}$. 
If $D$ is an effective divisor on $\oM_{0,n}$ such that all components of $D$ intersect $M_{0,n}$, then 
all components of $\phi^*D$ intersect $U$. We will consider divisors $D$ as in (\ref{D}).
Hence, any component of $\phi^*D$ is a component of the closure
$$E'=\ov{\phi^{-1}(\rho^{-1}(E))\cap U}$$

In the cases that we study (Examples \ref{ex1} and \ref{ex2}) the divisor $E'$ is 
irreducible and linearly equivalent to a sum of boundary divisors.

The space $\bA^1[n]$ can be described as the blow-up of $\bA^n$ along diagonals in increasing order of dimension.
The divisor $D_I$ is the exceptional divisor corresponding to the diagonal  
$\Delta_I$ where coordinates $x_i$ for $i\in I$ are equal. 

If $F$ is an irreducible polynomial in $x_1,\ldots,x_n$ defining 
$E'$ in $U$ and having multiplicity $n_I$ along the $\Delta_I$,
the class of $E'$ in $\bA^1[n]$  is given by:
$$E'=-\sum n_I D_I.$$
\end{Review}

\begin{Example}\label{ex1}
Let $\rho:\oM_{0,7}\ra\oM_{3,1}$ be the map that identifies pairs of markings $(12)$, $(34)$, $(56)$. Let 
$E$ be the closure  in $\oM_{3,1}$ of the locus corresponding to pairs $(C,p)$ with $C$ a smooth genus $3$ curve 
and $p\in C$ a Weierestrass point. Every smooth curve $C$ has a finite number of Weierstrass points; hence, $E$ is a 
divisor in $\oM_{3,1}$. 

Consider now an integral nodal curve $C$. Let $w_C$ be the dualizing sheaf  and let
$w_1,\ldots, w_g$ be a basis for $\H^0(C,w_C)$. Locally at a smooth point $p$ of $C$
we have $w_i=f_i(t)dt$, where $t$ is a local parameter at $p$, $f_i$ is a regular function.
Just as for smooth curves, the point $p$ is called a Weierstrass point if and only if the 
Wronskian of the functions $f_1,\ldots, f_g$ vanishes at $p$. 
(See \cite{LW} for the general definitions of Weierstrass points on singular curves.)

Let $(C,p)\in\rho(M_{0,7})$. As there can only be finitely many Weierstrass points on $C$, it follows that 
$C$ has a Weierstrass point at $p$ if and only if $(C,p)$ belongs to $E$.

Let $t$, $x_i,y_i$ ($1\leq i\leq3$) be coordinates on $\bA^7$, with the pairs ($x_i$, $y_i$) the markings 
that get identified. The locus $E'\cap U$ (see \ref{E'}) parametrizes data $$(\bA^1,x_1,x_2, x_3, y_1,y_2, y_3,t)$$ 
for which  the corresponding tri-nodal curve has a Weierstrass point at $t$. 

Consider the basis of (local) differentials $f_1dt$, $f_2dt$, $f_3 dt$, where
$$f_i(t)=\frac{1}{(t-x_i)(t-y_i)}\quad (1\leq i\leq 3).$$
The divisor $E'$ is defined in $U$ by the vanishing of the Wronskian:
$$
\left|\begin{matrix}
f_1(t)&f_2(t)&f_3(t)\cr
f_1'(t)&f_2'(t)&f_3'(t)\cr
f_1''(t)&f_2''(t)&f_3''(t)\cr
\end{matrix}\right|.
$$

Moreover, it is easy to see that the determinant doesn't change if we replace the rational functions $f_i$ with 
the polynomials:
\begin{equation}\cooltag\label{canonical coordinates}
g_i(t)=f_i(t)\prod_{i=1}^3(t-x_i)(t-y_i).
\end{equation}

Using the program Macaulay we compute the multiplicities of this determinant along the 
diagonals $\Delta_I\subset \bA^{7}$. Using Lemma \ref{FM space}, we compute the class of the divisor 
$D=\phi(E')\subset\oM_{0,7}$ in the Kapranov model with respect to the $7$-th marking to be:
$$ 
D=3H-\sum_{i=1}^6 E_i-E_{12}-E_{34}-E_{56}.
$$

The divisor  $D$ is clearly big, as $H$ is big and we have $D=H+Q$, where $Q$ is the proper transform of any 
quadric in $\PP^4$ that contains the lines determined by the pairs of points $(12)$, $(34)$, $(56)$.
\end{Example}

\begin{Example}\label{ex3}
Consider the map $\rho:\oM_{0,7}\ra\oM_{3,1}$ as in Example \ref{ex1}.
Let $E$ be the closure  in $\oM_{3,1}$ of the locus corresponding to pairs $(C,p)$ with $C$ a 
smooth genus $3$ curve and $p\in C$ a point on a bitangent, i.e., $w_C=\O_C(2p+2q)$ for some $q\in C$.

Let $t$, $x_i,y_i$ ($1\leq i\leq3$) be coordinates on $\bA^7$, with the pairs ($x_i$, $y_i$) the markings 
that get identified. The locus $E'\cap U$ (see \ref{E'}) parametrizes data $$(\bA^1,x_1,x_2, x_3, y_1,y_2, y_3,t)$$ 
for which  the corresponding tri-nodal curve has the tangent at $t$ tangent at some other point. Consider the canonical 
map $$g:\bA^1\ra\PP^2,\quad t\mapsto (g_1(t), g_2(t), g_3(t)),$$
where $g_1,g_2,g_3$ are as in (\ref{canonical coordinates}). Consider the function in $t,s\in\bA^1$:
$$m(t,s)=\frac{1}{(s-z)^2}
\left|\begin{matrix}
g_1(t)&g_2(t)&g_3(t)\cr
g_1(s)&g_2(s)&g_3(s)\cr
g_1'(t)&g_2'(t)&g_3'(t)\cr
\end{matrix}\right|.
$$

For a fixed $t\in\bA^1$, the equation $m(t,s)=0$ computes the points of intersection of the tangent line at $t$ 
with the curve $g(\bA^1)$. The function $m(t,s)$ is quadratic in $s$:
$$m(t,s)=As^2+Bs+C,\quad A=\frac{\partial^2 m}{\partial^2 s}(t,0),\quad B=\frac{\partial m}{\partial s}(t,0),\quad C=m(t,0)$$

The divisor $E'$ is defined in $U$ by the vanishing of the discriminant 
$$\De=B^2-4AC.$$

As in Example \ref{ex1}, using the program Macaulay and Lemma \ref{FM space}, we compute the class of the divisor 
$D=\phi(E')\subset\oM_{0,7}$ (in the Kapranov model with respect to the $7$-th marking) to be:
$$ 
D=8H-4\sum_{i=1}^6 E_i-2\sum_{i\neq j\in\{1,\ldots,6\}} E_{ij}-2E_{123}-2E_{456}.
$$

The divisor $D$ is clearly big, as $H$ is big and we have 
$$D=2(H+\de_{ij}+\de_{lm}+\de_{kn}),$$ 
for any $\{i,j,k\}=\{1,2,3\}$, $\{l,m,n\}=\{4,5,6\}$.
\end{Example}

\begin{Example}\label{ex2}
Let $\rho:\oM_{0,10}\ra \oM_5$ be the map that identifyies the pairs of points $(11'), (22'), (33'), (44'), (55')$. Let
$E$ be the Brill-Noether divisor of trigonal curves in $\oM_5$. By the calculations in \cite{HM}, \cite{EH}, 
the class of $E$ is equal to
$$8\lambda-\delta_{irr}-4\delta_1-6\delta_2.$$
Using standard formulas for pull-backs of tautological classes, it is easy to compute
the class of its pull-back $\rho^*E$.
To preserve all symmetries of this divisor in the notation, we give the formula for the pull-back
$\pi_{11}^*\rho^*E$ to $\oM_{0,11}$ in the Kapranov model with respect to 
the $11$-th marking. The class is given by  
$$
\pi_{11}^*\rho^*E=20H-16\sum E_1- 12\sum E_{12} - 12\sum E_{11'}-9\sum E_{123}-$$
$$-6\sum E_{121'}-7\sum E_{1234}-4\sum E_{1231'}-6\sum E_{121'2'}-$$
$$-6\sum E_{12345}-3\sum E_{12341'}-$$
$$-3\sum E_{123451'}-2\sum E_{1231'2'3'}-\sum E_{123451'2'}+2\sum E_{12341'2'3'},$$
To explain the notation, the sums are taken over all permutations that preserve the number of pairs from 
$(11'), (22'), (33'), (44'), (55')$. For example:
$$\sum E_1=\sum_{i=1}^5 E_i+\sum_{i=1}^5 E_{i'},\quad\sum E_{11'}=\sum_{1\leq i\leq5 } E_{ii'},$$ 
$$\sum E_{12}=\sum_{i\neq j, 1\leq i,j\leq5 } E_{ij}+E_{i'j'}+E_{ij'},$$
$$\sum E_{123}=\sum_{1\leq i,j,k\leq5, i\neq j,k, j\neq k} E_{ijk}+E_{i'j'k'}+E_{ijk'}+E_{ij'k'},\quad\hbox{etc.}$$

It is clear from this formula that $\rho^*E$ is reducible and contains the main component (that intersects $M_{0,10}$)
as well as some boundary divisors. These boundary divisors can be easily determined using the method
of admissible covers, but computing the corresponding multiplicities is a bit subtle.

As in the previous examples, we can nevertheless compute the class of the main component. 
Let 
$$x_1,\ldots, x_5,y_1,\ldots, y_5$$
be the coordinates on $\bA^{10}$, with the pairs ($x_i$, $y_i$) the markings that get identified. 
Using the theory of admissible covers, the locus $E'\cap U$ (see \ref{E'}) parametrizes data $(\bA^1,x_1,\ldots, x_5, y_1,\ldots, y_5)$ for which  the 
corresponding nodal curve in $\oM_5$ has a $g^1_3$, i.e., the chords determined by the pairs $(x_i, y_i)$ have a common 
transversal when considering $\bA^1\subset\PP^3$ as a twisted cubic  via the Veronese embedding.
Consider the Grassmannian of lines in $\PP^3$:
$$\bG(1,3)\subset \PP(\wedge^2k^4)$$ 
together with its Pl\"ucker embedding. By \cite{St}~Exmple 3.4.6, five lines in $\PP^3$ with  Pl\"ucker 
coordinates $L_1,\ldots, L_5$ have a common transversal if and only if the determinant 
$$
\left|\begin{matrix}
0&L_1.L_2&L_1.L_3&L_1.L_4&L_1.L_5\cr
L_2.L_1&0&L_2.L_3&L_2.L_4&L_2.L_5\cr
L_3.L_1&L_3.L_2&0&L_3.L_4&L_3.L_5\cr
L_4.L_1&L_4.L_2&L_4.L_3&0&L_4.L_5\cr
L_5.L_1&L_5.L_2&L_5.L_3&L_5.L_4&0\cr
\end{matrix}\right|
$$
vanishes. Here $L_i.L_j$ denotes the wedge product in $\wedge^2 k^4)$. Let 
$L_i$ be the chord determined by the pair $(x_i,y_i)$. The lines $L_i$ connect points on 
the twisted cubic, which we can parametrize as $(1, t, t^2, t^3)$. It follows that the wedge product 
$L_i.L_j$ is the Vandermonde determinant:
$$
\left|\begin{matrix}
1&1&1&1\cr
x_i&y_i&x_j&y_j\cr
x_i^2&y_i^2&x_j^2&y_j^2\cr
x_i^3&y_i^3&x_j^3&y_j^3\cr
\end{matrix}\right|.
$$
As the terms $(y_i-x_i)$, $(y_j-x_j)$ can be factored out, we are left with a degree $20$
polynomial $F$ that defines $E'$ in $U$. Using the program Macaulay
it is easy to see that $F$ is irreducible and we can compute the multiplicities of $F$ along the 
diagonals $\Delta_I\subset\bA^{10}$. Using Lemma \ref{FM space}, it follows that the class of the pull-back 
$\pi_N^*D$ of the divisor $D=\phi(E')\subset\oM_{0,10}$ to $\oM_{0,11}$ (in the Kapranov model with respect to 
the $11$-th marking) is given by: 
$$
\pi_N^*D=20H-16\sum E_1- 12\sum E_{12} - 12\sum E_{11'}-9\sum E_{123}-$$
$$-8\sum E_{121'}-7\sum E_{1234}-5\sum E_{1231'}-6\sum E_{121'2'}-$$
$$-6\sum E_{12345}-3\sum E_{12341'}-3\sum E_{1231'2'}-$$
$$-3\sum E_{123451'}-\sum E_{12341'2'}-2\sum E_{1231'2'3'}-\sum E_{123451'2'}.$$

The divisor  $\pi_N^*D$ is linearly equivalent to a sum of boundary divisors: consider the sum of the $20$ hyperplanes
determined by choosing a pair of points $\{x_i, x_j\}$, or $\{y_i, y_j\}$ 
and taking the hyperplane passing through the remaining points (these are the boundary divisors 
$\de_{x_i,x_j}$ and $\de_{y_i,y_j}$). It is easy to see that all the multiplicities of this union 
of hyperplanes are larger than the multiplicities in the formula for $\pi_N^*D$.

It follows that $\pi_N^*D$  is a moving divisor. Since any effective divisor  linearly equivalent to $\pi_N^*D$  is a pull-back
by $\pi_N$ from $\oM_{0,10}$, it follows that $D$ is a moving divisor.
\end{Example}

\end{document}